\documentclass[10pt]{amsart}

\usepackage[utf8]{inputenc}
\usepackage{color}
\usepackage{amsmath} 
\usepackage{amssymb} 
\usepackage{amsthm}
\usepackage{graphicx}
\usepackage{esint}
\usepackage[colorlinks=true,linkcolor=blue]{hyperref}

\theoremstyle{plain}
\newtheorem{thm}{Theorem}
\newtheorem{prop}{Proposition}[section]
\newtheorem{lem}[prop]{Lemma}
\newtheorem{cor}[prop]{Corollary}

\newtheorem{notation}[prop]{Notation}
\newtheorem*{claim*}{Claim}
\newtheorem{assum}[prop]{Assumption}
\newtheorem{defi}[prop]{Definition}
\newtheorem{rmk}[prop]{Remark}

\newcommand {\R} {\mathbb{R}} 
 \newcommand {\N} {\mathbb{N}}
 
\newcommand {\p} {\partial}

\newcommand {\D} {\Delta}

\newcommand {\diam} {\text{diam}}

\DeclareMathOperator {\dist} {dist}
\DeclareMathOperator {\Ree} {Re}
\DeclareMathOperator {\Imm} {Im}

\DeclareMathOperator {\supp} {supp}
\DeclareMathOperator {\spa} {span}
\DeclareMathOperator {\inte} {int}

\usepackage{xspace}
\newcommand{\secaux}{2.1 in \cite{KRS14}\xspace}
\newcommand{\secCarl}{3 in \cite{KRS14}\xspace}
\newcommand{\secreg}{4 in \cite{KRS14}\xspace}
\newcommand{\defvanishing}{4.1 in \cite{KRS14}\xspace}
\newcommand{\lemnewCarl}{3.2 in \cite{KRS14}\xspace}
\newcommand{\lemalmosthom}{4.2 in \cite{KRS14}\xspace}
\newcommand{\lemgrowthimp}{4.1 in \cite{KRS14}\xspace}
\newcommand{\lemfreebgrowth}{4.3 in \cite{KRS14}\xspace}
\newcommand{\propchange}{2.1 in \cite{KRS14}\xspace}
\newcommand{\propindep}{4.1 in \cite{KRS14}\xspace}
\newcommand{\propsemicont}{4.2 in \cite{KRS14}\xspace}
\newcommand{\propdoubling}{4.3 in \cite{KRS14}\xspace}
\newcommand{\propblowup}{4.4 in \cite{KRS14}\xspace}
\newcommand{\lemhomo}{4.5 in \cite{KRS14}\xspace}
\newcommand{\prophomotwo}{4.6 in \cite{KRS14}\xspace}
\newcommand{\propalmost}{4.8 in \cite{KRS14}\xspace}
\newcommand{\corconsequenceCarl}{3.1 in \cite{KRS14}\xspace}
\newcommand{\corlowunif}{4.1 in \cite{KRS14}\xspace}
\newcommand{\eqvCarl}{(7) in \cite{KRS14}\xspace}
\newcommand{\rmkten}{10 in \cite{KRS14}\xspace}
\newcommand{\rmktwoD}{17 in \cite{KRS14}\xspace}

\title[The Variable Coefficient Thin Obstacle Problem]{The Variable Coefficient Thin Obstacle Problem: Optimal Regularity and Regularity of the Regular Free Boundary} 
\author{Herbert Koch}
\author{Angkana R\"uland }
\author{Wenhui Shi}

\address{
Mathematisches Institut, Universit\"at Bonn, Endenicher Allee 60, 53115 Bonn, Germany }
\email{koch@math.uni-bonn.de}

\address{
Mathematical Institute of the University of Oxford, Andrew Wiles Building, Radcliffe Observatory Quarter, Woodstock Road, OX2 6GG Oxford, United Kingdom }
\email{ruland@maths.ox.ac.uk}

\address{
Mathematisches Institut, Universit\"at Bonn, Endenicher Allee 64, 53115 Bonn, Germany  }
\email{wenhui.shi@hcm.uni-bonn.de}

\begin{document}

\begin{abstract}
This article deals with the variable coefficient thin obstacle problem in $n+1$ dimensions. We address the regular free boundary regularity, the behavior of the solution close to the free boundary and the optimal regularity of the solution in a low regularity set-up. \\
We first discuss the case of zero obstacle and $W^{1,p}$ metrics with $p\in(n+1,\infty]$. In this framework, we prove the $C^{1,\alpha}$ regularity of the regular free boundary and derive the leading order asymptotic expansion of solutions at regular free boundary points. We further show the optimal $C^{1,\min\{1-\frac{n+1}{p}, \frac{1}{2}\}}$ regularity of solutions. 
New ingredients include the use of the Reifenberg flatness of the regular free boundary, the construction of an (almost) optimal barrier function and the introduction of an appropriate splitting of the solution. Important insights depend on the consideration of various intrinsic geometric structures.\\ 
Based on variations of the arguments in \cite{KRS14} and the present article, we then also discuss the case of non-zero and interior thin obstacles. We obtain the optimal regularity of the solutions and the regularity of the regular free boundary for $W^{1,p}$ metrics and $W^{2,p}$ obstacles with $p\in (2(n+1),\infty]$. 
\end{abstract}



\thanks{
H.K. acknowledges support by the DFG through SFB 1060.
A.R. acknowledges that the research leading to these results has received funding from the European Research Council under the European Union's Seventh Framework Programme (FP7/2007-2013) / ERC grant agreement no 291053 and a Junior Research Fellowship at Christ Church.
W.S. is supported by the Hausdorff Center of Mathematics.}
\maketitle

\tableofcontents

\section{Introduction}
In this article we continue our discussion of the variable coefficient \emph{thin obstacle} or \emph{Signorini problem} in a low regularity framework. Here our main objectives are an improved understanding of the (regular) free boundary, the determination of the asymptotic behavior of solutions close to the (regular) free boundary and the derivation of optimal regularity estimates for solutions. To achieve this in our low regularity set-up, we in particular introduce two key new arguments: The identification of the regular free boundary as a \emph{Reifenberg flat} set, which enables us to construct (almost) optimally scaling barrier functions (c.f. Section \ref{sec:Reifenberg}) and a ``\emph{splitting technique}'' (c.f. Proposition \ref{prop:v1}), which allows us to deal with divergence form right hand sides and to identify the leading order contributions in the respective equations.\\

Let us explain the precise set-up of our problem: We consider local minimizers of the constrained Dirichlet energy:
\begin{align}
\label{eq:energy}
J(w) = \int\limits_{B_1^+} a^{ij} (\p_i w) (\p_j w) dx,
\end{align}
where we use the Einstein summation convention and assume that
\begin{align*}
 w\in \mathcal{K}:=\{v\in H^1(B_1^+)| \ v \geq 0 \mbox{ on } B_1':= B_1^+\cap \{x_{n+1}=0\} \}.
\end{align*}
Here, the metric $a^{ij}: B_1^+ \rightarrow \R^{(n+1)\times (n+1)}_{sym}$ is a symmetric, uniformly elliptic tensor field which is $W^{1,p}$, $p\in (n+1,\infty]$, regular and $B_1^+:=\{x\in B_1\subset \R^{n+1}| \ x_{n+1}\geq 0\}$ denotes the upper half-ball.
On the upper half-ball it is possible to consider arbitrary local variations of local minimizers. Hence, local minimizers solve an elliptic divergence form equation in the upper half-ball, on this set they are ``free''. However, on the codimension one surface $B_1'$ they obey the convex constraint $w\geq 0$ which leads to \emph{complementary} or \emph{Signorini} boundary conditions. In this sense the obstacle is ``thin''.\\

In the sequel we are in particular interested in obtaining an improved understanding of the free boundary 
$$\Gamma_w := \partial_{B_1'} \{x\in B_1'| \ w(x)>0 \}.$$
This set (which for Lipschitz metrics $a^{ij}$ is of Hausdorff dimension $n-1$, c.f. Remark \ref{rmk:Hausdorff}) separates the \emph{contact set}, $\Lambda_w:=\{x\in B_1'| \ w(x)=0\},$
in which the solution coincides with the obstacle, from the \emph{positivity set}, 
$ \Omega_w:=\{x\in B_1'| \ w(x)>0\},$
in which the solution is ``free''.  Moreover, we seek to understand the structure of the solution close to the free boundary.\\

Considering variations of local minimizers in the energy functional (\ref{eq:energy}) leads to an equivalent formulation of the local minimization problem (\ref{eq:energy}) in the form of a variational inequality posed in the energy space $\mathcal{K}$ \cite{U87}: For a solution, $w\in \mathcal{K}$, of (\ref{eq:energy}) we have
\begin{align*}
\int_{B_1^+}a^{ij}(\p_iw)\p_j(v-w) dx \geq 0 \mbox{ for all }  v\in \mathcal{K}.
\end{align*}
If in addition $w\in H^2(B_1^+)$, this corresponds to an elliptic equation with \emph{complementary} or \emph{Signorini} boundary conditions:
\begin{equation}
\begin{split}
\label{eq:varcoef'}
\p_i  a^{ij} \p_j  w & = 0 \mbox{ in } B^+_{1}, \\
 w\geq 0,  -a^{n+1,j}\p_{j}w\geq 0, \ w (a^{n+1,j}\p_{j} w)&= 0 \mbox{ on } B_{1}'.
\end{split}
\end{equation}
Here the Signorini condition is derived from the pointwise inequality $-a^{n+1,j}\p_j w(v-w)\geq 0$ on $B'_1$ which holds for any $v\in \mathcal{K}$.\\

In the sequel we investigate the thin obstacle problem by studying solutions of (\ref{eq:varcoef'}). Moreover, we address variants of it which involve inhomogeneities, \emph{non-flat} obstacles and boundaries and \emph{non-flat interior} obstacles.

\subsection{Main results}

In this article we first derive the $C^{1,\alpha}$ regularity of the so-called regular set of the free boundary in the presence of $W^{1,p}$ metrics, $a^{ij}$, with $p>n+1$. Moreover, in this framework we deduce a leading order aymptotic expansion of solutions to (\ref{eq:varcoef'}) with error estimates. Combining the regularity of the regular free boundary with the Carleman estimate from \cite{KRS14}, we then show the optimal $C^{1,\min\{1-\frac{n+1}{p},\frac{1}{2}\}}$ regularity of solutions with $W^{1,p}$ metrics, $a^{ij}$, for $p>n+1$.
In addition to this, we also treat perturbations of the thin obstacle problem including non-flat free boundaries and obstacles, as well as inhomogeneities in the equations and the interior thin obstacle problem.\\
In order to deduce these results, we rely on two main new ingredients: A ``splitting argument'' and the construction of (almost) optimally scaling barrier functions. The latter builds on the identification of the regular free boundary as a Reifenberg flat set. \\
In the following subsections we elaborate on these results, put them into the context of the literature on the thin obstacle problem, and explain the main difficulties and the new arguments which are used to overcome these.\\

We first recall the main results from \cite{KRS14} on free boundary points: All free boundary points $x\in \Gamma_w$ are classified by their associated vanishing order $\kappa_x$ (c.f. Section \secreg):
$$\Gamma_w = \Gamma_{3/2}(w)\cup \bigcup\limits_{\kappa \geq 2}\Gamma_{\kappa}(w).$$  
Here $\Gamma_{3/2}(w)$ is the  so-called \emph{regular} free boundary which is given by all free boundary points with vanishing order $3/2$. The remaining set $\bigcup\limits_{\kappa \geq 2}\Gamma_{\kappa}(w)$ consists of all free boundary points with a higher order of vanishing (c.f. Definition \defvanishing or Section~\ref{sec:not} for the precise definition of $\kappa_x$).
Moreover, the map $\Gamma_w \ni x \mapsto \kappa_x$ is upper semi-continuous. In \cite{KRS14} we also proved that for each $x\in \Gamma_w$ with $\kappa_x<\infty$, there exists an $L^2$-normalized blow-up sequence $w_{x,r_j}$ such that the limit $w_{x,0}$ is a homogeneous global solution with homogeneity $\kappa_x$. Furthermore, if $\kappa_x=\frac{3}{2}$, the blow-up limit $w_{x,0}$ is two-dimensional and (up to a rotation of coordinates) is equal to $c_n\Ree(x_n+ix_{n+1})^{3/2}$. \\
 
In the first part of the present paper, we study the regular free boundary $\Gamma_{3/2}(w)$, which is relatively open by the upper semi-continuity of $\kappa_x$. 
Relying on comparison principles (c.f. Proposition \ref{prop:nondeg}) combined with a splitting technique for equations with divergence right hand sides (c.f. Proposition \ref{prop:v1}), we obtain the $C^{1,\alpha}$ regularity of the regular free boundary $\Gamma_{3/2}(w)$:

\begin{thm}
\label{thm:C1a}
Let $a^{ij}: B_1^+ \rightarrow \R^{(n+1)\times (n+1)}_{sym}$ be a uniformly elliptic, symmetric $W^{1,p}$, $p\in(n+1,\infty]$, tensor field. Assume that $w$ is a solution of the variable coefficient thin obstacle problem (\ref{eq:varcoef'}). For each $x_0\in \Gamma_{3/2}(w)$, there exist a parameter $\alpha \in (0,1]$, a radius $\rho=\rho(x_0,w)$ and a $C^{1,\alpha}$ function $g: B_{\rho}''(x_0) \rightarrow \R$ such that (possibly after a rotation)
\begin{align*}
\Gamma_w \cap B_{\rho}' (x_0) = \Gamma_{3/2}(w)\cap B_{\rho}' (x_0)= \{x| \ x_n = g(x'')\}\cap B_{\rho}'(x_0).
\end{align*}
\end{thm}

We remark that we turn the usual order of the arguments around: Instead of \emph{first} proving optimal regularity of the solution, and \emph{then} regularity of the (regular) free boundary (c.f. \cite{AC06}, \cite{ACS08}, \cite{PSU}), we \emph{first} prove regularity of the (regular) free boundary, and \emph{then} deduce optimal regularity of the solution. The regular free boundary is $C^{1,\alpha}$ under conditions which do \emph{not} imply $C^{1,\frac{1}{2}}$ regularity of the solution.
\\

With the $C^{1,\alpha}$ regularity of the (regular) free boundary at hand, it then becomes possible to study the local behavior of solutions around regular free boundary points. This is based on identifying the leading order in the asymptotics of solutions of (\ref{eq:varcoef'}) at regular free boundary points in the presence of $W^{1,p}$ metrics with $p\in(n+1,\infty]$ (c.f. Proposition \ref{prop:wasympt}). As the asymptotics are complemented by a higher order error estimate, this allows us to obtain local growth bounds. Then combining this with our Carleman estimate, we are able to obtain the $C^{1,1/2}$ optimal regularity of solutions associated with $W^{1,p}$ metrics for $p\in(2(n+1),\infty]$:

\begin{thm}[Optimal regularity]
\label{prop:full_opti}
Let $a^{ij}: B_1^+ \rightarrow \R^{(n+1)\times (n+1)}_{sym}$ be a uniformly elliptic, symmetric $W^{1,p}$, $p\in[2(n+1),\infty]$, tensor field.
Assume that $w$ is a solution of the variable coefficient thin obstacle problem (\ref{eq:varcoef'}). Then, there exists a constant $C>0$ depending only on $\|a^{ij}\|_{W^{1,p}(B_{1}^+)}, n, p$ such that
\begin{align*}
\left\| w \right\|_{C^{1,1/2}(B_{1/2}^+)} \leq C \left\| w\right\|_{L^2(B_1^+)}.
\end{align*}
\end{thm}

Finally, in the last part of the paper we prove that these results are not restricted to the flat thin obstacle problem, i.e. the setting in which the obstacle is the zero function. We show that, on the contrary, it is possible to deal with \emph{inhomogeneities}, \emph{non-constant} obstacles, \emph{non-flat} boundaries and even \emph{non-flat interior} obstacles (c.f. Section \ref{sec:pert}). For instance, we prove the following result for non-flat obstacles:

\begin{thm}
Let $a^{ij}: B_1^+ \rightarrow \R^{(n+1)\times (n+1)}_{sym}$ be a uniformly elliptic, symmetric $W^{1,p}$, $p\in(2(n+1),\infty]$, tensor field. Suppose that $\varphi \in W^{2,p}(B'_{1})$. Let $w:B_1^+ \rightarrow \R$ be a solution of the thin obstacle problem
\begin{equation}
\label{eq:varcoef_f}
\begin{split}
\p_i a^{ij} \p_j w & = 0 \mbox{ in } B_{1}^+,\\
w  \geq \varphi, \ - a^{n+1,j} \p_j w \geq 0, \ (w-\varphi) ( a^{n+1,j} \p_j w) & = 0 \mbox{ on } B_{1}'.
\end{split}
\end{equation}
Then, the following statements hold:
\begin{itemize}
\item[(i)] The function $w$ has the optimal Hölder regularity: 
$$ w \in C^{1,1/2}(B_{1/2}^+).$$
\item[(ii)] Assuming that $0\in \Gamma_{3/2}(w)$, there exist a radius $\rho>0$, a parameter $\alpha \in (0,1]$ and a $C^{1,\alpha}$ function $g$ such that (potentially after a rotation) 
$$\Gamma_w\cap B_{\rho}'=\Gamma_{3/2}(w) \cap B_{\rho}' = \{x| \ x_n = g(x'')\}\cap B_{\rho}'.$$
\end{itemize}
\end{thm}

\subsection{Literature and context}
In the last years the thin obstacle problem and its variants have been a very active field of research. After the complete characterization of the two-dimensional, constant coefficient thin obstacle problem by Lewy \cite{Le72} and Richardson \cite{Ri78} as well as impressive partial results on the general problem \cite{Fr77}, \cite{Ki81}, \cite{U85}, a major new idea emerged only relatively recently: In \cite{AC06} and \cite{ACS08} Athanasopoulos, Caffarelli and Salsa introduced Almgren's frequency function as a powerful tool of obtaining optimal regularity estimates and the $C^{1,\alpha}$ regularity of the free boundary for the constant coefficient problem. Here $\alpha$ is some constant in $(0,1]$. Later this was extended by \cite{Si} and \cite{CSS} to the related obstacle problem for the fractional Laplacian. Moreover, in a recent article \cite{KPS} Koch, Petrosyan and Shi prove the analyticity of the regular free boundary for the constant coefficient operator by introducing a connection between the thin obstacle problem and the Grushin Laplacian. Simultaneously, Savin and De Silva \cite{DSS14} obtained the $C^{\infty}$ regularity of the regular free boundary by exploiting higher order boundary Harnack inequalities. This discussion of the regularity of the regular free boundary is complemented by an article of Garofalo and Petrosyan \cite{GP09} in which a monotonicity formula is used to characterize the structure of the singular set of the thin obstacle problem.\\

While these results illustrate that there has been great progress in the \emph{constant} coefficient thin obstacle problem, less is known in the \emph{variable} coefficient framework. Here the work of Uraltseva \cite{U85} has shown that it is possible to obtain $C^{1,\alpha}$ Hölder regularity of solutions of the thin obstacle problem in the presence of $W^{1,p}$, $p\in (n+1,\infty]$, metrics. In her result the Hölder exponent is some value $\alpha \in (0,1/2]$ which depends on the ellipticity constants of the coefficients and the value of $p$.\\
Only very recently, the variable coefficient problem has been further investigated: In \cite{Gu} Guillen deals with the problem in the context of Hölder regular $C^{1,\alpha}$, for some $\alpha\in (0,1)$, metrics. This is improved in an article by Garofalo and Smit Vega Garcia \cite{GSVG14} in which the authors derive the optimal regularity of solutions of the thin obstacle problem in the presence of $C^{0,1}$ metrics and $C^{1,1}$ obstacles. This argument is based on an extension of Almgren's monotonicity formula to the setting of low regularity metrics and obstacles. In a recent preprint \cite{GPSVG15} Garofalo, Petrosyan and Smit Vega Garcia further build on this and deduce the $C^{1,\alpha}$, for some $\alpha\in (0,1)$, regularity of the regular free boundary in the presence of $C^{0,1}$ metrics and $C^{1,1}$ obstacles by proving an epiperimetric inequality.\\
Complementing these results by relying on Carleman estimates instead of frequency formulae and working in the setting of Sobolev metrics $a^{ij}\in W^{1,p}$, $p\in (n+1,\infty]$, \cite{KRS14} provides an alternative proof of the almost optimal regularity of solutions to (\ref{eq:varcoef'}). In the present article we extend these ideas, and, building on our previous work, prove optimal regularity in the framework of $W^{1,p}$, $p\in (n+1,\infty]$, metrics as well as the $C^{1,\alpha}$ regularity of the regular free boundary.

\subsection{Difficulties and strategy}
As already in \cite{KRS14} the central difficulty with which we deal throughout the article is the low regularity of the metric.\\

Building on the results from \cite{KRS14}, we analyze the (regular) free boundary (c.f. Section \ref{sec:boundary}). Traditionally, it is studied by relying on comparison principles (c.f. \cite{PSU}): Differentiating the equation (\ref{eq:varcoef'}), an analysis of the equation for the tangential derivatives allows to transfer positivity properties of derivatives of the blow-up solutions to the problem at hand. This then permits to conclude the $C^{1,\alpha}$ regularity of the regular free boundary.\\
In a low regularity set-up this becomes more difficult. In particular we have to include divergence form right hand sides in our discussion. Hence, as a key new ingredient we exploit a ``splitting technique'' in which we divide our solution into a ``controlled error'' which handles the low regularity contributions originating from the metric, and a ``main part'' which captures the behavior of our solutions (c.f. Proposition \ref{prop:v1}). This allows us to argue along similar lines as in the literature. \\
Yet, we have to overcome a second difficulty: In order to provide a framework which also deals with non-flat obstacles of low regularity, we have to work as close as possible to the scaling critical setting. Thus, instead of proving \emph{linear} non-degeneracy of the tangential derivatives, we prove a (nearly) \emph{square root} non-degeneracy in appropriate cones (c.f. Proposition \ref{prop:nondeg}). In order to achieve this, we introduce a second main new ingredient and construct a new barrier function which exploits the Reifenberg flatness of the free boundary (c.f. Proposition \ref{prop:barrier}).
\\

In the second part of our argument, we return to the investigation of solutions of the thin obstacle problem (c.f. Section \ref{sec:optimal}). Here we rely on the free boundary regularity which allows to improve the almost optimal growth estimates which were previously obtained from the Carleman estimate (c.f. Lemma \lemgrowthimp and Corollary \ref{cor:optgrowth}).\\
We argue in two steps in which we combine \emph{local} with \emph{global} information: First we prove a \emph{local} growth estimate around regular free boundary points in which we obtain the optimal growth estimates. As these bounds however rely on comparison arguments which depend on the free boundary itself, they are \emph{not} uniform in the free boundary points.
By virtue of the $C^{1,\alpha}$ regularity of the (regular) free boundary, these growth estimates can be obtained by an asymptotic expansion of solutions around regular free boundary points. For the identification of the leading order contribution of the expansion and of the corresponding higher order error estimates, we exploit our boundary Harnack estimate in combination with the boundary regularity for $W^{1,p}$ metrics with $p\in(n+1,\infty]$. Here we exploit our main new results on the free boundary regularity.\\
In the second step in our optimal regularity argument for $W^{1,p}$ metrics with $p>2(n+1)$, we then combine the optimal, but non-uniform local with optimal \emph{global} information. For this we rely on Lemma \lemalmosthom which allows us to transfer the local into global information. This is the only point at which we directly return to the arguments from \cite{KRS14}.\\

Finally, in Section \ref{sec:pert} we comment on the stability of the described methods by applying them to variants of the thin obstacle problem. Using the scaling of the Carleman inequality, we first show that it is possible to deal with inhomogeneities (c.f. Section \ref{sec:inhomo}). This then immediately entails that all the previous results remain true for sufficiently regular, non-constant obstacles, although we only require that they are $W^{2,p}$ regular for some $p>2(n+1)$ (c.f. Section \ref{sec:nonfl}).\\
Concluding the section on variants of the thin obstacle problem, we discuss the setting of the \emph{interior} thin obstacle problem (c.f. Section \ref{sec:int_obst}). Here we are confronted with additional difficulties, which arise due to a slightly modified boundary condition (instead of a sign condition on the Neumann derivative, there is a sign condition on the fluxes across the interior boundary). This leads to slight modifications in the derivation of the Carleman inequality from \cite{KRS14} and yields leading order \emph{linear} contributions in the asymptotics of the Neumann derivative (c.f. Proposition \ref{prop:intobst}).

\subsection{Organization of the paper}
After briefly recalling auxiliary results, conventions and the notation from \cite{KRS14} in Section \ref{sec:not}, we begin with the analysis of the regular free boundary in Section \ref{sec:boundary}. Here Proposition \ref{prop:v1} is a crucial technical tool to overcome the difficulties with the low regularity set up. In Section \ref{subsec:Lip} we first prove the Lipschitz regularity of the regular free boundary (c.f. Proposition \ref{prop:Lip}). This is based on the observation that the free boundary is Reifenberg flat, which we exploit in Section \ref{sec:Reifenbergbarrier}, in order to construct an appropriate, sufficiently well scaling barrier function (c.f. Proposition \ref{prop:barrier}). In Section \ref{subsec:C1a} we improve the Lipschitz regularity to gain $C^{1,\alpha}$ regularity (c.f. Proposition \ref{prop:C1a}).\\
In the second part of the paper we return to the study of solutions of (\ref{eq:varcoef}): In Section \ref{sec:optimal} we identify the leading order asymptotics of solutions of the thin obstacle problem (Proposition \ref{prop:asympt}). This allows to derive growth bounds (Corollary \ref{cor:lowerbound} and Corollary \ref{cor:optgrowth}) as well as the optimal regularity of solutions of (\ref{eq:varcoef'}) (c.f. Theorem \ref{thm:optimal_reg}). Finally, in the last part of the article, in Section \ref{sec:pert}, we illustrate how the previous results can be transferred to variations of the thin obstacle problem.

\section{Preliminaries}
\label{sec:prelim}

In this section, we briefly explain our normalizations and notational conventions.

\subsection{Auxiliary results}
We start by recalling that, due to the discussion in Section \secaux, we may, without loss of generality, consider solutions $w$ of (\ref{eq:varcoef'}) with
\begin{itemize}
\item[(A0)] $\left\| w \right\|_{L^2(B_1^+(0))}=1$,
\item[(A1)] $a^{i, n+1}(x',0)=0 \mbox{ on } \R^{n} \times \{0\}$ for $i=1,\ldots, n$,
\item[(A2)] $a^{ij}$ is symmetric and uniformly elliptic with eigenvalues in the interval $[1/2, 2]$.
\end{itemize}

In addition, we in the sequel also make the following hypotheses (where we either assume (A3) or (A3')):
\begin{itemize}
\item[(A3)]  $a^{ij}\in W^{1,p}(B_1^+(0))$ for some $p\in (n+1,\infty]$, 
\item[(A3')] $a^{ij}\in W^{1,p}(B_1^+(0))$ for some $p\in (2(n+1),\infty]$, 
\item[(A4)] $a^{ij}(0)= \delta^{ij}$.
\end{itemize}

These assumptions allow us to reduce (\ref{eq:varcoef'}) to 
\begin{equation}
\begin{split}
\label{eq:varcoef}
\p_i  a^{ij} \p_j  w & = 0 \mbox{ in } B^+_{1}, \\
 w\geq 0,  -\p_{n+1}w\geq 0, \ w (\p_{n+1} w)&= 0 \mbox{ on } B_{1}'.
\end{split}
\end{equation}
Due to the $H^2$ estimates and the almost optimal regularity result of \cite{KRS14}, it is possible to interpret (\ref{eq:varcoef}) not only as a variational inequality in the energy space $H^1(B_1^+)$, but also to understand the equation and its boundary values in a classical pointwise sense.\\

Let us comment on the conditions (A0)-(A4): We recall that it is always possible to achieve (A0) by a suitable normalization. Condition (A1) is a consequence of an appropriate change of coordinates, c.f. Uraltseva \cite{U85}. Also, assumption (A2) and (A4) can always be achieved by an additional affine change of coordinates (and a rescaling) and thus do not pose additional restrictions on our set-up. \\
Conditions (A3) and (A3') are regularity assumptions which allow us to apply the results of \cite{KRS14}. Here condition (A3') is slightly more restrictive. While this assumption is \emph{not} needed in the case of flat obstacles, it becomes necessary in our treatment of \emph{non-flat} obstacles, as a consequence of our strategy of proof: We work on the level of the differentiated equation. Moreover, we stress that the integrability assumption $p\geq 2(n+1)$ yields the embedding $W^{1,p}\hookrightarrow C^{0,1/2}$ which, by interior regularity, is needed (and might also suffice) to derive the (optimal) $C^{1,1/2}$ regularity of solutions to the variable coefficient thin obstacle problem.
Both conditions (A3) and (A3') imply that
\begin{align*}
|a^{ij}(x)-\delta^{ij}| \leq C_n \left\| \nabla a^{ij} \right\|_{L^p(B_1')}|x|^{1-\frac{n+1}{p}} \mbox{ for all } x\in B_{1}^+
\end{align*}
by Morrey's inequality.

\subsection{Notation}
\label{sec:not}
We use the same notation as in \cite{KRS14}, which we briefly recall in the sequel:
\begin{itemize}
\item $\R^{n+1}_+ := \{x\in \R^{n+1}| \ x_{n+1}\geq 0\}$,  $\R^{n+1}_- := \{x\in \R^{n+1}| \ x_{n+1}\leq 0\}$.
\item For points $x\in \R^{n+1}$ we also use the notation $x=(x',x_{n+1})$ or $x=(x'',x_{n},x_{n+1})$, if we want to emphasize the roles of the respective lower dimensional coordinates.
\item Let $x_0=(x_0',0) \in \R^{n+1}_+$. For the upper half-ball of radius $r>0$ around $x_0$ we write $B_r^+(x_0):=\{x\in \R^{n+1}_+| \ |x-x_0| < r \}$; the projection onto the boundary of $\R^{n+1}_+$ is respectively denoted by $B_r'(x_0):=\{x\in \R^{n} \times \{0\}| \ |x-x_0| < r \}$. If $x_0 = (0,0)$ we also write $B_r^+$ and $B_r'$. Analogous conventions are used for balls in the lower half sphere: $B^-_r(x_0)$. Moreover, we use the notation $B_r''(x_0) = \{x\in \R^{n-1}\times \{(x_0)_n\}\times \{0\}| \ |x''-x''_0|<r\}$, where $x_0=(x_0'', (x_0)_n,0), x= (x'', x_n,0)$.
\item Annuli around a point $x_0=(x_0',0)$ in the upper half-space with radii $0<r<R<\infty$ as well as their projections onto the boundary of $\R^{n+1}_+$ are denoted by by $A_{r,R}^+(x_0):= B_{R}^+(x_0)\setminus B_{r}^+(x_0)$ and $A_{r,R}'(x_0):= B_{R}'(x_0)\setminus B_{r}'(x_0)$ respectively. For annuli around $x_0=(0,0)$ we also omit the center point. Furthermore, we set $A_{r,R}(x_0):= A_{r,R}^+(x_0)\cup A_{r,R}^-(x_0)$.
\item We use $\mathcal{C}_\eta(e_n)$ to denote an $(n+1)$-dimensional cone with with opening angle $\eta$ and axis $e_n$. Analogously, $\mathcal{C}'_\eta(e_n)$ refers to a flat (i.e. $n$-dimensional) cone on $\{x_{n+1}=0\}$ with opening angle $\eta$ and axis $e_n$.
\item For $f: \R^{n+1} \rightarrow \R$ we set $\nabla' f$ and $\nabla'' f$ to denote the derivatives with respect to the $x'$ and $x''$ components of $x$.
\item We use $|\cdot|$ to denote the standard Euclidean norm.
\item Distances with respect to a Riemannian metric $g$ (e.g. if they are induced by certain operators as in Section \ref{sec:propv1}) are denoted by $d_g(\cdot, \cdot)$.
\item  $\dist_{H}(X,Y):= \max\{\sup\limits_{x\in X} d(x,Y), \sup\limits_{y\in Y} d(y,X)\}$ denotes the Hausdorff distance of two subsets $X$ and $Y$ in $\R^m$.
\item Let $w:B_{1}^+ \rightarrow \R$ be a solution of (\ref{eq:varcoef}). Then
\begin{itemize}
\item $\Omega_w:= \{x\in \R^{n}\times \{0\}| \ w(x)>0\}$ denotes the \emph{positivity set}.
\item $\Gamma_w:=\partial_{B_1'} \Omega_w$ is the \emph{free boundary}.
\item $\Lambda_w:= B_1' \setminus \Omega_w$ is the \emph{coincidence set}.
\item $\Gamma_{\frac{3}{2}}(w):= \{x \in \Gamma_w| \kappa_x = \frac{3}{2}\} \subset \Gamma_w$ is the \emph{regular set} or the \emph{regular free boundary}. Here $\kappa_x:= \limsup \limits_{r\rightarrow 0} \frac{\ln(r^{-\frac{n+1}{2}}\left\| w \right\|_{L^2(A_{r,2r}^+(x))})}{\ln(r)} $ is the \emph{vanishing order} of $w$ at $x$.
\end{itemize}
\item As in \cite{KRS14} we use the notation $w_{r}(x):= \frac{w(r x)}{ r^{- \frac{n+1}{2} } \left\| w \right\|_{L^2(B_{r}^+)} }$ to denote the $L^2$ normalized rescaling of $w$ at zero. Analogously we consider the $L^2$-normalized blow-ups at arbitrary free boundary points $x_0 \in \Gamma_w$:
$w_{x_0,r}(x):= \frac{w(x_0+r x )}{ r^{- \frac{n+1}{2} } \left\| w \right\|_{L^2(B_{r}^+(x_0))} }$. 
\item We set $w_{3/2}(x):= c_n \Ree(x_n + i x_{n+1})^{3/2}$, where $c_n$ is chosen such that $\left\| w_{3/2} \right\|_{L^2(B_1^+)}=1$. 
\item In the following we often use the abbreviation $\ell_0:= (2\sqrt{n})^{-1}$ as well as $c_{\ast}:=\left\| \nabla a^{ij} \right\|_{L^p(B_1)}$. 
\item We use $L_0 = a^{ij}\p_{ij}$ and $L = \p_i a^{ij} \p_j$ to denote the non-divergence and divergence form operators involved.
\item We reserve the parameter $\epsilon>0$ to quantify the Reifenberg flatness of a domain $\Gamma \subset \R^{n}\times \{0\}$, as well as the parameter $\epsilon_0$ to measure the closeness of a solution of (\ref{eq:varcoef}) to the $L^2$ normalized model solution $w_{3/2}(x)$.
\item We use the notation $A\lesssim B$ to denote that there exists an only dimension dependent constant such that $A \leq C B$. Similar conventions are used for $\gtrsim$.
\end{itemize}

\section{Regularity of the Regular Free Boundary}
\label{sec:boundary}

In this section we deduce the $C^{1,\alpha}$ regularity of the regular part of the free boundary for solutions of the thin obstacle problem (\ref{eq:varcoef}) and thus provide the proof of Theorem \ref{thm:C1a}.\\

We recall that it is possible to classify free boundary points by their vanishing order $\kappa$ (Section 4 of \cite{KRS14}):
\begin{align*}
\Gamma_w = \Gamma_{\frac{3}{2}}(w) \cup \bigcup\limits_{\kappa \geq 2} \Gamma_{\kappa}(w),
\end{align*}
where $\Gamma_{\kappa}(w):= \{x\in \Gamma_w \big| \ \kappa_{x}= \kappa \}$. The set $\Gamma_{\frac{3}{2}}(w)$ is referred to as the \emph{regular set} of $w$ or the \emph{regular free boundary} of $w$. A corresponding point $x\in \Gamma_{\frac{3}{2}}(w)$ is called \emph{regular}. \\

Although we are mainly interested in the \emph{regular} free boundary, in passing, we make the following observation on the Hausdorff-dimension of the \emph{(whole) free boundary}:

\begin{rmk}[Hausdorff dimension of $\Gamma_w$]
\label{rmk:Hausdorff}
We claim that, if the obstacle is zero and the metric $a^{ij}$ is Lipschitz continuous, the Hausdorff-dimension of the \emph{(whole) free boundary} is less than or equal to $n-1$. Indeed, this follows as in Lemma 4.1 and Theorem 8.10 in \cite{SW10}. Similarly as in these results, we prove a slightly stronger statement: We show that
\begin{align*}
\mathcal{H}^{s}(S_{w}\cap B_1') = 0 \mbox{ for all } s>n-1,
\end{align*}
where $S_w:=\{x\in B_1^+| \ w(x)=|\nabla w(x)|=0\}$ in particular contains $\Gamma_w$.\\ 
We only provide a sketch of the proof here and refer to \cite{SW10} and \cite{Si96}, Chapter 3.3 for the details. As in \cite{SW10} we argue in two steps and first show that if $\varphi \in C^{1,\alpha}$ for some $\alpha>0$ is a solution of the \emph{constant} coefficient thin obstacle problem, then the Hausdorff dimension of the associated free boundary $S_{\varphi}\cap B_{1}'$ is less than or equal to $n-1$. \\
The argument for this follows by contradiction: Assuming that the statement of the claim were wrong, there existed $s>n-1$ such that
\begin{align*}
\mathcal{H}^s(S_{\varphi}\cap B_1') > 0.
\end{align*}
By density arguments for the Hausdorff measure (c.f. \cite{E15}, Chapter 2.3, Theorem 2.7) there exists a point $z\in S_{\varphi}\cap B_1'$ such that
\begin{align*}
\theta_{\mu_s}^{*}(S_{\varphi},z):= \limsup\limits_{\rho \rightarrow 0} \rho^{-s}\mu_s(S_{\varphi}\cap B_{\rho}'(z))>0,
\end{align*}
where for $A\subset B_1'$ we have $\mu_s(A):= \inf \sum\limits_{j=1}^{\infty} \rho_j^s$ and $A \subset \bigcup\limits_{j=1}^{\infty} B_{\rho_j}'(y_j)$ and $y_j \in B_{1}'$. Hence there exists a sequence $\sigma_j>0$ with $\sigma_j \rightarrow 0$ such that
\begin{align*}
\lim\limits_{\sigma_j \rightarrow 0} \sigma_j^{-s} \mu_s(S_{\varphi}\cap B_{\sigma_j}'(z))>0.
\end{align*}
Rescaling the function by setting $\varphi_{j}(x):= \frac{\varphi(z+\sigma_j x)}{\sigma_j^{-\frac{n+1}{2}}\|\varphi\|_{L^2(B_{\sigma_j}^+(z))}}$ with $x\in B_{\sigma_j^{-1}(1-|z|)}$ results in
\begin{align*}
\liminf\limits_{j\rightarrow 0} \mu_s(S_{\varphi_j}\cap \overline{B_1'})>0.
\end{align*}
By compactness (c.f. \cite{SW10}, \cite{PSU}) we have that along a subsequence
\begin{align*}
\varphi_{j'} \rightarrow \psi \text{ in } C^1_{loc}(B_1^+),
\end{align*}
where $\psi \in C^{1,\alpha}\cap W^{2,2}$ is a homogeneous solution with degree larger than or equal to $1+\alpha$ and it is normalized such that $\|\psi\|_{L^2(B_1^+)}=1$. Moreover, as in \cite{SW10}
\begin{align*}
\mu_s(S_{\psi}\cap \overline{B_1'})>0.
\end{align*}
We now repeat the outlined argument with $\psi$ instead of $\varphi$. This yields a function $\psi_1$ with $\mu_s(S_{\psi_1}\cap \overline{B_1'})>0$, $0\in S_{\psi_1}$ and $\psi_1$ being invariant under composition with translations in the direction $\lambda z$ for all $\lambda \in \R$. Indeed, the last point follows from the observation that $\mathcal{N}_{\psi_1}(0)= \mathcal{N}_{\psi_1}(z)$, where 
$$\mathcal{N}_{\psi_1}(x):= \lim\limits_{\rho\rightarrow 0} \frac{\rho \int\limits_{B_{\rho}^+(x)}|D\psi_1(y)|^2dy}{\int\limits_{\partial B_{\rho}(x)}|\psi_1 (y)|^2 dy}$$ 
denotes the frequency function (c.f. \cite{SW10}, Section 2 and \cite{PSU}, Chapter 9 for the existence of this limit and its properties). We also note that this is well-defined for $x\in S_{\psi_1}$ and not only for $x\in\Gamma_{\psi_1}$. Similarly as in Remark 2.4 (2) in \cite{SW10} this implies that for the homogeneous homogeneous function $\psi_1$ it holds $\psi_1(x+\lambda z) = \psi_1(x)$ for all $x\in \R^{n+1}_+$. Arguing inductively and carrying out the argument a further $n-1$ times (which is possible by our contradiction argument), we end up with a function $\psi_{n}(x_1,\dots, x_{n+1})=f(x_{n+1})$ which only depends on the $x_{n+1}$-variable and is non-trivial. As it is however also harmonic and satisfies $\psi_{n}(0)=|\nabla \psi_{n}(0)|=0$, this yields a contradiction and hence results in our claim.\\
In a second step, the more general case of a solution $w$ to the \emph{variable} coefficient problem (\ref{eq:varcoef}) with Lipschitz coefficients $a^{ij}$ is treated. Here the claim is that also in this situation the Hausdorff-dimension of $S_w \cap B_{1/2}'$ is less than or equal to $n-1$, i.e. for all $s>n-1$ it holds $\mathcal{H}^{s}(S_w \cap B_{1/2}')=0$. Again this argument follows by contradiction assuming the claim on the Hausdorff dimension were wrong. Then however blow-up argument (in combination with compactness, which holds in the setting of Lipschitz coefficients, c.f. for instance \cite{SW10}) reduces the situation to that of a non-trivial solution to the constant coefficient problem with Hausdorff dimension larger than $n-1$. This however is in contradiction with our first claim.
\end{rmk}

Returning to our main problem, in the following we study the regularity properties of the regular free boundary $\Gamma_{\frac{3}{2}}(w)$. This is divided into several steps: In Section \ref{sec:Reifenberg} we begin by proving that the regular free boundary is a relatively open set of the free boundary and that, due to the good compactness properties of blow-up solutions, it is Reifenberg flat. Then in the following Section \ref{subsec:Lip} we deduce its Lipschitz regularity. Our argument for this strongly relies on the boundary's Reifenberg flatness, as this allows to construct (nearly) optimally scaling lower barrier functions. In Section \ref{subsec:C1a} we exploit the Lipschitz regularity to infer the $C^{1,\alpha}$ regularity of the regular free boundary. Here we argue via a Carleson and a boundary Harnack estimate. Last but not least, in Section \ref{sec:propv1} we provide the proof of one of our most central technical tools, the splitting technique which is stated in Proposition \ref{prop:v1}. 

\subsection{Reifenberg flatness}
\label{sec:Reifenberg}

We begin the investigation of the regular set by showing that it is a relatively open subset of $\Gamma_w$:

\begin{prop}[Relative openness of $\Gamma_{3/2}(w)$]
\label{prop:rel_op}
Let $a^{ij}: B_1^+ \rightarrow \R^{(n+1)\times (n+1)}_{sym}$ be a uniformly elliptic $W^{1,p}$ tensor field with $p\in (n+1,\infty]$. Let $w$ be a solution of (\ref{eq:varcoef}) in $B_{1}^+$. Then $\Gamma_{3/2}(w)$ is a relatively open subset of $\Gamma_w$.
\end{prop}

\begin{proof}
By Proposition \propsemicont the mapping $\Gamma_w \ni x\mapsto \kappa_x$ is upper semi-continuous. Therefore, the preimage of the set of free boundary points with $\kappa_x < 2$ is relatively open. Due to the classification of the possible homogeneous solutions and due to the regularity of general solutions, $\kappa_x <2 $ already implies that $\kappa_x = \frac{3}{2}$ (c.f. Proposition \prophomotwo). Hence, we infer that $\Gamma_{\frac{3}{2}}(w)$ is a relatively open set of the free boundary.
\end{proof}

In order to study the regular set in greater detail, we recall a strong compactness result for $L^2$ normalized blow-ups of solutions of (\ref{eq:varcoef}). In contrast to the setting of free boundary points of higher vanishing order, it is possible to show that at \emph{regular} free boundary points any blow-up is a $3/2$-\emph{homogeneous} global solution. \\

For convenience of notation, in the sequel we set 
\begin{align}
\label{eq:modelsol}
w_{3/2}(x):= c_n \Ree(x_n + i x_{n+1})^{3/2} \mbox{ for } x\in B_{1}^+,
\end{align}
where $c_n>0$ is an only dimension dependent constant, which is chosen such that $\left\| w_{3/2} \right\|_{L^2(B_1^+)}=1$. \\
 
Using this convention, we formulate the $L^2$ normalized blow-up result at regular free boundary points:

\begin{lem}
\label{lem:3/2homo}
Let $a^{ij}: B_1^+ \rightarrow \R^{(n+1)\times (n+1)}_{sym}$ be a uniformly elliptic $W^{1,p}$ tensor field with $p\in (n+1,\infty]$. Let $w$ be a solution of (\ref{eq:varcoef}) in $B_{1}^+$.
Consider $x_0\in \Gamma_{3/2}(w)\cap B_{1/2}'$ and $w_{r,x_0}(x):= \frac{w(x_0+rx)}{r^{-\frac{n+1}{2}} \left\| w \right\|_{L^2(B_r^+(x_0))} }$. Then for any sequence $\{r_j\}_{j\in\N}$ with $r_j \rightarrow 0$ there exists a subsequence $\{r_{j_k}\}_{k\in\N}$ and a matrix $Q\in SO(n+1)$ such that
\begin{align}
\label{eq:3/2homo}
w_{r_{j_k},x_0}(x) \rightarrow  w_{3/2}(Q x) \mbox{ in } C^{1}(B_{1/2}^+).
\end{align} 
\end{lem}

\begin{proof}
The proof of (\ref{eq:3/2homo}) relies on an argument by Andersson \cite{An} and reduces the problem to Benedicks's theorem \cite{Bene80}. For completeness we give a brief sketch of it. \\
To this end, we first note that by Proposition \propblowup there exists a subsequence $\{r_{j_k}\}_{k\in\N}$ such that
\begin{align*}
w_{r_{j_k},x_0} \rightarrow  w_0 \mbox{ in } C^{1}(B_{1/2}^+).
\end{align*} 
Here $w_0$ is a global solution of the constant coefficient thin obstacle problem with $\|w_0\|_{L^2(B_1^+)}=1$. By Remark \rmktwoD we further infer that $w_0$ grows of a rate $R^{\kappa}$ with $\kappa \in (1,2)$ at infinity. Analogously, any tangential derivative $\p_e w_0$, with $e\in S^{n}\cap B_{1}'$, has at most a growth rate of $R^{\kappa-1}$ with $\kappa \in (1,2)$ at infinity.\\
Next we show that there exists $Q\in SO(n+1)$ such that $w_0(x)=w_{3/2}(Qx)$. Indeed, due to the at most $R^{\kappa-1}$, $\kappa \in (1,2)$, growth at infinity of $\p_e w_0$, the Friedland-Hayman inequality \cite{FH76} (see also Lemma 5.1 in \cite{An}), yields that for any tangential vector $e\in S^{n}\cap B_{1}'$ the directional derivative $\p_e w_0$ has a fixed sign in $\R^{n+1}_+$. Thus, arguing as in Proposition 9.9 in \cite{PSU}, we have that $w_0$ only depends on two variables, which, up to a rotation $Q$, can be written as $w_0(x_n,x_{n+1})$. As a consequence of the sign condition on $\p_n w_0$ (without loss of generality we assume $\p_n w_0\geq 0$), the fact that $\p_{n+1}w(0)=0$ (due to the $C^{1}$ convergence), and the growth rate at infinity, the coincidence set $\Lambda_0$ of $w_0$ is the $n$-dimensional half-plane $\Lambda_0=\{x_n\leq 0,x_{n+1}=0\}$. 
Due to the Signorini condition, after an even reflection about $x_{n+1}$, $\p_n w_0$ solves 
\begin{align*}
\D \p_n w_0 &= 0 \mbox{ in } \R^{n+1}\setminus \Lambda_0, \ \p_n w_0 = 0 \mbox{ on } \Lambda_0, \ \p_n w_0\geq 0 \text{ in } \R^{n+1}\setminus \Lambda_0.
\end{align*}
Applying Benedicks's theorem (c.f. Proposition 1 in \cite{An} and also \cite{Bene80}) to $\p_{n}w_0$ (or a Liouville type theorem for harmonic functions after opening up the domain), we have that up to a multiplicative constant $\p_n w_0(x)=\Ree(x_n+i|x_{n+1}|)^{1/2}$. A similar argument applied to $\p_{n+1}w_0$ yields $\p_{n+1}w_0(x)=\Imm (x_n+ix_{n+1})^{1/2}$ in $B_1^+$. Hence, necessarily $w_0(x) =c_n\Ree(x_n+ix_{n+1})^{3/2}= w_{3/2}(x)$ in $B_1^+$(where we used that $\|w_0\|_{L^2(B_1^+)}=\|w_{3/2}\|_{L^2(B_1^+)}=1$). 
\end{proof}

Keeping this compactness result in the back of our minds, we proceed by showing that the regular free boundary is Reifenberg flat. For this we first recall the following definition (c.f. \cite{CKL05}, \cite{LMS12}), which allows us to infer our first regularity result, Proposition \ref{prop:Reifenberg}, for the free boundary.

\begin{defi}[Reifenberg flatness]
\label{defi:Reifenberg}
A locally compact set $\Gamma \subset \R^{m}$ is \emph{$(\delta, R)$ Reifenberg flat} if for every $x_0\in \Gamma$ and every $r\in (0,R]$ there is a hyperplane $L(r,x_0)$ containing $x_0$ such that
\begin{align}
\label{eq:Reifenberg}
\frac{1}{r}\dist_{H} (\Gamma \cap B_r(x_0), L(r,x_0)\cap B_{r}(x_0)) \leq \delta,
\end{align}
where $\dist_{H}(X,Y):= \max\{\sup\limits_{x\in X} d(x,Y), \sup\limits_{y\in Y} d(y,X)\}$ denotes the Hausdorff distance of two subsets $X$ and $Y$ in $\R^m$.
\end{defi}

Recalling the convention from (\ref{eq:modelsol}), we have the following regularity result:

\begin{prop}
\label{prop:Reifenberg}
Let $a^{ij}: B_1^+ \rightarrow \R^{(n+1)\times (n+1)}_{sym}$ be a uniformly elliptic $W^{1,p}(B_1^+)$ tensor field with $p\in (n+1,\infty]$ which satisfies (A1), (A2) and (A4). Let $w$ be a solution of (\ref{eq:varcoef}) in $B_{1}^+$. Then for each $\epsilon>0$ there exists $\delta>0$ such that if
\begin{itemize}
\item[(i)] $\left\| w - w_{3/2} \right\|_{C^1(B_1^+)} \leq \delta$,
\item[(ii)] $\left\| a^{ij} - \delta^{ij} \right\|_{L^{\infty}(B_1^+)} \leq \delta$,
\end{itemize}
then for each $x_0\in \Gamma_{w} \cap B_{1/2}'$ and each radius $r\in(0,1/2)$, there is a rotation $S=S(r,x_0)\in SO(n+1)$ such that
\begin{align}
\label{eq:Reifenb_close}
\left\| w_{r,x_0}(x) - w_{3/2}(S x) \right\|_{C^1(B_{1/2}^+)} \leq \epsilon.
\end{align} 
In particular, if $\delta$ is chosen sufficiently small, then 
$ \Gamma_w \cap B_{1/2}'= \Gamma_{3/2}(w) \cap B_{1/2}'.$
\end{prop}

\begin{proof}
We argue in two steps. First, we derive a non-degeneracy condition at free boundary points, which is crucial in order to obtain good compactness properties. This first step also immediately implies that $\Gamma_{3/2}(w) \cap B_{1/2}' = \Gamma_w \cap B_{1/2}'$. Then, in the second step, we argue that (\ref{eq:Reifenb_close}) holds.\\

\emph{Step 1: Compactness.}
We begin by proving the following claim:
\begin{claim*}
Assume that the conditions of Proposition \ref{prop:Reifenberg} are satisfied.
Then for each $\lambda>0$, if $\delta=\delta(n,p,\lambda)$ is chosen sufficiently small, for all $r\in (0,1/2)$ and for all $x_0 \in \Gamma_w \cap B_{1/2}'$, we have that
\begin{align}
\label{eq:nondeg_contr}
r^{-\frac{n+1}{2}}\left\| w \right\|_{L^2(B_{r}^+(x_0))} \geq \frac{1}{2}r^{\frac{3}{2} + \lambda}.
\end{align}
\end{claim*}
We prove the claim as a consequence of the closeness assumption (i) to the model solution and of Corollary \corconsequenceCarl for a suitable choice of $\tau$. \\

Consider $R:= (4\delta)^{\frac{2}{n+4}}$ with $0<\delta<\left(\frac{1}{2}\right)^{\frac{n}{2}+4}$. Then, by the triangle inequality and by condition (i), we have that for any $x_0\in \Gamma_w\cap B_{1/2}'$
\begin{align*}
\left\| w \right\|_{L^2(B_{r}^+(x_0))} \geq  \left\| w_{3/2} \right\|_{L^2(B_{r}^+(x_0))} - \delta  \geq \left\| w_{3/2} \right\|_{L^2(B_{r/2}^+(\bar{x}_0))} - \delta.
\end{align*}
for all $r\in [R,1)$. Here $\bar{x}_0\in \Gamma_{w_{3/2}}\cap B_{1/2}'$ denotes a point such that $|\bar{x}_0 -x_0| \leq \delta$ (which exists by our assumption (i)).
By virtue of the scaling of $w_{3/2}$ at free boundary points and by the choice of $R$, we thus infer that for all $r\in [R,1/2)$
\begin{align}
\label{eq:low_est}
r^{- \frac{n+1}{2}}\left\| w \right\|_{L^2(B_{r}^+(x_0))} \geq \frac{1}{2} r^{\frac{3}{2}}.
\end{align}
This shows (\ref{eq:nondeg_contr}) for all $r\in [R,1/2)$. \\
We now argue that \eqref{eq:nondeg_contr} is also true for $r\in (0,R)$:
Let $r\in(0,R)$ and $\tau= \frac{1}{1+c_0\frac{\pi}{2}}\frac{\ln(\left\| w \right\|_{L^2(B_r^+(x_0))})-\ln (r)}{\ln(r)}$ (where $c_0$ is the constant in Corollary \corconsequenceCarl). Applying Corollary \corconsequenceCarl at the scales $r, R$ and $1/2$ (where $R$ is defined as above) yields:
\begin{align*}
e^{\tau \tilde{\phi}(\ln R)}R^{-1}|\ln R|^{-2}\left\| w \right\|_{L^2(A_{R,2R}^+(x_0))} &\leq C(n,p) \left( e^{\tau \tilde{\phi}(\ln r)}r^{-1} \left\| w\right\|_{L^2(B_r^+(x_0))} \right.\\
& \quad \left. + e^{\tau \tilde{\phi}(\ln(1/4))} \left\| w\right\|_{L^2(B_{1/2}^+(x_0))} \right)\\
& \leq C(n,p).
\end{align*}
Here the last line follows from the definition of $\tau$. As a result,
\begin{align*}
\left\| w \right\|_{L^2(A_{R,2R}^+(x_0))} \leq C(n,p) e^{-\tau \tilde{\phi}(\ln R)}R |\ln R|^2.
\end{align*}
Combining this with (\ref{eq:low_est}) and that $\tilde{\phi}(t)=-(1+c_0\frac{\pi}{2})t-c_0\left(1+\ln(-t)+\mathcal{O} (\frac{1}{t})\right)$ as $t\rightarrow -\infty$, we infer that for our choice of $R=(4\delta)^{\frac{2}{n+4}}$
\begin{align*}
\frac{1}{2}R^{\frac{n+4}{2}} \leq \left\| w \right\|_{L^2(A_{R,2R}^+(x_0))} \leq C(n,p) R^{\tau(1+c_0\frac{\pi}{2})}e^{\frac{C\tau c_0}{|\ln R|}} R|\ln R|^{2+c_0\tau}.
\end{align*}
Hence, using that $c_0,\tau\leq 1$
\begin{align*}
&\left( 1 +c_0 \frac{\pi}{2} +  \frac{C c_0}{|\ln R|}  \right)\tau + 1  \leq \frac{n+4}{2} - \frac{\ln(C(n,p))}{\ln(R)} - 4\frac{\ln(|\ln R|)}{\ln(R)} ,
\end{align*}
which, by plugging in the expression for $\tau$, can be rewritten as 
\begin{align*}
&\left(1 + \frac{Cc_0}{(1+c_0 \frac{\pi}{2})|\ln R|} \right)\ln(\left\| w \right\|_{L^2(B_r^+(x_0))}) \ \\
&  \geq \left( \frac{n+4}{2}  - \frac{\ln(C(n,p))}{\ln(R)} - 4 \frac{\ln|\ln(R)|}{\ln(R)} + \frac{Cc_0}{1+c_0 \frac{\pi}{2}|\ln R|}\right) \ln(r).
\end{align*}
Thus, a sufficiently small choice of $R$ (and thus $\delta $ depending on $\lambda$) then yields \eqref{eq:nondeg_contr}. Recalling that $r\in(0,R)$ was chosen arbitrarily and combining the just derived estimate with (\ref{eq:low_est}), finally yields the full non-degeneracy condition (\ref{eq:nondeg_contr}) (i.e. for the full range of radii $r\in(0,1/2)$) and hence concludes the proof of the claim.\\

As a direct consequence of the claim, we note that any free boundary point $x_0\in \Gamma_w \cap B_{1/2}'$ has vanishing order at most $\frac{3}{2}+\lambda$. By choosing $\lambda < \frac{1}{2}$, we infer that $\kappa_{x_0}\in [0,2)$. By Proposition \prophomotwo this implies that $\kappa_{x_0}= 3/2$. Therefore, $\Gamma_w \cap B_{1/2}' = \Gamma_{3/2}(w)\cap B_{1/2}'$.\\

\emph{Step 2: Proof of (\ref{eq:Reifenb_close}).} 
We argue by contradiction. 
Suppose that (\ref{eq:Reifenb_close}) were wrong. Then there existed a parameter $\epsilon_0>0$, a sequence $\delta_k\rightarrow 0$ and a sequence, $w^k$, of solutions to the thin obstacle problem with associated metrics $a^{ij}_k$, i.e.
\begin{align*}
\p_i a^{ij}_k \p_j w^k = 0 \text{ in } B_1^+,\\
w^k\geq 0, \ -\partial_{n+1} w^k\geq 0, \ w^k \p_{n+1}w^k=0 \text{ on } B'_1,
\end{align*} 
as well as a sequence of radii $r_k \in (0,1/2)$ and a sequence of points $x_k \in \Gamma_{w^k}\cap B_{1/2}'$ such that 
\begin{itemize}
\item[(i)] $\|w^k - w_{3/2}\|_{C^1(B_{1}^+)}\leq \delta_k$, 
\item[(ii)] $\|a^{ij}_k-\delta^{ij}\|_{L^{\infty}(B_1^+)}\leq \delta_k$,
\end{itemize}
but for all $S\in SO(n+1)$
\begin{align}
\label{eq:contra}
\left\| w_{r_k,x_k}^k(x) - w_{3/2}(S x) \right\|_{C^1(B_{1/2}^+)} \geq \epsilon_0.
\end{align} 
We can also assume that $\Gamma_{w^k}\cap B_{1/2}'  \ni x_k \rightarrow x_0 \in \Gamma_{w_{3/2}}\cap \overline{B_{1/2}'}$ and that $r_k \rightarrow r_0$.\\

We now show how the compactness statement of step 1 yields a contradiction to our assumption (\ref{eq:contra}). For this we consider the $L^{2}$ normalized rescalings, $w_{r,\bar{x}}^k(x)$, of $w$ at a free boundary point $\bar{x}\in \Gamma_{w^k}\cap B_{1/2}'$. As the conditions of the claim from step 1 are satisfied at any free boundary point $\bar{x}\in \Gamma_{w^k}\cap B_{1/2}'$, the doubling property (c.f. Proposition \propdoubling) is satisfied uniformly in the choice of the free boundary point $\bar{x}\in \Gamma_{w^k}$ and in the sequence $w^k$. Hence, considering the sequences $r_k$, $x_k$ from above, entails that on the one hand $w_{r_k,x_k}^k(x) $ is bounded in $H^1(B_1^+)$. Therefore, up to a subsequence it converges strongly in $L^2(B_1^+)$ to a limit $w_0(x)\in H^1(B_1^+)$ with $\|w_0\|_{L^2(B_1^+)}=1$. By the $L^2-C^{1,\alpha}$ estimates this also entails that $w_{r_k,x_k}^k(x) \rightarrow w_0$ in $C^1(B_{1/2}^+)$. 
On the other hand, we recall the $C^1$ convergence 
\begin{align}
\label{eq:converge}
w^k (x)\rightarrow w_{3/2}(x),
\end{align}
which follows from (i). We now distinguish two cases:
\begin{itemize}
\item[Case 1:] $r_0>0$. In this case we observe that due to (\ref{eq:converge})
\begin{align*}
w^k(x_k + r_k x) &\rightarrow w_{3/2}(x_0+ r_0 x) = r_0^{3/2} w_{3/2}(x),\\
r^{-\frac{n+1}{2}}_k \left\| w^k\right\|_{L^2(B_{r_k}^+(x_k))} &\rightarrow r^{-\frac{n+1}{2}}_0 \left\| w_{3/2}\right\|_{L^2(B_{r_0}^+(x_0))} = r^{3/2}_0,
\end{align*}
where in the last line we used the $L^2$ normalization of $w_{3/2}$ and that $w_{3/2}$ is homogeneous of degree $3/2$. Consequently, $w_{r_k, x_k}^k(x) \rightarrow w_{3/2}(x)$ in $C^1(B_{1/2}^+)$, which is a contradiction to (\ref{eq:contra}).
\item[Case 2:] $r_0=0$. In this case we note that, since $w^k$ is a solution of the thin obstacle problem, the same is true for $w_{r_k,x_k}^k(x)$. Moreover, by assumption (ii) $a^{ij}_k \rightarrow \delta^{ij}$, which implies that $w_0(x)$ is a nontrivial global solution of the constant coefficient thin obstacle problem which satisfies $\|w_0\|_{L^2(B_1^+)}=1$. By \eqref{eq:nondeg_contr} and the uniform upper growth estimate of Lemma \lemgrowthimp, the growth of $w_0$ at infinity is of the rate $R^\kappa$ with $\kappa\in (1,2)$. Thus, by the proof of Lemma \ref{lem:3/2homo}, $w_0(x)= w_{3/2}(S x)$ for some rotation $S\in SO(n+1)$. This yields a contradiction to (\ref{eq:contra}).
\end{itemize}
Combining the previous two cases shows (\ref{eq:Reifenb_close}).
\end{proof}

Proposition \ref{prop:Reifenberg} immediately entails the following Corollary:

\begin{cor}
\label{cor:Reifenberg}
Let $a^{ij}: B_1^+ \rightarrow \R^{(n+1)\times (n+1)}_{sym}$ be a uniformly elliptic $W^{1,p}(B_1^+)$ tensor field with $p\in (n+1,\infty]$ which satisfies (A1), (A2) and (A4). Let $w$ be a solution of (\ref{eq:varcoef}) in $B_{1}^+$. For each $\epsilon>0$ there exists $\delta>0$ such that if
\begin{itemize}
\item[(i)] $\left\| w - w_{3/2} \right\|_{C^1(B_1^+)} \leq \delta$,
\item[(ii)] $\left\| a^{ij} - \delta^{ij} \right\|_{L^{\infty}(B_1^+)} \leq \delta$,
\end{itemize} 
then the free boundary $\Gamma_w\cap B_{1/2}'$ is $(\epsilon, 1/2)$ Reifenberg flat.
\end{cor}

\begin{proof}
We regard $\Gamma_{w}\cap B_{1/2}'$ as a subset of $\R^{n} \times \{0\}$. By Proposition \ref{prop:Reifenberg}, we have that for each $r\in(0,1/2)$ and $x_0\in \Gamma_{w}\cap B_{1/2}'$, there exists $S(r,x_0)\in SO(n+1)$ with
\begin{align}
\label{eq:close1}
\| w_{x_0, r}( S(r,x_0)x) - w_{3/2}(x) \|_{C^1(B_1^+)} \leq \epsilon.
\end{align}
Consider $L(r,x_0):=\{x_0 + x\in \R^{n}\times \{0\}| \  x\cdot (S(r,x_0) e_{n}) = 0\}$. Due to (\ref{eq:close1}), we obtain that
\begin{align*}
\frac{1}{r}\dist_{H} (\Gamma_w \cap B_r(x_0), L(r,x_0)\cap B_{r}(x_0)) \leq \epsilon,
\end{align*}
which shows the $(\epsilon,1/2)$ Reifenberg flatness of $\Gamma_{w}\cap B_{1/2}'$.
\end{proof}

By using an $L^2$ normalized blow-up, this can be transferred to obtain the Reifenberg flatness of the regular free boundary of a general solution of the thin obstacle problem (which does not necessarily satisfy the closeness assumption (i)):

\begin{cor}[Reifenberg flatness of $\Gamma_{3/2}(w)$]
\label{cor:Reifenberg2}
Let $a^{ij}: B_1^+ \rightarrow \R^{(n+1)\times (n+1)}_{sym}$ be a uniformly elliptic $W^{1,p}$ tensor field with $p\in (n+1,\infty]$. Let $w$ be a solution of (\ref{eq:varcoef}) in $B_{1}^+$. 
For any $\epsilon>0$ and any $K\Subset \Gamma_{3/2}(w)$, there exists a radius $r_0= r_0(\epsilon,K,w)\in (0,1)$, such that $\Gamma_{3/2}(w)\cap K$ is $(\epsilon,r_0)$-Reifenberg flat.
\end{cor}

\begin{proof}
\emph{Step 1: Uniformity in compact sets.}
As in Proposition \ref{prop:Reifenberg}, the Corollary reduces to showing the following claim:
\begin{claim*}
\label{claim:Reifenberg2}
Let $a^{ij}: B_1^+ \rightarrow \R^{(n+1)\times (n+1)}_{sym}$ be a uniformly elliptic $W^{1,p}$ tensor field with $p\in (n+1,\infty]$. Let $w$ be a solution of (\ref{eq:varcoef}) in $B_{1}^+$. Then, for any $\epsilon>0$ and $K\Subset \Gamma_{3/2}(w)$, there exists $r_0=r_0(\epsilon, w, K)$  such that for any $x_0\in K$ and $r\leq r_0$, there is a matrix $S=S(r,x_0)\in SO(n+1)$ such that
\begin{align}
\label{eq:close}
\|w_{x_0,r} (Sx)-w_{3/2}(x)\|_{C^1(B_1^+)}\leq \epsilon,
\end{align}
where $w_{3/2}(x)$ is as above.
\end{claim*}
In order to infer the claim, we argue by contradiction and assume that the statement were wrong. Then there were an $\epsilon>0$, a sequence $r_k$ with $r_k\rightarrow 0$ and and a sequence $x_k$ with $x_k \in K$ such that for all $S\in SO(n+1)$
\begin{align*}
\| w_{x_k, r_k}(S x) - w_{3/2}(x) \|_{C^1(B_1^+)} > \epsilon \quad \mbox{ for all } k\in \N.
\end{align*}
By compactness of $K$ we may assume that $x_k \rightarrow \bar{x}$. However, by compactness (which for variable centers we infer from a combination of Proposition \propdoubling and Corollary \corlowunif) of the sequence $w_{x_k, r_k}$ and by Lemma \ref{lem:3/2homo}, we obtain a subsequence, which we do not relabel, and a rotation $S$ such that
\begin{align*}
\| w_{x_k, r_k}( x) - w_{3/2}(S x) \|_{C^1(B_1^+)} \rightarrow 0.
\end{align*}
\emph{Step 2: Conclusion.}
By rescaling with the radius $r_{0}=r_{0}(\epsilon, K, w)$, it is possible to reduce the claim of Corollary \ref{cor:Reifenberg2} to the setting of Corollary \ref{cor:Reifenberg}.
\end{proof}

\subsection{Lipschitz regularity}
\label{subsec:Lip}

In this section we present a first improvement of the regularity of the regular free boundary.  \\

In the sequel we will always assume that $0\in \Gamma_w$. Moreover, we will use the notation
$$w_{3/2}(x):=c_n\Ree(x_n+ix_{n+1})^{3/2},$$
to denote our model solution. Here $c_n>0$ is a normalization constant such that $\|w_{3/2}\|_{L^2(B_1^+)}=1$. Keeping this convention in the back of our minds, we prove the Lipschitz regularity of the regular free boundary:

\begin{prop}[Lipschitz regularity]
\label{prop:Lip}
Let $a^{ij}:B_1^+ \rightarrow \R^{(n+1)\times (n+1)}_{sym}$ be a tensor field which satisfies (A1), (A2), (A3), (A4). Let $w$ be a solution of (\ref{eq:varcoef}) in $B_{1}^+$. 
Assume that for some small positive constants $\epsilon_0$ and $c_\ast$
\begin{itemize}
\item[(i)] $\left\| w -  w_{3/2} \right\|_{C^1(B_{1}^+)} \leq \epsilon_0$,
\item[(ii)]$\|\nabla a^{ij}\|_{L^p(B_1^+)}\leq c_\ast$.
\end{itemize}
Then if $\epsilon_0$ and $c_\ast$ depending on $n,p$ are chosen sufficiently small, there exists a Lipschitz function $g:B_{1/2}'' \rightarrow \R$ such that (possibly after a rotation)
\begin{align*}
\Gamma_w\cap B_{1/2}' = \{x_{n}= g(x'') \big| x'' \in B_{1/2}'' \}.
\end{align*}
Moreover, there is a large constant $\tilde{C}=\tilde{C}(n,p)$, such that $\|\nabla'' g\|_{L^\infty(B''_{1/2})}\leq \tilde{C}\max\{\epsilon_0,c_\ast\}$.
\end{prop}

We remark that in addition to the results from Proposition \ref{prop:Lip}, we actually show that $g$ -- and hence the regular free boundary -- is $C^{1}(B_{1/2}'')$ regular (c.f. Remark \ref{rmk:diff}). \\

In order to prove the Lipschitz regularity of the free boundary, we deduce positivity of the directional derivatives in a cone of tangential directions (c.f. Proposition \ref{prop:nondeg}). 
To achieve this, we use a comparison argument, which is standard in the constant coefficient case. We start by considering the equation for $v:=\p_e w$, where $e\in S^n \cap B_1'$ is a tangential direction. From the $H^2$ estimate of $w$ (c.f. \cite{U87}), $v\in H^1(B^+_1)$. By differentiating the equation of $w$, we obtain the equation of $v$ in $B^+_1\cap \{x_{n+1}>0\}$:
\begin{align*}
\p_i a^{ij}\p_j v &=\p_i F^i\text{ in } B_1^+\cap \{x_{n+1}>0\},\quad\text{where } F^i =-(\p_ea^{ij})\p_j w.
\end{align*} 
Compared with the constant coefficient situation, the fact that the coefficients of our equation are not better than $W^{1,p}$ regular, leads to a divergence right hand side with $F^i\in L^p$. This causes difficulties in applying the comparison argument directly. To resolve this difficulty we use an appropriate decomposition to split the solution $v$ into an ``error term'' $v_1$ (c.f. equation \eqref{eq:v1}), which deals with the divergence term $F^i$, and into a ``main term'' $v_2$ (c.f. equation \eqref{eq:v2}), which captures the behavior of the respective solutions. 

\begin{rmk}[Reflection and extension]
\label{rmk:ref_ext}
As it proves convenient to work on the whole ball instead of only on the upper half ball (for instance for applying interior elliptic regularity theory), we reflect the metric $a^{ij}$ to $B_1^-$ by setting
\begin{equation}
\label{eq:extend_reflect}
\begin{split}
a^{ij}(x',x_{n+1})&=a^{ij}(x',-x_{n+1}),\quad i,j=1,\dots,n,\\
a^{n+1,j}(x',x_{n+1})&=-a^{n+1,j}(x',-x_{n+1}),\quad j=1,\dots,n,\\
a^{n+1,n+1}(x',x_{n+1})&=a^{n+1,n+1}(x',-x_{n+1}),
\end{split}
\end{equation}
for $(x',x_{n+1})\in B_1^-$. We reflect $w$ to $B^-_1$ by defining $w(x',x_{n+1})=w(x',-x_{n+1})$.  Here we use the same notation to denote the functions after refection. 
Due to the off-diagonal assumption (A1), the reflected coefficient matrix $a^{ij}$ is in $W^{1,p}(B_1)$.
Since $\p_{n+1}w=0$ in $B'_1\setminus \Lambda_w$, we have $w\in C^{1,1/2-\delta}(B_1\setminus \Lambda_w)\cap H^2(B_1\setminus \Lambda_w)$, where $\delta\in(0,1)$ is an arbitrarily small constant, and $v=\p_ew$ satisfies
\begin{align}\label{eq:extend_v}
\p_ia^{ij}\p_j v&=\p_i F^i\text{ in } B_1\setminus \Lambda_w, \quad F^i=-(\p_e a^{ij})\p_j w, \\
v&=0\text{ on } \Lambda_w. \notag
\end{align}

As we are mainly interested in the local properties of solutions and of the (regular) free boundary and in order to avoid technical issues on the boundary $\p B_1\cap \Lambda_w$, we will often consider \eqref{eq:extend_v} as an equation in the whole space $\R^{n+1}$: We simply extend $a^{ij}$ to $\R^{n+1}$ such that the extended coefficients remain Sobolev regular, $a^{ij}\in W^{1,p}_{loc}(\R^{n+1})$, with equivalent ellipticity constants and with $\Vert \nabla a^{ij} \Vert_{L^p(\R^{n+1})} $ at most twice as large as it had originally been. Moreover, we extend $F^i$ to $\R^{n+1}$ by multiplying by the characteristic function $\chi_{B_1}$. With a slight abuse of the notation, we denote the set $\Lambda_w\cap B'_1$ by $\Lambda_w$. In the sequel, whenever the equations are written in the whole space, we will always have this extension procedure in mind, even though we may not refer to it explicitly.
\end{rmk}  

In order to deal with the $L^p$ divergence form right hand side in (\ref{eq:extend_v}), we decompose $v=v_1+v_2$ with 
\begin{equation}
\label{eq:v1}
\begin{split}
\p_i a^{ij} \p_jv_1 -K\dist(x,\Gamma_w)^{-2}v_1 &= \p_i F^i \mbox{ in } \R^{n+1}\setminus \Lambda_w,\\
v_1&= 0 \mbox{ on } \Lambda_w,
\end{split}
\end{equation}
where $K$ is a large constant. Then $v_2:= v - v_1$ solves 
\begin{equation}
\label{eq:v2}
\begin{split}
\p_i a^{ij} \p_j v_2 &= - K\dist(x,\Gamma_w)^{-2} v_1 \mbox{ in } \R^{n+1} \setminus \Lambda_w,\\
v_2 &= 0 \mbox{ on } \Lambda_w.
\end{split}
\end{equation} 
Intuitively, we expect that $v_1$ is a ``controlled error'' (c.f. the bounds in Proposition \ref{prop:v1}) and that $v_2$ determines the behavior of our solution $v$.\\

That this is indeed the case and that $v_1$ can be treated as a ``controlled error'' is a consequence of the following more general result which will be proved in Section \ref{sec:propv1}:

\begin{prop}
\label{prop:v1}
Let $\Lambda$ and $\Gamma$ be closed, non-empty sets in $\R^n\times\{0\}$ with $\Gamma = \partial \Lambda$. Let $a^{ij}:\R^{n+1} \rightarrow \R^{(n+1)\times (n+1)}_{sym}$ be a bounded, measurable, uniformly elliptic tensor field whose eigenvalues are in $[1/2, 2]$. For $K>0$ consider the equation
\begin{equation}
\label{eq:globalv1}
\begin{split}
\partial_ia^{ij}\partial_j u-K\dist(x, \Gamma)^{-2} u &= \partial_i F^i + g \text{ in } \R^{n+1}\setminus \Lambda,\\
u&=0 \text{ on } \Lambda,
\end{split}
\end{equation}
with $F^i\in L^2(\R^{n+1})$, $i\in \{1,\dots,n+1\}$ and $\dist(x,\Gamma) g\in L^2(\R^{n+1})$. Then there exists a unique solution $u$ with $\nabla u \in L^2(\R^{n+1})$, $\dist(x,\Gamma)^{-1}u\in L^2(\R^{n+1})$ and $u(x)=0$ on $\Lambda$ for $H^n$ a.e. $x$. Under the above conditions, if we additionally
assume that for some $\sigma\geq 0$ and $p\in (n+1,\infty]$
\begin{align*}
F^i\dist(x,\Gamma)^{-\sigma}&\in L^p(\R^{n+1}), \  i\in\{1,\dots, n+1\},\\
 g\dist(x,\Gamma)^{1-\frac{n+1}{p}-\sigma}&\in L^{p/2}(\R^{n+1}), 
\end{align*}
then there exist $c=c(n)$ large and $C_0=C_0(n,p)$,  such that for $K=c(n)(\sigma+1)$,
\begin{multline*}
     |u(x)| \le C_0   \dist(x,\Gamma)^{1-\frac{n+1}{p} +\sigma}\left(\left\|F^i \dist(\cdot,\Gamma)^{-\sigma}\right\|_{L^p(\R^{n+1})}\right.\\
     \left.+\|g \dist(\cdot,\Gamma)^{1-\frac{n+1}{p}-\sigma}\|_{L^{p/2}(\R^{n+1})}\right).
\end{multline*} 
Moreover, $u$ vanishes continuously as $x$ approaches $\Lambda$.
\end{prop}

\begin{rmk}
With slight modifications of the function spaces, the statement of Proposition \ref{prop:v1} remains true for $\sigma < 0$.
\end{rmk}

\begin{rmk}
\label{rmk:impv1}
If in addition $a^{ij}\in W^{1,p}(B_1^+)$ with $p\in (n+1,\infty)$, it is possible to make the vanishing order of $u$ towards $\Lambda$ more explicit (c.f. Remark~\ref{rmk:impv2} after the proof of Proposition \ref{prop:v1}). More precisely, for $p\in (n+1,\infty)$ we have
\begin{multline*}
     |u(x)| \le C_0  \dist(x,\Gamma)^{\sigma}\dist(x,\Lambda)^{1-\frac{n+1}{p}}  \left(\left\|F^i \dist(\cdot,\Gamma)^{-\sigma}\right\|_{L^p(\R^{n+1})}\right.\\
     \left. +\|g \dist(\cdot,\Gamma)^{1-\frac{n+1}{p}-\sigma}\|_{L^{p/2}(\R^{n+1})}\right).
\end{multline*} 
\end{rmk}

Although we postpone the proof of Proposition \ref{prop:v1} to the end of this section, we will in the sequel frequently apply it, for instance with $\Lambda = \Lambda_w$ and $\Gamma = \Gamma_w$.\\

The remainder of the proof of Proposition \ref{prop:Lip} is organized as follows: Working in the framework of general $(\epsilon,1)$-Reifenberg flat sets $\Lambda$ and $\Gamma$, in Section \ref{sec:Reifenbergbarrier} we first construct a lower barrier function (Proposition \ref{prop:barrier}) which serves as the basis for the then following comparison argument. Here we deduce an (almost) optimal non-degeneracy condition in cones for elliptic equations with controlled inhomogeneities in domains with Reifenberg slit (Proposition \ref{prop:nondeg}). In Section \ref{sec:proofLip} this is then applied to prove the Lipschitz regularity of the (regular) free boundary for the thin obstacle problem (Proposition \ref{prop:Lip}).

\subsubsection{Construction of a barrier function and a comparison argument} 
\label{sec:Reifenbergbarrier}

In this section we first construct a barrier function for the slit domain with Reifenberg flat slit (c.f. Proposition \ref{prop:barrier}). This is crucial for the then following comparison argument which allows us to work with essentially critically scaling bounds.\\

In the whole subsection, we work under the following assumption:
\begin{assum}
\label{assum:Reifenberg}
We consider a closed set $\Lambda\subset\R^n\times\{0\}$ with boundary $\Gamma:= \partial_{\R^n\times\{0\}} \Lambda$. We assume that $0\in \Gamma$ and that $\Gamma\cap B_1$ is $(\epsilon, 1)$-Reifenberg flat. Here $\epsilon>0$ is a fixed small constant. 
\end{assum}

We start by considering a Whitney decomposition,  $\{Q_j\}_{j\in \N}$, of $B_{1} \setminus \Gamma$, i.e. $\{Q_j\}$ is a collection of dyadic cubes $Q_j$ with disjoint interiors that cover $B_{1} \setminus \Gamma$ such that: 
\begin{itemize}
\item [(W1)] $\diam(Q_j)\leq \dist(Q_j, \Gamma \cap B_{1})\leq 4\diam(Q_j)$.
\item [(W2)] If $Q_1$ and $Q_2$ touch, then $(1/4)\diam(Q_2)\leq \diam(Q_1)\leq 4\diam (Q_2)$.
\item [(W3)] There are at most $(12)^{n+1}$ cubes in $\{Q_j\}$ which touch a given cube $Q\in \{Q_j\}$ (cf. Chapter VI \cite{St} for the existence of such cubes). 
\end{itemize}
Moreover, as $\Lambda, \Gamma \subset \R^{n}\times \{0\}$, it is further possible to choose the Whitney cubes symmetric with respect to the plane $\{x_{n+1}=0\}$.\\

For each Whitney cube $Q_j$ with center $\hat x_j$ and diameter $r_j$, let $x_j \in \Gamma$ be a (not necessarily unique) point such that $\dist(\hat x_j,\Gamma)=|\hat x_j-x_j|$. Using this point, we associate a rotation $S_j$ to each Whitney cube $Q_j$ by choosing one of the (not necessarily unique) rotations $S(64 r_j,x_j)$ which are defined by the property that $S(64 r_j,x_j) e_n$ determines the normal of the plane $L( 64 r_j, x_j)$ in the Reifenberg condition (\ref{eq:Reifenberg}) with $\delta = \epsilon$ for some $\epsilon\ll 1$. We define $\nu_j:=S_j e_n$, which we refer to as the \textsl{approximate normal} at $x_j$ in $B_{64r_j}(x_j)$ associated with $Q_j$. Note that $\nu_j \in S^n\cap \{x_{n+1}=0\}$, as we interpret $\Gamma$ as a subset of $\R^{n}\times \{0\}$.\\

We observe that the approximate normals only vary slowly on neighboring Whitney cubes:

\begin{lem}
\label{lem:normals}
Let $ \Lambda, \Gamma$ be as in Assumption \ref{assum:Reifenberg}.
Let $\{Q_j\}$ be a Whitney decomposition of $B_{1}\setminus \Gamma$ and let $Q_j$, $Q_k$ be neighboring Whitney cubes. Let $\nu_j$ and $\nu_k$ be the associated approximate normals defined above. Then
\begin{align*}
|\nu_j - \nu_k| \leq C \epsilon.
\end{align*}
\end{lem}

\begin{proof}
Let $\hat x_j$ and $\hat x_k$ be the centers of the Whitney cubes $Q_j$ and $Q_k$ respectively, and let $x_j,x_k\in \Gamma_w$ be (not necessarily unique) points which realize the distance of $\hat x_j,\hat x_k$ to $\Gamma_w$. Let $r_j$ be the diameter of $Q_j$. By using (W1) and (W2), it is immediate to check that
\begin{align*}
0\leq |x_j-x_k|\leq 32 r_j.
\end{align*}
Let $S_j=S(64r_j, x_j)$ and let $\nu_j=S_j e_n$ be as above. 
Let $L_j= \{x\in \R^n\times \{0\}| \ x\cdot \nu_j=0\}$ be an $(n-1)$-dimensional hyperplane in $\R^n\times \{0\}$. Then condition \eqref{eq:Reifenberg} in Definition \ref{defi:Reifenberg} gives that 
$$\dist_H(\Gamma_w\cap B_{64 r_j}(x_j), x_j+L_j \cap B_{64 r_j}(x_j))\leq C\epsilon  r_j.$$ 
Similarly at $x_k$ we have
\begin{align*}
\dist_H(\Gamma_w\cap B_{64 r_k}(x_k), x_k+L_k \cap B_{64 r_k}(x_k))\leq C\epsilon r_k.
\end{align*}
Comparing the above two inequalities and using the triangle inequality for $\dist_H$, we deduce
$$|\nu_j-\nu_k|\leq C\epsilon.$$
\end{proof}

With this auxiliary result at hand, we construct a first barrier function for the divergence form operator $L= \p_i a^{ij}\p_j$ with $a^{ij}$ satisfying the assumptions (A2), (A3), (A4). The idea is first to locally construct subsolutions in each Whitney cube $Q_j$ and then to patch them together by a partition of unity. The patched function remains a subsolution as the approximate normals vary slowly due to Lemma~\ref{lem:normals}. In order to implement this idea, in a first step, we view the operator $L$ as a perturbation of the non-divergence form operator $L_0:= a^{ij}\p_{ij}$. Then in the second step we use the splitting technique from Proposition \ref{prop:v1} to derive a barrier for the full operator $L$.\\

In order to simplify notation, we introduce the following parameters as abbreviations:
\begin{notation} 
\label{not:not1}
We abbreviate by setting
\begin{itemize}
\item $c_{\ast}:=\left\| \nabla a^{ij} \right\|_{L^p(B_1)}$,
\item $\ell_0:= (2^{4}\sqrt{n})^{-1}$.
\end{itemize}
\end{notation}

Using this notation we proceed with the barrier construction.

\begin{prop}[Barrier function]
\label{prop:barrier}
Let $\Lambda, \Gamma, \epsilon$ be as in Assumption \ref{assum:Reifenberg} and let $c_{\ast}, \ell_0$ be as in Notation \ref{not:not1}.
Then for any $s\in (0,1/2)$, if $\epsilon=\epsilon(n,s)$ and $c_\ast=c_\ast(n,p,s)$ are chosen sufficiently small, there exists a function $h^-_s: B_1 \rightarrow \R $ such that
\begin{itemize}
\item[(i)] $L h^-_s(x) \geq c_s\dist(x,\Gamma)^{-\frac{3}{2}+\frac{s}{2}} \mbox{ in } B_1 \setminus \Lambda$ for some $c_s>0$,
\item[(ii)] $h^-_s(x)   \geq  c_n \dist(x,\Gamma)^{\frac{1}{2}+\frac{s}{2}} \text{ on } B_1\cap \{x| \ \dist(x,\Lambda)\geq\ell_0 \dist(x,\Gamma)\}$,
\item[(iii)] $h^-_s(x) \geq -c_n \dist(x,\Gamma)^{\frac{3}{2}-\frac{n+1}{p}}$ in $B_1$ and $h^-_s(x)=0$ on $\Lambda$. 
\end{itemize}
\end{prop}

\begin{proof}
\emph{Step 1: Barrier for $L_0$.}
Let $\eta_j$ be a partition of unity associated to the Whitney decomposition of $B_1\setminus \Gamma$ from above, i.e. the functions $\eta_j$ satisfy $0\leq \eta_j\leq 1$, $\eta_j = 1$ on $\frac{1}{2}Q_j$ and $\eta_j = 0$ for all $x\in Q_k$ if $Q_k\notin \mathcal{N}(Q_j):=\{Q_i| \ Q_i \text{ touches }Q_j\}$.
We consider
\begin{align*}
h^-(x) :=  \sum\limits_{k\in \N} \eta_k (x) v_k^-(x),
\end{align*}
where $v_k^-$ is constructed as follows:
We first define an affine transformation, $T_k$, associated with each Whitney cube $Q_k$. More precisely, we set $T_k (x):= R_k B^{-1}_k(x- x_k)$, where as before $x_k$ denotes a point on $\Gamma\cap B_1$ which realizes the distance of the center $\hat x_k$ of $Q_k$ to $\Gamma\cap B_1$, $B_k B_k^{t}:=(a^{ij}(x_k))$, and $R_k\in SO(n+1)$ is defined such that $R_k B^{t}_k \nu_k =\bar{c}_{1,k} e_n$ and $R_k B^{t}_k e_{n+1}=\bar{c}_{2,k} e_{n+1}$. Here $\nu_k$ is the associated approximate normal which was defined before and described in Lemma~\ref{lem:normals}. The constants $\bar{c}_{1,k},\bar{c}_{2,k}$ with $\bar{c}_{1,k},\bar{c}_{2,k} \in (\frac{1}{\sqrt{2}}, \sqrt{2})$ only depend on the ellipticity constants of $a^{ij}$. 
Next, for $s\in (0,1)$ we consider
$$v^-(x) := w_{1/2}(x_n, x_{n+1})^{1+s}, \quad w_{1/2}(x_n, x_{n+1})=\Ree(x_n+i|x_{n+1}|)^{1/2}.$$
A direct computation leads to 
\begin{align*}
\Delta v^-(x)&=s(1+s)(x_n^2+x_{n+1}^2)^{-1/2}(w_{1/2})^{-1+s}\\
&\geq s(1+s) (x_n^2+x_{n+1}^2)^{-\frac{3}{4}+\frac{s}{4}}.
\end{align*}
Then we define 
$$v_k^-(x):=v^-(T_k x).$$ 
We observe that by definition
\begin{align}\label{eq:Tk}
T_k (x)\cdot e_n = \bar{c}_{1,k}^{-1} (x-x_k)\cdot \nu_k, \quad T_k (x)\cdot e_{n+1}= \bar{c}_{2,k}^{-1} (x-x_k)\cdot e_{n+1}.
\end{align}
Thus, $v_k^-(x)=(\Ree(\bar{c}_{1,k}^{-1}(x-x_k)\cdot \nu_k+ i |\bar{c}_{2,k}^{-1} (x-x_k)\cdot e_{n+1}|)^{\frac{1}{2}})^{1+s}$.
By using the computation for $\Delta v^-(x)$, one can check that $v_k^-$ satisfies 
\begin{align*}
a^{ij}(x_k)\p_{ij} v_k^-(x) &\geq s(1+s)((T_k(x)\cdot e_n)^2+(T_k(x)\cdot e_{n+1})^2)^{-\frac{3}{4}+\frac{s}{4}}\\ &\quad \text{ in } \R^{n+1}\setminus \{x_{n+1}=0, (x-x_k)\cdot \nu_k\leq 0\},\\
v_k^-(x) &= 0 \text{ on } \{x_{n+1}=0, (x-x_k)\cdot \nu_k\leq 0\}.
\end{align*}
Due to the $(\epsilon, 1)$-Reifenberg condition on $\Gamma$, the definition \eqref{eq:Tk} and due to the property (W2), there exists an absolute constant $c\in (0,1)$ such that
\begin{align*}
a^{ij}(x_k)\p_{ij} v_k^-(x) \geq cs(1+s) r_k^{-\frac{3}{2}+\frac{s}{2}}, \quad r_k=\diam(Q_k)
\end{align*}
for each $x\in \supp(\eta_k)$.
By (W1) and (W2) this can further be rewritten as
\begin{align}\label{eq:lower_bar}
a^{ij}(x_k)\p_{ij} v_k^-(x) \geq cs(1+s) \dist(x,\Gamma)^{-\frac{2}{3}+\frac{s}{2}}, \quad x\in \supp(\eta_k).
\end{align}
We compute $L_0(\sum_k\eta_kv_k^-)$, which is
\begin{align*}
L_0(\sum_k\eta_k v_k^-)&=\sum_k a^{ij}(x_k)\p_{ij}(\eta_kv_k^-)+\sum_k(a^{ij}(x)-a^{ij}(x_k))\p_{ij}(\eta_k v_k^-)\\
&=\sum_k(a^{ij}(x_k)\p_{ij}v_k^-) \eta_k + \sum_k(a^{ij}(x_k)\p_{ij}\eta_k) v_k^- \\
& \quad+ 2\sum_k a^{ij}(x_k)\p_i\eta_k\p_j v_k^-  +\sum_k(a^{ij}(x)-a^{ij}(x_k))\p_{ij}(\eta_k v_k^-).
\end{align*}
Given any $x\in B_1\setminus \Gamma$, there is a unique $Q_\ell$ such that $x\in Q_\ell$. By definition of our partition of unity, $\eta_j \neq 0$ in $Q_\ell$ iff $Q_j \in \mathcal{N}(Q_\ell)$. 
Using \eqref{eq:lower_bar} and (W2) we thus deduce (with a possibly different $c$)
\begin{align*}
\sum_k(a^{ij}(x_k)\p_{ij}v_k^-(x)) \eta_k(x)\geq cs(1+s) r_\ell^{-\frac{3}{2}+\frac{s}{2}}.
\end{align*}
This will be the main contribution in $L_0 h^-$. In order to see this, we now estimate the error terms:
For any $x\in  Q_\ell$,
\begin{align*}
&\sum_k(a^{ij}(x_k)\p_{ij}\eta_k(x))v_k^-(x)\\
=&\sum_k [a^{ij}(x_k)-a^{ij}(x_\ell)] \p_{ij}\eta_k(x) v_k^-(x) + \sum_k a^{ij}(x_\ell) \p_{ij} \eta_k(x) v_k^-(x)\\
=&\sum_k [a^{ij}(x_k)-a^{ij}(x_\ell)] \p_{ij}\eta_k(x) v_k^-(x) + \sum_k a^{ij}(x_\ell) \p_{ij} \eta_k(x) [v_k^-(x)-v_\ell^-(x)]\\
\leq &C(n)\left(c_{\ast} r_\ell^{\gamma -\frac{3}{2}+\frac{s}{2}}+ \epsilon r_\ell^{-\frac{3}{2}+\frac{s}{2}}\right), \quad \gamma=1-\frac{n+1}{p},
\end{align*}
where in the second equality we used $a^{ij}(x_\ell)\p_{ij} (\sum_k \eta_k(x))=\sum_k a^{ij}(x_\ell)\p_{ij}\eta_k(x)=0$. In the third inequality we applied Lemma~\ref{lem:normals} and noted that for $x\in \supp(\eta_\ell)$ and $Q_k\in \mathcal{N}(Q_\ell)$,
\begin{align}\label{eq:barrier_est}
|v_k^-(x)-v_\ell^-(x)|=|v^-(T_k(x))-v^-(T_\ell (x))|\leq C \max\{c_\ast r_\ell^{\gamma}, \epsilon\} r_\ell^{\frac{1+s}{2}}.
\end{align} 
Indeed, this follows, since $|T_k(x)-T_l(x)| \lesssim \max\{c_\ast r_\ell^{\gamma}, \epsilon\} r_\ell $.\\
Similarly, by using $\sum_k a^{ij}(x_\ell)\p_i\eta_k(x)=0$ for all $j\in \{1,\dots, n+1\}$ as well as Lemma~\ref{lem:normals}, we infer that
\begin{align*}
&\sum_k a^{ij}(x_k)\p_i\eta_k\p_j v_k^- \\
&= \sum_k [a^{ij}(x_k)-a^{ij}(x_\ell)]\p_i\eta_k(x)\p_j v_k^-(x)+ \sum_k a^{ij}(x_\ell) \p_i \eta_k(x)[\p_jv_k^-(x)-\p_jv_\ell^-(x)]\\
&\leq C(n)\left(c_{\ast} r_\ell^{\gamma -\frac{3}{2}+\frac{s}{2}}+ \epsilon r_\ell^{-\frac{3}{2}+\frac{s}{2}} \right).
\end{align*}
Analogous to the estimate from before we used that
\begin{align*}
|\p_j v_k^-(x)-\p_jv_\ell^-(x)|=|\p_j v^-(T_k x)-\p_jv^-(T_\ell x)|\lesssim \max\{c_\ast r_\ell^{\gamma},\epsilon\} r_\ell^{\frac{s-1}{2}} 
\end{align*}
in the last inequality.
Finally, 
\begin{align*}
\sum_k(a^{ij}(x)-a^{ij}(x_k))\p_{ij}(\eta_k v_k^-(x))\leq C(n)c_{\ast} r_\ell^{\gamma-\frac{3}{2}+\frac{s}{2}}.
\end{align*}
Combining all the previous estimates, if $\epsilon=\epsilon(n,s)$ and $c_{\ast} =c_{\ast}(n,s)$ are small enough, we have for some $c_s>0$
\begin{equation}
\label{eq:lower_bound}
L_0h^-(x) \geq c_s\dist(x,\Gamma)^{- \frac{3}{2}+\frac{s}{2}}.
\end{equation}

Next we show that $h^-=0$ on $\Lambda$. By construction we directly note that $h^-=0$ on $\Lambda\setminus \Gamma$. Due to the continuity of $h^-$ and due to the fact that $\Gamma = \partial \Lambda$, this then also holds on $\Gamma$.\\

\emph{Step 2: Perturbation.}
We now modify $h^-$, so as to become a barrier for the full operator $L$. For this we note that $G= \chi_{B_{1}}(\p_i a^{ij}) \p_j h^-$ satisfies the bound
\begin{align*}
\| G(x) \dist(x,\Gamma)^{\frac{1}{2}} \|_{L^p(\R^{n+1})} \leq C(n) c_{\ast},
\end{align*}
since
\begin{align*}
|\p_j h^-(x)| \leq C(n)\dist(x,\Gamma)^{-\frac{1}{2}} \mbox{ for all } x\in B_{1}\setminus \Lambda.
\end{align*}
Let $q(x)$ be the solution of
\begin{align*}
L q(x) - K \dist(x,\Gamma)^{-2} q(x) &= - G(x) \mbox{ in } \R^{n+1}\setminus \Lambda,\\
q(x) & = 0 \mbox{ on } \Lambda,
\end{align*}
for some large constant $K=K(n)$. Then, by Proposition \ref{prop:v1} there exists a constant $C_0(n,p)$ such that
\begin{align}\label{eq:est_q}
|q(x)| \leq C_0 c_{\ast} \dist(x,\Gamma)^{\frac{3}{2}- \frac{n+1}{p}},
\end{align}
and $q$ vanishes continuously up to $\Lambda$. Let 
$$h^-_s(x):=h^-(x)+q(x).$$
Thus, for $c_{\ast}=c_{\ast}(n,p,s)$ small, we have 
\begin{align*}
L(h^- _s) &= L_0 h^- +(\p_i a^{ij}) \p_j h^- + K\dist(x,\Gamma)^{-2}q - (\p_i a^{ij}) \p_j h^-\\
& \geq c_s\dist(x,\Gamma)^{-\frac{3}{2}+\frac{s}{2}} - C_0 K c_{\ast} \dist(x,\Gamma)^{-\frac{1}{2} - \frac{n+1}{p}}\\
& \geq \frac{1}{2}c_s\dist(x,\Gamma)^{-\frac{3}{2}+\frac{s}{2}}.
\end{align*}
Moreover, we note that for each $k$ with $\supp(\eta_k)\cap \{\dist(x,\Lambda)\geq \ell_0\dist(x,\Gamma)\}\neq \emptyset$ and for each $x\in\supp(\eta_k)\cap \{\dist(x,\Lambda)\geq \ell_0\dist(x,\Gamma)\}$, we have that $v_k^-(x)\geq c_n \dist(x,\Gamma)^{\frac{1}{2}+\frac{s}{2}}$. Using the definition of $h^-$ and the estimate \eqref{eq:est_q} for $q(x)$ (as well as a sufficiently small choice of $c_{\ast}$), we therefore arrive at
\begin{align}
\label{eq:h+}
h^-_s(x)\geq c_n \dist(x,\Gamma)^{\frac{1}{2}+\frac{s}{2}} \text{ for } x\in B_1\cap \{\dist(x,\Lambda)\geq \ell_0 \dist(x,\Gamma)\}.
\end{align}
\end{proof}

Using the previously constructed barrier function, we proceed by deducing a central non-degeneracy condition (in cones) for solutions of elliptic equations with controlled inhomogeneities in domains with Reifenberg slits. This corresponds to a comparison principle for solutions to divergence form, elliptic equations with rough metrics and divergence form right hand sides in domains with Reifenberg slits. \\
For convenience of notation, we change the geometry slightly, i.e. in the following we use cylinders instead of balls: For $r,\ell>0$ and $x_0\in \R^{n}\times\{0\}$, $B'_r(x_0)\times (-\ell,\ell):=\{(x',x_{n+1}) \in \R^{n+1}| \ |x'-x'_0|<r, |x_{n+1}|<\ell\}$.

\begin{prop}
\label{prop:nondeg}
Let $\Lambda, \Gamma, \epsilon$ be as in Assumption \ref{assum:Reifenberg}. Assume that $c_{\ast}$, $\ell_0$ are as in Notation \ref{not:not1}.
Suppose that $u\in H^1(B'_1\times (-1,1))\cap C(B'_1\times (-1,1))$ solves
\begin{align*}
\p_ia^{ij}\p_j u=\p_i F^i +g_1 + g_2 \text{ in } B'_1\times (-1,1)\setminus \Lambda, \quad u= 0 \text{ on } \Lambda,
\end{align*}
and that the following conditions are satisfied:
\begin{itemize}
\item[(i)] The metric $a^{ij}$ satisfies the assumptions (A1), (A2), (A3), (A4) and for some constants $\delta_0\in (0,1)$, $\sigma>\frac{n+1}{p}-\frac{1}{2}$ and $\bar{\delta}>0$,
\begin{align*}
&\sum_{i=1}^{n+1}\left\|\dist(\cdot, \Gamma)^{-\sigma}F^i \right\|_{L^p(B'_1\times(-1,1))}
+
\left\|\dist(\cdot,\Gamma)^{1-\frac{n+1}{p}-\sigma}g_1\right\|_{L^{p/2}(B'_1\times(-1,1))}\\
& \quad + 
\left\|\dist(\cdot,\Gamma)^{\frac{3}{2}-\delta_0}g_2\right\|_{L^{\infty}(B'_1\times(-1,1))}\leq \bar{\delta}.
\end{align*}
\item[(ii)] $u \geq 1$ on $B'_1\times [-1,1]\cap \{|x_{n+1}|\geq \ell_0\}$.
\item[(iii)] $u\geq -2^{-8} $ in $B'_1\times (-\ell_0,\ell_0)$.
\end{itemize}
Then, if $\epsilon$, $\bar \delta$, $c_{\ast}$ are sufficiently small depending on $n,p,\delta_0, \sigma$, there exists a constant $c_n>0$ such that  
\begin{equation}
\label{eq:nondeg_a}
u(x)\geq c_n \dist(x,\Gamma)^{\frac{1}{2}+\frac{\bar\epsilon}{2}},\quad \bar \epsilon=\min\{\delta_0,\frac{1}{2}-\frac{n+1}{p}+\sigma\}
\end{equation}
for all $x\in  (B_{1/2}'\times (-1/2,1/2)) \cap \{x| \dist(x,\Lambda)\geq \ell_0 \dist(x,\Gamma)\}$, and 
\begin{equation}
\label{eq:nondeg_b}
u(x)\geq - c_n\dist(x,\Gamma)^{1-\frac{n+1}{p}+\sigma} 
\end{equation}
for all $x\in B_{1/2}'\times (-1/2,1/2)$. 
\end{prop}

\begin{rmk}
By a scaling argument, Proposition \ref{prop:nondeg} remains true, if condition (iii) is replaced by $u\geq \bar c_0$ on $B'_1\times [-1,1]\cap \{|x_{n+1}|\geq \ell_0\}$ for some $\bar c_0>0$ (and if the remaining constants in Proposition \ref{prop:nondeg} are rescaled appropriately).
\end{rmk}

\begin{rmk}
The parameter $\sigma>\frac{n+1}{p}-\frac{1}{2}$ quantifies the necessary decay conditions on the inhomogeneities $F^i$, $i\in\{1,\dots,n+1\},$ $g_1$, $g_2$ in dependence on the integrability of these quantities (measured in terms of $p$). In the application to the thin obstacle problem, the corresponding inhomogeneities satisfy these decay assumptions, if either
\begin{itemize}
\item $p>n+1$ and no obstacle is present,
\item or $p>2(n+1)$ in which case an arbitrary $W^{2,p}$ obstacle is allowed.
\end{itemize}
More precisely, in the setting of the thin obstacle problem, we apply Proposition~\ref{prop:nondeg} to the tangential derivatives $\p_e w$, which satisfy the divergence form equation
\begin{align*}
\p_ia^{ij}\p_j u = \p_i F^i, \quad \text{where } F^i=(\p_ea^{ij})\p_j w+\p_e(a^{ij}\p_j\phi),
\end{align*} 
with metric $a^{ij}\in W^{1,p}$ and obstacle $\phi \in W^{2,p}$. We note that on the one hand by Lemma \lemfreebgrowth, the first term in $F^i$ satisfies
\begin{align*}
 \dist(x,\Gamma_w)^{-1/2}\ln (\dist(x,\Gamma_w))^{-2} (\p_ea^{ij})\p_j w \in L^p.
\end{align*}
On the other hand, the second term $\p_e(a^{ij}\p_j\phi)\in L^p$, which originates from a potentially non-vanishing obstacle (c.f. Section \ref{sec:nonfl}), does \emph{not} carry any additional decay. However, 
\begin{itemize}
\item if $p\in (n+1,2(n+1)]$, then $\sigma>\frac{n+1}{p}-\frac{1}{2}\geq 0$, i.e. decay on $F^i$ is required. The assumption (i) from Proposition \ref{prop:nondeg} on $F^i$ is thus in general only satisfied if we set $\phi=0$. 
\item If $p\in (2(n+1),\infty]$, we have $\frac{n+1}{p}-\frac{1}{2}<0$. Hence, we do not need any additional decay on $F^i$. In this case Proposition~\ref{prop:nondeg} applies for any $\phi\in W^{2,p}$. 
\end{itemize}
\end{rmk}

\begin{proof}
In order to prove the result, we consider an appropriate comparison function. For $x_0\in B_{1/2}'\times (-\ell_0, \ell_0)$, we set $s=\bar\epsilon$ (where $\bar{\epsilon}$ is the constant from (\ref{eq:nondeg_a})) and consider
\begin{align*}
h(x):= u(x) + |x'-(x_0)'|^2 - 2^{-8}h_s^-(x) - 2 n x_{n+1}^2.
\end{align*}
A direct computation shows that $h$ satisfies
\begin{align*}
\p_i a^{ij} \p_j h =  \tilde{g}_1+\tilde{g}_2,
\end{align*}
where 
\begin{align*}
\tilde{g}_1&= \p_iF^i+g_1 + 2(x_j-(x_0)_j)(\p_ia^{ij}) - 4n(x_{n+1})( \p_{j} a^{n+1,j}) ,\\
\tilde{g}_2&= g_2 -2^{-8}\p_ia^{ij}\p_j h_s^-+2 a^{ii} - 4 n a^{n+1,n+1}.
\end{align*}
We split $h=h_1+h_2$, where
\begin{align*}
\p_i a^{ij}\p_j h_1 -K \dist(x,\Gamma)^{-2} h_1&=\tilde{g}_1 \text{ in } \R^{n+1}\setminus \Lambda,\\
h_1&=0 \text{ on } \Lambda.
\end{align*}
By Proposition~\ref{prop:v1}, there exists $C_0=C_0(n,p)$ such that 
\begin{align}
\label{eq:estimate_for_h1}
|h_1(x)|\leq 
C_0 (c_{\ast}+ \bar{\delta} )\dist(x,\Gamma)^{1-\frac{n+1}{p}+\sigma}.
\end{align}

Now we consider $h_2$. Recalling the definition of $\ell_0$, we have that  
\begin{align*}
h_2&\geq 1 -C_0 (c_{\ast} + \bar{\delta})  - 2^{-8} - 2n\ell_0^{2} \geq \frac{1}{2} \mbox{ on } B_{\frac{3}{4}}'\times \{\pm \ell_0\},\\
h_2&\geq -2^{-8} +\frac{1}{16 } -C_0 (c_{\ast} + \bar{\delta}) - 2^{-8} - 2n\ell_0^{2} \geq 2^{-8} \mbox{ on } \p B'_{\frac{3}{4}}\times (-\ell_0,\ell_0).
\end{align*}
We further recall that $u=h_1=h_s^-=0$ on $\Lambda$. Thus, $h_2 \geq  0$ on $\Lambda$.\\
If $\bar{\delta}$ and $c_{\ast}$ are chosen sufficiently small, $h_2$ satisfies
\begin{align*}
\p_ia^{ij}\p_jh_2&=\tilde{g}_2-K\dist(x,\Gamma)^{-2}h_1\\
&\leq (\bar{\delta}+ c_{\ast}) C_0K\dist(x,\Gamma)^{-1-\frac{n+1}{p}+\sigma}- 2^{-8}c_s\dist(x,\Gamma)^{- \frac{3}{2}+ \frac{s}{2}} \\
& \quad + 2a^{ii} -4 n a^{n+1,n+1} + \bar{\delta} \dist(x,\Gamma)^{-\frac{3}{2}+\delta_0} \\
&\leq 0 \quad \mbox{ in } (B'_{\frac{3}{4 }}\times (-\ell_0,\ell_0)) \setminus \Lambda,
\end{align*}
where we used the smallness of $\bar{\delta}$ and the definition of $s$ in terms of $\bar{\epsilon}$.
Hence, by the comparison principle, $h_2\geq 0$ in $B'_{\frac{3}{4 }}\times (-\ell_0,\ell_0)$. 
Combining this with the estimate \eqref{eq:estimate_for_h1} gives 
$$h(x)\geq -C_0 (c_{\ast} + \bar{\delta})\dist(x,\Gamma)^{1-\frac{n+1}{p}+\sigma}\text{ in }B'_{\frac{3}{4 }}\times (-\ell_0,\ell_0).$$  
In particular, 
$$h(x_0)\geq -C_0 (c_{\ast}+ \bar{\delta}) \dist(x_0,\Gamma)^{1-\frac{n+1}{p}+\sigma}.$$ 
Rewriting this in terms of $u_2$, yields 
\begin{align*}
u(x_0) &\geq - C_0 (c_{\ast} + \bar{\delta}) \dist(x_0,\Gamma)^{1-\frac{n+1}{p}+\sigma} + 2^{-8}h_s^{-}(x_0).
\end{align*}
Using the growth properties of the barrier function $h_s^-$ (c.f. Proposition \ref{prop:barrier}) and by choosing $c_{\ast}, \bar{\delta}$ sufficiently small in dependence of $n,p$ and $\delta_0$, implies the estimates (\ref{eq:nondeg_a}) and (\ref{eq:nondeg_b}) in $B_{1/2}'\times (-\ell_0, \ell_0)$.
Finally, the full non-degeneracy in $B_{1/2}'\times (-1/2,1/2) \cap \{x| \dist(x,\Lambda)\geq \ell_0 \dist(x,\Gamma)\}$ follows from our assumption (ii) in combination with the fact that $0\in \Gamma$.
\end{proof}

\subsubsection{Proof of Proposition \ref{prop:Lip}}
\label{sec:proofLip}

With the aid of Proposition~\ref{prop:nondeg} we can now prove the Lipschitz (and $C^1$) regularity of the regular set of the free boundary. We only sketch the proof as there are no real modifications with respect to the one in \cite{PSU}:

\begin{proof}[Proof of Proposition \ref{prop:Lip}]
By choosing $\epsilon_0$ and $c_{\ast}$ sufficiently small and by invoking Proposition \ref{prop:Reifenberg}, we may assume that $\Gamma_{w}\cap B_{1/2}'$ is $(\epsilon,1/2)$ Reifenberg flat. Using this and the $C^1(B_1^+)$ closeness of $w$ to $w_{3/2}$, we apply Proposition \ref{prop:nondeg}.
In order to satisfy its assumptions, we transfer positivity from $\p_e w_{3/2}$ to $\p_e w$, where $e\in S^{n-1}\times \{0\}$ is a tangential direction. We consider $\ell_0= (2^{4}\sqrt{n})^{-1}$ as in Proposition \ref{prop:nondeg}. Then for any $\eta>0$ and any tangential vector $e\in \{x \in S^n \cap B_1' \big| \ x_{n} > \eta |x''|\}$, the function $w_{3/2}$ satisfies: 
\begin{align*}
\p_e w_{3/2}(x) &\geq 0 \mbox{ if } x\in B_{1},\\
\p_e w_{3/2}(x',x_{n+1}) &\geq c_n \eta (1+\eta^2)^{-1/2}\ell_0^{1/2} >0 \mbox{ if } x \in B_{1} \cap \{|x_{n+1}|\geq \ell_0\}.
\end{align*}
Moreover, $\p_e w$ solves
\begin{align*}
\p_i (a^{ij}  \p_j \p_ew ) =  \p_i F^{i}\mbox{ in } B_{1}\setminus \Lambda_{w},\quad F^i= -(\p_e a^{ij} )\p_j w.
\end{align*}
By Lemma \lemfreebgrowth,  $\|\dist(\cdot,\Gamma_w)^{-\frac{1}{2}}\ln(\dist(x,\Gamma_w))^{-2}F^i\|_{L^p(B_1)}\leq C(n,p)c_\ast$.
Hence, by Proposition~\ref{prop:nondeg} (applied with $\sigma = \frac{1}{2}-\mu$, where $\mu$ is an arbitrarily small but fixed constant), if $\epsilon_0$ and $c_\ast$ are chosen sufficiently small depending on $n, p, \mu$, then there exists a positive constant $c_{n,p}$ such that 
\begin{equation}
\label{eq:positive_cone}
\p_{e}w(x)\geq c_{n,p}\dist(x,\Gamma_{w})^{\frac{1}{2}+ \frac{\bar\epsilon}{2}}\geq 0
\end{equation}
for all $x\in B_{1/2}\cap \{x| \dist(x,\Lambda_{w}) \geq \ell_0\dist(x,\Gamma_{w})  \} $ and $e\in \{x\in S^n\cap B'_1| x_n>\eta|x''|\}$. Here $\eta\geq \tilde{C}\max\{\epsilon_0, c_\ast\}$ for a large enough positive constant $ \tilde{C}=\tilde{C}(n,p)$ and $\bar{\epsilon}= \sigma+\frac{1}{2}-\frac{n+1}{p}$. Moreover, Proposition \ref{prop:nondeg} also implies the global lower bound $\p_e w(x) \geq -\dist(x,\Gamma_w)^{1-\frac{n+1}{p}+\sigma}$ in $B_{1/2}$.
As $\p_e w=0$ on $\Lambda_w$ and as $\Omega_w$ is covered by the set $\{x| \dist(x,\Lambda_{w}) \geq \ell_0\dist(x,\Gamma_{w})  \} $ (on which the bound (\ref{eq:positive_cone}) holds), we deduce that $\p_e w \geq 0$ in $B'_{ 1/2}$, with strict positivity in $B'_{ 1/2}\cap \{x| \dist(x,\Lambda_{w}) > \ell_0 \dist(x,\Gamma_{w})  \} $.
As a consequence (c.f. \cite{PSU}),
\begin{align*}
\{w(x',0)>0\}\cap B_{1/2}' = \{x_{n}>g(x'') \big| \ x'' \in B''_{1/2}\},
\end{align*}
where $g$ is a Lipschitz continuous function with $\|\nabla'' g\|_{L^\infty(B''_{1/2})}\leq \tilde{C}\max\{\epsilon_0,c_\ast\}$.
\end{proof}

\begin{rmk}
\label{rmk:diff}
If we return to an arbitrary solution $w$ of  (\ref{eq:varcoef}), we can even conclude the $C^1$ regularity of the free boundary from the previously given argument. Indeed, after rescaling, it is always possible to satisfy the closeness assumption (i) in Proposition \ref{prop:Lip} with an arbitrarily small parameter $\epsilon_0$.
Hence, we can infer that for each $\eta>0$ there exists $\rho_{\eta}>0$ such that $\left\| \nabla'' g\right\|_{L^{\infty}(B_{\rho_{\eta}}'')}\leq \eta$. This yields that $g$ is differentiable at $0$, or in other words, the tangent plane of $\Gamma_{3/2}(w)$ exists at the origin. The same holds true at other regular free boundary points. Moreover, from \eqref{eq:positive_cone} it is not hard to see that the tangent plane varies continuously in a neighborhood of the origin.  Hence, we indeed obtain that $\Gamma_{3/2}(w)\cap B_\rho$ is not only Lipschitz but $C^1$ regular.
\end{rmk}

\begin{rmk}
At this stage and with these (comparison) techniques we do not know how to further estimate the modulus of continuity. Heuristically, the improvement comes from the fact that the convergence to the homogeneous solution is much faster than the compactness arguments suggest, once the solution is close to the homogeneous solution in the sense of Proposition \ref{prop:Reifenberg}. Thus, in the next section we use a boundary Harnack inequality to overcome this difficulty.
\end{rmk}

\subsection{$C^{1,\alpha}$ regularity}
\label{subsec:C1a}

In this section we improve the regularity of the (regular) free boundary to obtain its $C^{1,\alpha}$ regularity for some $\alpha \in (0,1]$:

\begin{prop}
\label{prop:C1a}
Let $a^{ij}$ be a uniformly elliptic $W^{1,p}(B_1^+)$, $p\in(n+1,\infty]$, tensor field with $a^{ij}(0)=\delta^{ij}$.
Assume that the conditions (i) and (ii) of Proposition \ref{prop:Lip} are satisfied with the same constants $\epsilon_0$ and $c_\ast$. Then if $\epsilon_0, c_{\ast}$ are sufficiently small depending on $n,p$, there exists $\alpha = \alpha(n,p) \in(0,1-\frac{n+1}{p}]$ such that $g\in C^{1,\alpha}(B_{1/4}'')$. Moreover, $\|\nabla ''g\|_{C^{0,\alpha}(B''_{1/4})}\leq C(n,p)\max\{\epsilon_0,c_\ast\}$.
\end{prop}

We remark that Theorem \ref{thm:C1a} is an immediate consequence of Proposition \ref{prop:C1a}:

\begin{proof}[Proof that Proposition~\ref{prop:C1a} implies Theorem \ref{thm:C1a}]
As a result of the analysis in \cite{KRS14}, we may assume that for an appropriate sequence of radii $r_j>0$, $r_j \rightarrow 0$, the ($L^2$)-rescalings, $w_{r_k}(x):= \frac{w(r_k x)}{ r_{k}^{- \frac{n+1}{2}} \left\| w \right\|_{L^2(B_{r_k}^+)}  }$, converge to a global solution $w_0$ in $C^{1,\alpha}_{loc}(\R^{n+1}_+)$ for some $\alpha \in (0,1/2)$, where after a possible rotation of coordinates
\begin{align*}
w_0(x)=w_{3/2}(x)=c_n \Ree(x_{n}+ ix_{n+1})^{\frac{3}{2}}.
\end{align*}
For each $k$, $w_{r_{k}}$ is a solution of \eqref{eq:varcoef} in $B_{r_k^{-1}}$ with metric $a^{ij}_{r_k}$, where $a^{ij}_{r_k}(x)=a^{ij}(r_k x)$, with rescaled free boundary $\Gamma_{w_{r_k}}$.    
A direct calculation shows that $$\|\nabla a^{ij}_{r_k}\|_{L^p(B_1)}=r_k^{1-\frac{n+1}{p}}\|\nabla a^{ij}\|_{L^p(B_{r_k})}.$$
Thus, for $r_{k_0}$ small enough, assumptions (i), (ii) in Proposition~\ref{prop:Lip} are satisfied for $w_{r_{k_0}}$. Thus, Proposition~\ref{prop:Lip} implies that $\Gamma_{w_{r_{k_0}}}\cap B_{1/2}$ is a Lipschitz graph with sufficiently small Lipschitz constant. Hence, Proposition \ref{prop:C1a} is applicable to $w_{r_{k_0}}$. Rescaling back to $w$ yields Theorem~\ref{thm:C1a}.
\end{proof}

The improvement from Lipschitz to $C^{1,\alpha}$ regularity of the regular free boundary is achieved with the aid of a boundary Harnack inequality (c.f. Theorem 7 in \cite{ACS08}) applied to tangential derivatives $v=\p_e w$ in the slit domain $B_1\setminus \Lambda_w$. Here we follow the ideas of \cite{CSS} by making use of the non-degeneracy condition of the solution. In this context, we have to be slightly more careful, as our inhomogeneity is in divergence form. Consequently, we rely on the decomposition, $v=v_1+v_2$, which we introduced above (c.f. Proposition \ref{prop:v1}).
Another difficulty, which we have to overcome, is the fact that we may not assume that $v$ is positive in the whole domain $B_1\setminus \Lambda_w$: On the one hand it satisfies (by Proposition~\ref{prop:nondeg})
\begin{equation*}
v(x)\geq c_0 \dist(x,\Gamma_w)^{\frac{1}{2}+\frac{\bar\epsilon}{2}} \text{ in } \{x\in B_1| \ \dist(x,\Lambda_w)\geq \ell_0 \dist(x,\Gamma_w)\},
\end{equation*}
for $\ell_0 = (2^{4}\sqrt{n})^{-1}$ and $\bar{\epsilon}=\min\{\delta_0,\sigma+\frac{1}{2}-\frac{n+1}{p}\}$,
while on the other hand in the complement we only have 
\begin{equation*}
v(x)\geq - c_1\dist(x,\Gamma_w)^{1-\frac{n+1}{p}+\sigma} \text{ in } B_1.
\end{equation*}
Here $\sigma>\frac{n+1}{p}-\frac{1}{2}$ is the same constant as in Proposition \ref{prop:nondeg} (i).\\

Despite these additional difficulties, it is possible to adapt the general strategy of \cite{CSS}: Working in the framework of general elliptic equations with controlled inhomogeneities in domains with Lipschitz slits, in Section \ref{sec:Carl_Harn} we first prove a Carleson estimate (Lemma \ref{lem:Carleson}) which is adapted to our setting. Then we continue with an appropriate boundary Harnack inequality (Lemma \ref{lem:boundHarn}). In Section \ref{sec:proofC1a} we finally apply these results in the setting of the thin obstacle problem and hence infer the desired regularity result of Proposition \ref{prop:C1a}. 

\subsubsection{Carleson and boundary Harnack estimates}
\label{sec:Carl_Harn}

In the sequel, we work in the slightly more general set-up of elliptic equations with controlled inhomogeneities in domains with Lipschitz slits. Hence, in this section we make the following assumptions:

\begin{assum} 
\label{assum:set}
Let $g: \R^{n-1}\rightarrow \R$ be a Lipschitz function with Lipschitz constant $\epsilon>0$ and $g(0)=0$.
Then we consider
\begin{align*}
\Lambda &:= \{x_{n}\leq g(x'')\} \subset \R^{n}\times \{0\},\\
\Gamma &:=\{x_n=g(x'')\}\subset \R^n\times \{0\}.
\end{align*} 
\end{assum}

We remark that the Lipschitz regularity of $\Gamma$ in particular implies its $(\epsilon,1)$-Reifenberg flatness. 
Moreover, in our set-up, the smallness assumption on the Lipschitz constant of $g$ does not pose restrictions and can always be achieved by choosing the constants $\epsilon_0, c_\ast$ sufficiently small in Proposition \ref{prop:Lip}.\\

As in the previous Section \ref{sec:boundary}, we use the abbreviations from Notation \ref{not:not1}, i.e. we again set $\ell_0 := (2^{4} \sqrt{n})^{-1}$ and $c_{\ast}:=\left\| \nabla a^{ij} \right\|_{L^p(B_1)}$ for notational convenience.\\

In this framework we derive the following Carleson estimate:

\begin{lem}[Carleson estimate]
\label{lem:Carleson} 
Let $\Lambda, \Gamma, \epsilon$ satisfy Assumption \ref{assum:set}.
Suppose that $u\in C(B_1)\cap H^{1}(B_1 \setminus \Lambda)$ solves
\begin{align*}
\p_i a^{ij} \p_j u &= \p_i F^i+g_1+g_2 \mbox{ in } B_1\setminus \Lambda, \quad u=0 \text{ on } \Lambda,
\end{align*} 
where we assume that condition (i) from Proposition \ref{prop:nondeg} and the following non-degeneracy assumption holds: For $\sigma, \bar{\epsilon}$ as in Proposition \ref{prop:nondeg} we have
\begin{equation}
\label{eq:nondeg}
\begin{split}
u(x)&\geq  \dist(x,\Gamma)^{\frac{1}{2}+\frac{\bar{\epsilon}}{2}} \text{ in } \{x\in B_1| \ \dist(x,\Lambda)\geq \ell_0 \dist(x,\Gamma)\},\\
u(x)&\geq - \dist(x,\Gamma)^{1-\frac{n+1}{p}+\sigma} \text{ in } B_1.
\end{split}
\end{equation}
Then if $\bar\delta=\bar\delta(n,p,\delta_0)$ and $\epsilon=\epsilon(n,p,\delta_0)$ are small enough, there exists $C=C(n,p,\epsilon)$ such that 
\begin{align*}
\sup\limits_{B_{r}(Q)}u \leq C u(Q+ r e_{n}),
\end{align*}
for all $Q\in \Gamma \cap B'_{1/2}$ and $r\in (0,1/4)$.
\end{lem}

\begin{proof}
Let $Q\in \Gamma\cap B'_{1/2}$ and $r\in (0,1/4)$. Let $u^*$ be the solution of 
\begin{align*}
\p_i a^{ij} \p_j u^* &= 0 \mbox{ in } B_{2 r}(Q)\setminus \Lambda,\\
u^* &= \max\{u,0\} \mbox{ on } \p (B_{2 r}(Q)\setminus \Lambda).
\end{align*}
Then on $B_{r}(Q) \setminus \Lambda$ Theorem 7 of \cite{ACS08} implies that there exists $C=C(n,\epsilon)$ such that
\begin{align*}
u^*(x) \leq Cu^*(Q+ r e_{n}).
\end{align*}
We transfer this estimate to $u$ by exploiting a comparison argument which makes use of the non-degeneracy of $u$. We consider
\begin{align*}
h(x):=  u^*(x) -u(x)-h^-_s(x) +8r^{\frac{1}{2}+\frac{s}{2}}, 
\end{align*}
where $h^-_s$ is the barrier function from Proposition \ref{prop:barrier} and $s=\bar \epsilon$. 
For $h$ we obtain
\begin{align*}
\p_i a^{ij}\p_j h&=-\p_iF^i -g_1-g_2-\p_ia^{ij}\p_j h^-_s\quad  \text{ in }  B_{2 r}(Q)\setminus \Lambda,\\
h&\geq 4 r^{(1+s)/2}\quad \mbox{ on } \p ( B_{2 r}(Q)\setminus \Lambda),
\end{align*}
where in the second inequality we have used $|u^\ast-u|=\max\{-u,0\}\leq r^{1-\frac{n+1}{p}+\sigma}$ on $\p (B_{2 r}(Q)\setminus \Lambda)$.
In order to estimate $u$, we use the splitting technique again. With a slight abuse of notation we write $\Lambda=\Lambda\cap B_{2r}(Q)$ for simplicity. We decompose $h=h_1+h_2$, where $h_1$ solves
\begin{align*}
\p_ia^{ij}\p_j h_1-K\dist(x,\Gamma)^{-2}h_1&=-\p_iF^i-g_1 \text{ in } B_1\setminus \Lambda,\\
h_1&=0 \text{ on } \Lambda.
\end{align*}
Then by Proposition~\ref{prop:v1}, 
$$|h_1(x)|\leq C(n,p) \bar \delta \dist(x,\Gamma)^{1-\frac{n+1}{p}+\sigma}.$$
For $\bar \delta=\bar \delta(n,p)$ sufficiently small, $h_2\geq 0$ on $\p (B_{2 r}(Q)\setminus \Lambda )$. In $B_{2 r}(Q)\setminus \Lambda$, for $\bar\delta=\bar\delta(n,p,\delta_0)$ sufficiently small,
\begin{align*}
\p_ia^{ij}\p_j h_2&=-K\dist(x,\Gamma)^{-2}h_1-g_2-\p_ia^{ij}\p_j h^-_s\\
&\leq KC(n,p) \bar \delta \dist(x,\Gamma)^{-1-\frac{n+1}{p}+\sigma}+\bar\delta \dist(x,\Gamma)^{-\frac{3}{2}+\delta_0}\\
& \quad -c(\bar \epsilon)\dist(x,\Gamma)^{-\frac{3}{2}+\bar \epsilon}\\
&\leq 0.
\end{align*}
Thus, by the comparison principle we have $h_2\geq 0$ in $B_{2r}(Q)$.
Hence, we obtain 
$$h(x)\geq -  r^{1-\frac{n+1}{p}+\sigma}\text{ for any } x\in B_{2r}(Q)\setminus\Lambda.$$
Combining this with the choice of $s$ yields 
\begin{equation*}
u^*-u\geq -Cr^{\frac{1}{2}+\frac{\bar\epsilon}{2}} \text{ in }  B_{2r}(Q)\setminus\Lambda,
\end{equation*}
for some absolute constant $C$.
Similarly, we deduce $u-u^*\geq -Cr^{\frac{1}{2}+\frac{\bar\epsilon}{2}}$ in $B_{2r}(Q)\setminus\Lambda$. 
As a result, for $x\in B_{r}(Q)$, 
\begin{align*}
u(x) &\leq u^*(x) + Cr^{\frac{1}{2}+\frac{\bar\epsilon}{2}}\\
&\leq C u^*(Q+r e_{n}) +  Cr^{\frac{1}{2}+\frac{\bar\epsilon}{2}}\\
&\leq C u(Q+r e_{n}) + Cr^{\frac{1}{2}+\frac{\bar\epsilon}{2}}+ Cr^{\frac{1}{2}+\frac{\bar\epsilon}{2}}\\
& \leq 3 C u(Q+r e_{n}), 
\end{align*}
where in the last line we used the non-degeneracy of $u$.
\end{proof}

With this preparation, we come to the boundary Harnack inequality. Let $\Lambda$ and $\Gamma$ be as in Assumption \ref{assum:set}. We now in addition choose $\theta=\theta(n,\epsilon)\in (0,2\pi)$ such that for any $x_0\in \Gamma$ 
$$x_0+\mathcal{C}_\theta'(e_n)\subset \{ x\in \R^{n}\times \{0\}| \ \dist(x,\Lambda)\geq \ell_0 \dist(x,\Gamma)\}. $$
Here $\mathcal{C}_{\theta}'(e_n)$ is a cone in $\R^{n}\times \{0\}$ with opening angle $\theta$ and $\ell_0$ is as in Notation \ref{not:not1}. 

\begin{lem}[Boundary Harnack]
\label{lem:boundHarn}
Let $\Lambda, \Gamma, \epsilon$ satisfy Assumption \ref{assum:set}.
Suppose that $u_1,u_2\in C(B_1)\cap H^{1}(B_1 \setminus \Lambda)$ solve
\begin{align*}
\p_i a^{ij} \p_j u_1 &= \p_i F^i +g_1+g_2 \mbox{ in }  B_{1}\setminus \Lambda,\quad  u_1=0 \text{ on } \Lambda,\\
\p_i a^{ij}\p_j u_2&=\p_i \tilde{F}^i+ \tilde{g}_1+\tilde{g}_2 \mbox{ in }  B_{1}\setminus \Lambda, \quad  u_2=0 \text{ on } \Lambda,
\end{align*}
where
\begin{itemize}
\item the coefficients $a^{ij}$ and inhomogeneities satisfy the condition (i)  from Proposition \ref{prop:nondeg} with the same constants $\sigma$ and $\bar\delta$,
\item both functions $u_1, u_2$ satisfy the non-degeneracy condition (\ref{eq:nondeg}) of Lemma~\ref{lem:Carleson} with the same constant $\bar \epsilon$. 
\end{itemize}
Then, if $\epsilon$, $\bar\delta$ and $c_{\ast}$ are sufficiently small depending on $n,p$ and $\delta_0$, there exists a constant $s_0=s_0(n,p, \bar\epsilon, \sigma)>0$, such that 
\begin{align*}
s_0\frac{u_2(\frac{1}{2}e_{n})}{u_1(\frac{1}{2} e_{n})}\leq \frac{u_2(x)}{u_1(x)} \leq s_0^{-1}\frac{u_2(\frac{1}{2}e_{n})}{u_1(\frac{1}{2} e_{n})} \text{ in } \bigcup_{x_0\in B_{1/2}\cap \Gamma} \left(x_0+\mathcal{C}'_\theta(e_n)\right)\cap B_{1/2}.
\end{align*}
Moreover, $u_2/u_1$ extends to a function which is Hölder continuous on $\Gamma\cap B'_{1/4}$. More precisely, there exist constants $\alpha \in (0,1]$ and $C>0$  depending on $n,p$ and $\bar\epsilon$, $\sigma$
such that for all $x_0\in \Gamma\cap B'_{1/4}$ 
\begin{align*}
\left|\frac{u_2}{u_1}(x) - \frac{u_2}{u_1}(x_0) \right| \leq C \frac{u_2(\frac{1}{2}e_n)}{u_1(\frac{1}{2}e_n)} |x-x_0|^{\alpha},\quad x \in B_{1/4}(x_0)\cap \left(x_0+\mathcal{C}'_\theta(e_n)\right).
\end{align*}
\end{lem}

\begin{proof}
We only prove the Hölder continuity at $x_0=0$. The estimate at the remaining points follows analogously. \\

\emph{Step 1: Boundedness of the quotient.} Without loss of generality we may assume that $u_1(\frac{e_{n}}{2}) = u_2(\frac{e_{n}}{2})=1$.
Then the Carleson estimate (Lemma~\ref{lem:Carleson}) and the non-degeneracy assumptions on $u_1,u_2$ imply that for all $\delta >0$ 
\begin{align*}
u_1(x) &\leq C \mbox{ in } B_{1},\\
u_2(x) & \geq \ell_0^{\frac{1}{2}+\frac{\bar{\epsilon}}{2}} \mbox{ for all } x \mbox{ with } |x_{n+1}|\geq \ell_0. 
\end{align*}
Therefore, Proposition \ref{prop:nondeg} implies that there exists $s_0\in (0,1)$ (depending only on the admissible quantities) such that
$$u_2-s_0 u_1 \geq 0 \mbox{ in } \{x\in B_{\frac{1}{2}}| \ \dist(x,\Lambda)\geq\ell_0 \dist(x,\Gamma)\}.$$
Hence, in particular,
\begin{align*}
\frac{u_2(x)}{u_1(x)} \geq s_0 \mbox{ for all } x\in \{x\in B_{\frac{1}{2}}| \ \dist(x,\Lambda)\geq\ell_0 \dist(x,\Gamma)\}.
\end{align*}
Exchanging the role of $u_1$ and $u_2$, we have
$$ u_1 -s_0u_2 \geq 0 \mbox{ in }\{x\in B_{\frac{1}{2}}| \ \dist(x,\Lambda)\geq\ell_0 \dist(x,\Gamma)\}.$$ 
\emph{Step 2: Iteration argument.} 
We only prove the H\"older regularity at $x_0=0$.\\
We show that there exist $a_k, b_k$ depending on $s_0$ and $n,p,\bar{\epsilon}$ such that
\begin{itemize}
\item[($\alpha$)] $s_0\leq a_k \leq 1\leq b_k \leq s_0^{-1}$ and $b_k - a_k \leq C \mu_1^k$ for some $\mu_1 \in (2^{-\frac{\bar{\epsilon}}{2}},1)$, 
\item[($\beta$)] there exists $\mu_2 \in [2^{-\frac{\bar{\epsilon}}{2}},\mu_1]$ with $b_{k}-a_{k}\geq C\mu_2^{k}$ for some constant $C\geq 4$,
\item[($\gamma$)] on $B_{2^{-k}}\cap \mathcal{C}'_\theta(e_n)$ we have
$ a_k u_1\leq u_2 \leq b_k u_1 $.
\end{itemize}
We argue by induction and define 
\begin{align*}
\tilde{w}_1(x) := \frac{u_2(2^{-k}x)- a_k u_1(2^{-k}x)}{b_k -a_k}, \ \tilde{w}_2(x):= \frac{b_k u_1(2^{-k}x) - u_2(2^{-k}x) }{b_k-a_k}.
\end{align*}
As $\tilde{w}_1(x) + \tilde{w}_2(x) = u_1(2^{-k}x)$, we may without loss of generality assume that $\tilde{w}_1(\frac{e_{n}}{2})\geq \frac{1}{2} u_1(\frac{2^{-k}e_{n}}{2})$. Dividing by $u_1(\frac{2^{-k}e_{n}}{2})$ and setting 
\begin{align*}
w_1(x) := \frac{\tilde{w}_{1}(x)}{ u_1(\frac{2^{-k}e_{n}}{2})}, \ \bar{u}(x):=\frac{u_1(2^{-k}x)}{ u_1(\frac{2^{-k}e_{n}}{2})},
\end{align*}
we obtain that 
\begin{align*}
2w_1(\frac{e_{n}}{2}) \geq  \bar{u}(\frac{e_{n}}{2}) = 1.
\end{align*}
Seeking to apply step 1, we have to show that $2 w_1$ and $\bar{u}$ satisfy its assumptions, i.e. the bounds on the associated inhomogeneity and the non-degeneracy. For this we observe:
\begin{itemize}
\item The inhomogeneity associated with $2w_1$ satisfies the bounds from Lemma \ref{lem:Carleson}. Indeed, $w_1$ solves
\begin{align*}
\p_i a^{ij}_{2^{-k}}\p_j w_1  &= \p_iH^i+G_1+G_2 \text{ in } B_1\setminus \Lambda_{2^{-k}},
\end{align*}
where
\begin{align*}
H^i(x)&=\frac{2^{-k}}{(b_k -a_k)  u_1(\frac{2^{-k}e_{n}}{2})}\left[\tilde{F}^i(2^{-k}x)-a_kF^i(2^{-k}x)\right],\\
G_1(x)&=\frac{2^{-2k}}{(b_k -a_k)  u_1(\frac{2^{-k}e_{n}}{2})}\left[\tilde{g}_1(2^{-k}x)-a_kg_1(2^{-k}x)\right],\\
G_2(x)&= \frac{2^{-2k}}{(b_k -a_k)  u_1(\frac{2^{-k}e_{n}}{2})}\left[\tilde{g}_2(2^{-k}x)-a_kg_2(2^{-k}x)\right],\\
a^{ij}_{2^{-k}}(x)&=a^{ij}(2^{-k}x),\\
\Lambda_{2^{-k}}&= \frac{\Lambda}{2^{-k}}=\{x\in B_1| \ 2^{-k} x \in \Lambda\},\\
\Gamma_{2^{-k}}&= \frac{\Gamma}{2^{-k}}=\{x\in B_1| \ 2^{-k} x \in \Gamma\}.
\end{align*}
Since $\dist(x,\Gamma_{2^{-k}})=2^k \dist(2^{-k}x,\Gamma)$, 
we estimate
\begin{equation}
\begin{split}
\label{eq:inhom}
&\left\| \dist(\cdot,\Gamma_{2^{-k}})^{-\sigma}H^i\right\|_{L^{p}(B_{1})}+\left\| \dist(\cdot,\Gamma_{2^{-k}})^{1-\frac{n+1}{p}-\sigma} G_1\right\|_{L^{p/2}(B_{1})}+\left\|\dist(\cdot,\Gamma_{2^{-k}})^{\frac{3}{2}-\delta_0}G_2\right\|_{L^\infty(B_1)}\\
&\leq \frac{2^{-k(1-\frac{n+1}{p}+\sigma)}}{(b_k -a_k)u_1(\frac{2^{-k}e_{n}}{2})}\left(\|\dist(\cdot,\Gamma)^{-\sigma}\tilde{F}^i\|_{L^p(B_{2^{-k}})} + a_k\|\dist(\cdot,\Gamma)^{-\sigma}F^i\|_{L^p(B_{2^{-k}})} \right. \\
& \quad \left.+\| \dist(\cdot,\Gamma)^{1-\frac{n+1}{p}-\sigma} \tilde{g}_1\|_{L^{p/2}(B_{2^{-k}})} + a_k\| \dist(\cdot,\Gamma)^{1-\frac{n+1}{p}-\sigma } g_1\|_{L^{p/2}(B_{2^{-k}})}\right)\\
& \quad +\frac{2^{-k(\frac{1}{2}+\delta_0)}}{{(b_k -a_k)u_1(\frac{2^{-k}e_{n}}{2})}}\left(\|\dist(\cdot,\Gamma)^{\frac{3}{2}-\delta_0} \tilde{g}_2\|_{L^{\infty}(B_{2^{-k}})}+a_k\| \dist(\cdot,\Gamma)^{\frac{3}{2}-\delta_0 }g_2\|_{L^{\infty}(B_{2^{-k}})}\right) \\
&\leq \frac{1}{(b_k -a_k)u_1(\frac{2^{-k}e_{n}}{2}) }2\bar \delta\cdot 2^{-k(\frac{1}{2}+\bar \epsilon)} \leq   \bar \delta.
\end{split}
\end{equation}
Here we used the non-degeneracy assumption \eqref{eq:nondeg} in the form $u_1(\frac{2^{-k}e_{n}}{2})\geq  2^{-(k+1)(\frac{1}{2}+\frac{\bar{\epsilon}}{2})}$ (for which we recall that $\bar\epsilon=\min\{\sigma+\frac{1}{2}-\frac{n+1}{p}, \delta_0\}$) as well as the condition ($\beta$), i.e. $\mu_2\geq 2^{-\frac{\bar{\epsilon}}{2}}$.
The inhomogeneity associated with $\bar{u}$ can be estimated similarly.
\item $2w_1$ and $\bar{u}$ satisfy the non-degeneracy assumption. As $2w_1(e_{n}/2)\geq 1$ and $\bar{u}(e_{n}/2)=1$, Harnack's inequality in the domain $B_1\cap \{x| \ \dist(x,\Lambda_{2^{-k}})\geq\ell_0\dist(x,\Gamma_{2^{-k}})\}$ combined with Proposition \ref{prop:nondeg} (after a rescaling which only depends on $n$) implies that there exists $c_0=c_0(n,p)$ such that the non-degeneracy condition \eqref{eq:nondeg} is satisfied for $2w_1$ and $\bar u$:
\begin{align*}
2w_1(x), \bar u(x)&\geq c_0 \dist(x,\Gamma_{2^{-k}})^{\frac{1}{2}+\frac{\bar{\epsilon}}{2}} \\
&\text{ in } \{x\in B_1| \ \dist(x,\Lambda_{2^{-k}})\geq \ell_0\dist(x,\Gamma_{2^{-k}})\},\\
2w_1(x), \bar u(x)&\geq - c_0\dist(x,\Gamma_{2^{-k}})^{1-\frac{n+1}{p}+\sigma} \\
&\text{ in } \{x\in B_1| \ \dist(x,\Lambda_{2^{-k}})< \ell_0 \dist(x,\Gamma_{2^{-k}})\}.
\end{align*}
We remark that the constant $c_0$ can be chosen in the same way in each iteration step.
\end{itemize}
Therefore, step 1 implies (with a possibly different $s_0$ but with the same parameter dependence)
\begin{align*}
\frac{2 w_1}{\bar{u}}\geq s_0 \text{ in } \{x\in B_{1/2}| \ \dist(x,\Lambda_{2^{-k}})\geq\ell_0 \dist(x,\Gamma_{2^{-k}})\}.
\end{align*}
Rescaling back and using the scale invariance of $\mathcal{C}'_\theta (e_n)$, we infer
\begin{align*}
\frac{u_2-a_k u_1}{(b_k -a_k)u_1} \geq \frac{s_0}{2}\text{ in } B_{2^{-k-1}}\cap \mathcal{C}'_\theta(e_n).
\end{align*}
Therefore,
$$
(a_k + \frac{s_0}{2}(b_k-a_k)) u_1 \leq u_2 \leq b_k u_1 \text{ in } B_{2^{-k-1}}\cap \mathcal{C}'_\theta(e_n).
$$
Thus, we define $a_{k+1}:= a_k + \frac{s_0}{2}(b_k-a_k)$ and $b_{k+1}:=b_k$. These satisfy
\begin{align*}
b_{k+1} - a_{k+1} = \left(1-\frac{s_0}{2}\right)(b_k -a_k).
\end{align*}
Let $\mu_1= 1 - \frac{s_0}{3}$ and $\mu_2 =1-\frac{s_0}{2}$ with $s_0$ chosen so small that $1-\frac{s_0}{2}>2^{-\frac{\bar{\epsilon}}{2}}$. This yields the conditions in ($\alpha$) and ($\beta$) and hence proves the desired estimate.
\end{proof}

\subsubsection{Proof of Proposition \ref{prop:C1a}}
\label{sec:proofC1a}
\begin{proof}[Proof of Proposition \ref{prop:C1a}]
The proof is standard (c.f. Theorem 6.9 of \cite{PSU}). We only give a brief sketch of it and indicate the main modifications. The idea is to apply the boundary Harnack inequality (Lemma~\ref{lem:boundHarn}) to the tangential derivatives $\p_e w$, $e\in \mathcal{C}'_\eta(e_n)$. This gives the uniform (in $h$) $C^{1,\alpha}$ regularity of level sets $\{w=h\}$ with small $h>0$. Then the uniform convergence of the level sets as $h\rightarrow 0$ allows to deduce  the $C^{1,\alpha}$ regularity of the zero level set, which is the free boundary $\Gamma_w$.\\

We consider tangential derivatives $v:=\p_e w$ with $e\in \{x\in S^n\cap B'_1\big| x_n>  \eta |x''|\}$. Here $\eta= \tilde{C}(n,p)\max\{\epsilon_0, c_\ast\}$  is the same constant as in \eqref{eq:positive_cone}. Using the off-diagonal assumption (A1) and extending $v$ to the whole ball $B_1$ as described in Remark \ref{rmk:ref_ext}, we note that it solves
\begin{align*}
\p_i a^{ij} \p_{j} v &=  \p_i F^{i} \mbox{ in } B_{1}\setminus \Lambda_w,\\
v &= 0 \mbox{ on } B_{1}' \cap \Lambda_w.
\end{align*}
By Lemma \lemfreebgrowth, $\|\dist(\cdot,\Gamma_w)^{-\frac{1}{2}}\ln(\dist(\cdot,\Gamma_w))^{-2}F^i\|_{L^p(B_1)}\leq C(n,p)c_\ast$.\\
Moreover, by \eqref{eq:positive_cone}, $v$ satisfies the non-degeneracy conditions in Lemma~\ref{lem:boundHarn} (after passing to the normalized function $\frac{v}{c_n}$). 
Thus, for $\epsilon_0$ and $c_\ast$ sufficiently small, the boundary Harnack inequality (Lemma~\ref{lem:boundHarn}) yields the existence of constants $\alpha \in (0,1)$, $\theta=\theta(n,p)$ and $C=C(n,p)$ such that for all $x_0 \in \Gamma_w\cap B'_{1/4}$
\begin{align}
\label{eq:boundary_holder}
\left| \frac{\p_e w}{\p_n w}(x) - \frac{\p_e w}{\p_n w }(x_0)\right| \leq C \frac{\p_e w}{\p_nw}(e_n/2)|x-x_0|^{\alpha}, \quad x\in (x_0+\mathcal{C}_\theta' (e_n))\cap B_{\frac{1}{4}}(x_0).
\end{align} 
We observe that for $j\in\{1,\dots,n-1\}$, the estimate (\ref{eq:boundary_holder}) can be improved to the following more quantitative control
\begin{align}
\label{eq:boundary_holder1}
\left|\frac{\p_j w}{\p_n w}(x)-\frac{\p_j w}{\p_nw}(x_0)\right|\leq C \max\{\epsilon_0,c_\ast\} |x-x_0|^{\alpha}.
\end{align}
Indeed, given $e_j\in S^n\cap B'_1$ with $j\in\{1,\dots, n-1\}$, we consider the vector $e=c_1 e_j+c_2 e_n$, where the constants $c_1,c_2$ are chosen such that $e\in \{x\in S^n\cap B'_1\big | x_n> \eta|x''|\}$. We note that (for sufficiently small values of $\eta$) it is possible to consider $ c_1 \geq \frac{1}{2}$ and $\tilde{C}(n,p)\max\{\epsilon_0,c_\ast\}\leq c_2\leq 2 \tilde{C}(n,p)\max\{\epsilon_0,c_\ast\}$. Inserting this into (\ref{eq:boundary_holder}) thus leads to (\ref{eq:boundary_holder1}).\\
We claim that (\ref{eq:boundary_holder}) and (\ref{eq:boundary_holder1}) imply the desired $C^{1,\alpha}$ regularity of the free boundary and give the desired explicit bound for the Hölder constant of $g$. Indeed, heuristically this follows from the identity 
\begin{align}
\label{eq:derg}
\p_j g(x'') = -\left. \frac{\p_j w}{\p_n w}\right|_{(x'',g(x''),0)},
\end{align}
and the fact that $\frac{\p_j w(x)}{\p_n w(x)}\big | _{B'_1\setminus \Lambda_w}$ is Hölder continuous up to $\Gamma_w$. Rigorously, we have to justify that the derivative of $g$ exists and satisfies the claimed identity. In fact, since $\p_n w>0$ in $B'_1\setminus \Lambda_w$, for $h>0$ small, the level sets $\{w=h\}\cap B'_1$ are given as graphs $x_n=g_h(x'')$, where $\{g_h\}_{h\geq 0}$ is an increasing sequence in $h$ which, by \eqref{eq:boundary_holder}, is uniformly bounded in $C^1(B''_1)$. Thus, $g_h$ converges uniformly to $g$ as $h$ goes to zero by Dini's theorem. Next we notice that 
\begin{align*}
\p_j g_h(x'')=- \left. \frac{\p_j w}{\p_n w}\right|_{(x'',g_h(x''),0)}.
\end{align*}
Therefore, by \eqref{eq:boundary_holder}, 
\begin{align*}
\left|\p_j g_h(x'')-\left. \left(-\frac{\p_j w}{\p_n w}\right)\right|_{(x'',g(x''),0)} \right|&= \left|\left.\frac{\p_j w}{\p_n w}\right|_{(x'',g_h(x''),0)}-\left. \frac{\p_j w}{\p_n w}\right|_{(x'',g(x''),0)} \right|\\
&\leq C\left| g_h(x'')-g(x'')\right|^\alpha.
\end{align*}
Since the sequence $g_h$ converges to $g$ uniformly, we infer
\begin{align*}
\left\| \p_j g_{h}(x'') - \left.\left(-\frac{\p_j w}{\p_n w} \right)\right|_{(x'',g(x''),0)} \right\|_{L^{\infty}(B_{\frac{1}{2}}'')} \rightarrow 0 \mbox{ as } h \rightarrow 0.
\end{align*} 
Therefore, by the fundamental theorem of calculus, the partial derivatives of $g$ exist and they satisfy the identity (\ref{eq:derg}). Moreover, by \eqref{eq:boundary_holder} we have $\p_jg\in C^{0,\alpha}(\Gamma_w\cap B'_{1/4})$. This finishes the proof of Proposition \ref{prop:C1a}.
\end{proof}

\begin{rmk}
\label{rmk:alpha} 
Under the conditions of Proposition \ref{prop:C1a}, the $C^{1,\alpha}$ norm of $g$ is bounded by a constant, which only depends on $n$ and $p$.  Indeed, by Proposition~\ref{prop:Lip} and \eqref{eq:boundary_holder1} (and by $g(0)=0$) we have that 
$$\|\nabla g\|_{C^{1,\alpha}(B''_{1/2})}\leq C(n,p)\max\{\epsilon_0,c_\ast\}.$$ 
Moreover, the proof of Lemma \ref{lem:boundHarn} and the proof of Proposition \ref{prop:C1a} show that in this case the H\"older exponent $\alpha\in(0,1)$ only depends on $n$ and $p$ (if $\epsilon_0, c_{\ast}$ are chosen sufficiently small in dependence of $n,p$).
\end{rmk}

\subsection{Proof of Proposition \ref{prop:v1}}
\label{sec:propv1} 

Last but not least, we come to the proof of Proposition \ref{prop:v1}. For convenience, we introduce the following notation
$$d(x):=\dist(x,\Gamma).$$ 

The proof of Proposition \ref{prop:v1} is a consequence of weighted energy estimates (with weights which are naturally induced by the symbol of the operator) and local pointwise estimates.

\begin{proof}
\emph{Step 1: Existence and uniqueness.} We define a Hilbert space $\mathcal{H}$ in which we solve an appropriate weak formulation of our problem. Let $u,v \in C^{\infty}_{0}(\R^{n+1}\setminus \Lambda)$. Consider the following inner product
\begin{align*}
&(\cdot, \cdot)_{\mathcal{H}}: C^{\infty}_{0}(\R^{n+1}\setminus \Lambda) \times C^{\infty}_{0}(\R^{n+1}\setminus \Lambda) \rightarrow \R,\\
& (u,v)_{\mathcal{H}}:= \int\limits_{\R^{n+1}}\nabla u \cdot\nabla v dx + \int\limits_{\R^{n+1}} d(x)^{-2} u v dx \mbox{ for all } u,v \in \mathcal{H}.
\end{align*}
Using this, we define our Hilbert space $\mathcal{H}$ as the closure of $C^{\infty}_{0}(\R^{n+1}\setminus \Lambda)$ with respect to $\left\| u\right\|_{\mathcal{H}}^2:=(u,u)_{\mathcal{H}}$:
\begin{align*}
\mathcal{H} : = \overline{C^{\infty}_{0}(\R^{n+1}\setminus \Lambda)}^{\left\| \cdot \right\|_{\mathcal{H}}} 
\end{align*}
On $\mathcal{H}$ the bilinear form associated with the equation \eqref{eq:globalv1},
\begin{equation*}
\mathcal{B}(u,v)=\int\limits_{\R^{n+1}} a^{ij}\partial_iu \partial_j v + K\int\limits_{\R^{n+1}} d(x)^{-2} uv, \quad u,v\in \mathcal{H},
\end{equation*}
is bounded and coercive, thus by the Lax-Milgram theorem a unique solution exists in $\mathcal{H}$. This proves the first part of Proposition \ref{prop:v1}.\\

\emph{Step 2: Weighted energy estimate.}
Let $\bar{g}$ be the intrinsic metric determined by the symbol $d(x)^2 |\xi|^2$ in the cotangent or by the symbol $d(x)^{-2}|\xi|^2$ in the tangent space, i.e
\begin{equation*}
\bar{g}_x(v,w)=d(x)^{-2}v\cdot w, \quad v,w \in T_x(\R^{n+1}\setminus \Gamma), \ x\in \R^{n+1}\setminus \Gamma.
\end{equation*}
The distance function associated with $\bar{g}$ is given by
\begin{equation}\label{eq:intrinsic_dist}
d_{\bar{g}}(x_1,x_2)=\inf_{C}\int_C d(x(s))^{-1} d s,
\end{equation}
where the infimum is taken over all rectifiable curves $C$ in $\R^{n+1}\setminus \Gamma$ joining two points $x_1, x_2 \in \R^{n+1}\setminus \Gamma$ parametrized by arc length. We note that $d_{\bar{g}}(\cdot,\cdot)$ is finite on each connected component of $\R^{n+1}\setminus \Gamma$. Let $f:\R^{n+1}\setminus \Gamma\rightarrow \R$ be a Lipschitz function w.r.t. the intrinsic metric $d_{\bar{g}}$.
It is not hard to see that $f$ is Lipschitz w.r.t. $d_{\bar{g}}$ if and only if it is locally Lipschitz w.r.t. the Euclidean distance and it satisfies
$$ |\nabla f(x)|\leq C d(x)^{-1} \text{ for a.e. } x\in\R^{n+1}\setminus \Gamma.$$

We decompose the open set $\R^{n+1}\setminus \Gamma$ into Whitney cubes $\{Q_j\}$ with respect to the Euclidean distance in $\R^{n+1}\setminus \Gamma$ (c.f. Section \ref{sec:Reifenbergbarrier} for the definition of a Whitney decomposition). Without loss of generality, we may assume that the decomposition is symmetric with respect to the $x_{n+1}$ variable. Indeed, the symmetry assumption can always be realized as $\Gamma \subset \R^{n}\times \{0\}$.\\
We now make the following claim:
\begin{claim*}
For any $\tilde{\kappa}>0$, there exists $s=s(n, \tilde{\kappa})>0$ sufficiently large, such that for any compact subset $W\subset \R^{n+1}\setminus \Gamma$ we have
\begin{equation}\label{eq:exp}
\sum_{j=1}^{\infty}\|d^{-\tilde{\kappa}}e^{-sd_{\bar{g}}(x,W)}\|_{L^\infty(Q_j)}\leq C(n)\dist(W,\Gamma)^{- \tilde{\kappa}}.
\end{equation}
\end{claim*}
We postpone the proof of the claim to step 4 and continue the proof of Proposition \ref{prop:v1}.\\

Let $\phi$ be a Lipschitz function w.r.t. the intrinsic distance $d_{\bar{g}}$ and assume that $|\phi(x)|\leq \kappa|\ln d(x)|$ for some $\kappa>0$. Consider $\tilde{u}:=e^{\phi} u$. Then (in a distributional sense) $\tilde{u}$ is a solution of 
\begin{equation}
\label{eq:conj_weight}
\begin{split}
\partial_i(a^{ij}\partial_j \tilde{u})-(\partial_i \phi) a^{ij}\partial_j \tilde{u}-\partial_i(a^{ij} \tilde{u}\partial_j \phi)
+(\partial_i \phi) (\partial_j \phi) a^{ij}\tilde{u}-Kd^{-2}\tilde{u}\\
=\partial_i(e^\phi F^i)-(\partial_i \phi) e^\phi F^i+e^\phi g.
\end{split}
\end{equation}
We insert $e^{-sd_{\bar g}(x,W)}\tilde{u}$ as a test function into the divergence form of (\ref{eq:conj_weight}), where $s=s(n,\kappa)$ is the constant from the claim and $W\subset \R^{n+1}\setminus \Gamma$ is a given set. Note that $e^{-sd_{\bar g}(x,W)}\tilde{u}$ is a function in $\mathcal{H}$. Indeed, using the definition of $\tilde{u}$, Hölder's inequality and the statement of the claim yields
\begin{align*}
\|d^{-1}e^{-sd_{\bar g}(x,W)}\tilde{u} \|_{L^2(\R^{n+1})}&\leq \|d^{-\kappa-1}e^{-sd_{\bar g}(x,W)} u \|_{L^2(\R^{n+1})}\\
&\leq \left( \sum_j \|d^{-\kappa}e^{-sd_{\bar g}(x,W)}\|_{L^\infty(Q_j)} \right) \sup\limits_{\hat{Q} \subset \{Q_k\}_k}\|d^{-1}u\|_{L^2(\hat{Q})}\\
&\leq C(n)\dist(W,\Gamma)^{-\kappa} \|d^{-1}u\|_{L^2(\R^{n+1})}.
\end{align*}
Similarly is is possible to estimate $\nabla (e^{-sd_{\bar g}(x,W)}\tilde{u})$.
As a consequence of the admissibility of $e^{-sd_{\bar g}(x,W)}\tilde{u}$ as a test function, we infer that if $K=K(\|\phi\|_{C^{0,1}_{\bar{d}_g}}, n)$ is large (where $\|\cdot \|_{C^{0,1}_{\bar{d}_g}}$ denotes the Lipschitz semi-norm with respect to the metric $d_{\bar{g}}$), then there exists $C>0$ depending on the ellipticity constants and dimension $n$ such that:
\begin{align*}
\| e^{-sd_{\bar g}(x,W)} \nabla \tilde{u} \|_{L^2(\R^{n+1})}+ K/2\|d^{-1}e^{-sd_{\bar g}(x,W)}\tilde{u}\|_{L^2(\R^{n+1})}\\
\leq C \left(\|e^{-sd_{\bar g}(x,W)} e^\phi F^i \|_{L^2(\R^{n+1})}+ \|e^{-sd_{\bar g}(x,W)} e^\phi d g \|_{L^2(\R^{n+1})}\right).
\end{align*}
This can be written in terms of $u$ as
\begin{align}\label{eq:weighted_energy}
\|e^{-s d_{\bar g}(x,W)} e^{\phi}\nabla u \|_{L^2(\R^{n+1})}+ \|e^{-s d_{\bar g}(x,W)} e^{\phi} d^{-1}u \|_{L^2(\R^{n+1})}\notag\\
\leq C\left(\|e^{-s d_{\bar g}(x,W)} e^\phi F^i\|_{L^2(\R^{n+1})}+\|e^{-s d_{\bar g}(x,W)} e^{\phi}d g\|_{L^2(\R^{n+1})}\right).
\end{align}

\emph{Step 3: Local pointwise estimate:} 
Consider a given point $\hat x \in \R^{n+1}\setminus \Gamma$ and let $\hat Q\in \{Q_j\}$ be the Whitney cube which contains $\hat x$. Let $$\hat d := d(\hat x).$$ By the weak maximum principle (c.f. Theorems 8.15 and 8.17 in \cite{GT}) as well as by the properties (W1) and (W2) from the definition of a Whitney decomposition (c.f. Section \ref{sec:Reifenbergbarrier}), there exists $C=C(n,p)$ such that
\begin{align}\label{eq:ptwise}
|u(\hat x)| \leq C\left( \hat d^{-\frac{n+1}{2}}\|u\|_{L^2(\mathcal{N}_{\hat Q})}+\hat d^{1-\frac{n+1}{p}}\|F^i\|_{L^p(\mathcal{N}_{\hat Q})}+\hat d^{2(1-\frac{n+1}{p})}\|g\|_{L^{p/2}(\mathcal{N}_{\hat Q})}\right),
\end{align}
where $\mathcal{N}_{\hat Q}$ is the union of all the Whitney cubes which touch $\hat Q$. Recalling the assumptions on $F^i$ and $g$ and using the properties (W1), (W2), we have
\begin{align}
\|F^i\|_{L^p(\mathcal{N}_{\hat Q})}&\lesssim  \|d^{-\sigma}F^i\|_{L^p(\R^{n+1})} \hat d^{\sigma}, \label{eq:estF1}\\
\|g\|_{L^{p/2}(\mathcal{N}_{\hat Q})}&\lesssim \|d^{1-\frac{n+1}{p}-\sigma}g\|_{L^{p/2}(\R^{n+1})}\hat d^{\sigma - 1+\frac{n+1}{p}} .\label{eq:estG1}
\end{align}

To estimate $\hat d^{-(n+1)/2}\|u\|_{L^2(\mathcal{N}_{\hat Q})}$, we seek to apply \eqref{eq:weighted_energy} with $W= \mathcal{N}_{\hat Q}$ and
$$\phi(x)=-\kappa\ln d(x), \quad \kappa=(n+1)\left(\frac{1}{2}-\frac{1}{p}\right)+\sigma.$$  
The Lipschitz constant (w.r.t. $d_{\bar{g}}$) of $\phi$ is bounded by $c_n(\sigma+1)$. Thus, by taking $K=C(\sigma+1)$ with $C=C(n)$ large, we have \eqref{eq:weighted_energy}. We notice that from the properties (W1) and (W2) for Whitney cubes as well as \eqref{eq:weighted_energy},
\begin{equation}
\label{eq:ch}
\begin{split}
\hat{d}^{-\frac{n+1}{2}}\|u\|_{L^2(\mathcal{N}_{\hat Q})}
&\lesssim \hat d^{\sigma+1-\frac{n+1}{p}} \|d^{-\kappa-1} u\|_{L^2(\mathcal{N}_{\hat Q})}\\
&\lesssim \hat d^{\sigma+1-\frac{n+1}{p}} \|d^{-\kappa-1}e^{-sd_{\bar g}(x,\mathcal{N}_{\hat Q})} u\|_{L^2(\R^{n+1})}\\
&\lesssim  \hat d^{\sigma+1-\frac{n+1}{p}} \left(\|d^{-\kappa}e^{-sd_{\bar g}(x,\mathcal{N}_{\hat Q})}F^i\|_{L^2(\R^{n+1})}\right.\\
& \quad \left.+\|d^{-\kappa+1}e^{-sd_{\bar g}(x,\mathcal{N}_{\hat Q})} g\|_{L^2(\R^{n+1})}\right).
\end{split}
\end{equation}
Due to H\"older's inequality, property (W2) and due to the definition of $\kappa$
\begin{align*}
\|d^{-\kappa} F^i\|_{L^2(Q')}&\lesssim \|d^{-\sigma} F^i\|_{L^p(\R^{n+1})}\|d^{-\kappa+\sigma}\|_{L^{\frac{2p}{p-2}}(Q')}  \\
&\lesssim\|d^{-\sigma} F^i\|_{L^p(\R^{n+1})},\notag\\
\|d^{-\kappa+1} g\|_{L^2(Q')}&\lesssim \|d^{1-\frac{n+1}{p}-\sigma} g\|_{L^{p/2}(\R^{n+1})},
\end{align*}
for any $Q'\in \{Q_j\}$. Thus, combining this estimate with \eqref{eq:exp} (applied with $\tilde{\kappa}=0$) leads to
\begin{equation}
\label{eq:exp2}
\begin{split}
\|d^{-\kappa}e^{-sd_{\bar g}(x,\mathcal{N}_{\hat Q})} F^i\|_{L^2(\R^{n+1})}
& \leq \sum\limits_j \left\| e^{-s d_{\bar{g}}(x,\mathcal{N}_{\hat{Q}})} \right\|_{L^{\infty}(Q_j)} \sup\limits_{Q'\in \{Q_j\}}\|d^{-\kappa} F^i\|_{L^2(Q')}\\
&\leq C(n) \|d^{-\sigma} F^i\|_{L^p(\R^{n+1})}.
\end{split}
\end{equation}
Analogously, we infer
\begin{align*}
\|d^{-\kappa+1}e^{-sd_{\bar g}(x,\mathcal{N}_{\hat Q})} g\|_{L^2(\R^{n+1})}&\leq C(n)\|d^{1-\frac{n+1}{p}-\sigma}g\|_{L^{p/2}(\R^{n+1})}.
\end{align*}
Combining \eqref{eq:ptwise}-\eqref{eq:estG1}, \eqref{eq:ch} and \eqref{eq:exp2}, yields a constant $C_0=C_0(n,p)$ such that 
\begin{equation*}
|u(\hat x)|\leq C_0 \hat d ^{\sigma+1-\frac{n+1}{p}}\left(\|d^{-\sigma}F^i\|_{L^p(\R^{n+1})}+\|d^{1-\frac{n+1}{p}-\sigma}g\|_{L^{p/2}(\R^{n+1})}\right).
\end{equation*}

\emph{Step 4: Proof of the claim.} Without loss of generality, we only prove the claim for $W=\mathcal{N}(\hat Q)$. Let $\mathcal{N}_1:=\{Q_j| \ Q_j\text{ touches } \mathcal{N}(\hat Q)\}$ be the set of all Whitney cubes which touch $\mathcal{N}(\hat Q)$. In general, for $k\geq 2$ let $\mathcal{N}_k$ denote all Whitney cubes which touch at least one of the Whitney cubes in $\mathcal{N}_{k-1}$. Hence $\mathcal{N}_k\nearrow \{Q_j\}$. It is immediate from property (W3) for Whitney cubes that:
\begin{itemize}
\item[($\alpha$)] Let $\sharp\mathcal{N}_k$ denote the number of the Whitney cubes contained in $\mathcal{N}_k$. Define $\sharp\mathcal{N}_0=0$. Then $\sharp\mathcal{N}_k-\sharp\mathcal{N}_{k-1}\leq (12)^{k(n+1)}$, $k\geq 1$. 
\end{itemize}
Moreover, from \eqref{eq:intrinsic_dist}, property (W1) for Whitney cubes and an induction argument, it is not hard to see that:
\begin{itemize}
\item[($\beta$)] For any $x\in Q_j\in \mathcal{N}_k$, $k\geq 2$ we have
$d_{\bar{g}}(x,\mathcal{N}(\hat Q))\geq \frac{k-1}{4\sqrt{n+1}}$.
\end{itemize}
Let $r_0$ be the minimal sidelength of $Q_j \in \mathcal{N}(\hat Q)$. Note that by (W1) and (W2) we have $r_0\leq \dist(\mathcal{N}_{\hat Q},\Gamma)$. Combining ($\alpha$) and ($\beta$), we obtain
\begin{align*}
&\sum_{j=1}^{\infty}\|d^{-\kappa}e^{-sd_{\bar{g}}(x,\mathcal{N}_{\hat Q})}\|_{L^\infty(Q_j)}\\
&\lesssim 
r_0^{-\kappa}+ (4^{-1}r_0)^{-\kappa} \sharp \mathcal{N}_1 +\sum_{k=2}^\infty\sum_{Q_j\in \mathcal{N}_k\setminus \mathcal{N}_{k-1}}(4^{-k}r_0)^{-\kappa}\|e^{-sd_{\bar{g}}(x,\mathcal{N}_{\hat Q})}\|_{L^\infty(Q_j)}\\
&\leq r_0^{-\kappa}+4^{\kappa}12^{n+1} r_0^{-\kappa} + r_0^{-\kappa}\sum_{k=2}^{\infty}4^{k \kappa}e^{-\frac{s(k-1)}{4\sqrt{n+1}}}12^{k(n+1)}\\
&\leq C(n)r_0^{-\kappa},
\end{align*}
if $s=s(n,\kappa)$ is chosen to be sufficiently large.\\

\emph{Step 5: Continuous realization of the boundary data.} In order to obtain the continuous attainment of the boundary data, we observe that any point $x \in \Lambda$ lies in an associated Whitney cube $Q$. Since the Whitney decomposition is considered with respect to $\Gamma$, the cube $Q$ does not intersect $\Gamma$. As $\Gamma = \partial \Lambda$, this implies that $Q \cap (\R^{n} \times \{0\}) \subset \Lambda$. Moreover, the same is true for all neighboring Whitney cubes $\mathcal{N}(Q)$. Hence, we find a ball $B_Q$ which contains $Q$ such that 
\begin{align*}
\p_i a^{ij}\p_j u - d(x)^{-2}u &= \p_i F^i + g \mbox{ in } B_Q,\\
u& = 0 \mbox{ on } B_Q\cap (\R^{n}\times \{0\}).
\end{align*}
In particular, this implies that the boundary data of $u$ are attained continuously in $Q$ (c.f. \cite{GT}, Corollary 8.28). As this holds for any point $x\in \Lambda$, it proves the desired continuous attainment of the boundary data.
\end{proof}

\begin{rmk}\label{rmk:impv2}
In order to deduce the improved behavior close to $\Lambda$ from Remark \ref{rmk:impv1}, we argue by elliptic regularity and a scaling argument. In the case of a point $x_0\in B_1$ with $d(x_0)\sim 1$ but $\dist(x_0,\Lambda)\leq c_n$ (where $c_n$ is a dimensional constant less than $1/2$) we have the bound
\begin{align*}
\| u\|_{W^{1,p}(B_{1/2}(x_0))} \leq C (\| d^{-\sigma}F\|_{L^p(\R^{n+1})} + \| d^{-1-\sigma}g\|_{L^{p/2}(\R^{n+1})} ).
\end{align*}
By the Sobolev embedding theorem and the fact that $u$ vanishes on $\Lambda$, this implies that 
\begin{align*}
|u(x_0)| \leq C (\| d^{-\sigma}F\|_{L^p(\R^{n+1})} + \| d^{1-\sigma-\frac{n+1}{p}}g\|_{L^{p/2}(\R^{n+1})} )\dist(x_0,\Lambda)^{1-\frac{n+1}{p}}.
\end{align*}
The general case follows by rescaling to the set-up with $d(x_0)\sim1$.
\end{rmk}

\section[Optimal Regularity]{Optimal Regularity and Homogeneity of Blow-Up Solutions}
\label{sec:optimal}

In this section we return to the study of solutions of the thin obstacle problem (\ref{eq:varcoef}) and present some consequences of the $C^{1,\alpha}$ regularity of the free boundary. Here our main result is an optimal regularity result for solutions of (\ref{eq:varcoef}), c.f. Theorem \ref{thm:optimal_reg}, and a leading order asymptotic expansion of the solution at free boundary points, c.f. Proposition~\ref{prop:wasympt}.\\

In the sequel, we assume that the metric $a^{ij}$ satisfies the assumptions (A1), (A2), (A3), (A4). 
Moreover, the metric $a^{ij}$ and the solution $w$ are extended to $B_1^-$ in the same way as in Remark~\ref{rmk:ref_ext}.
In this section we mainly work with the equation of $\p_j w$ in our $C^{1,\alpha}$ slit domain. For $j\in \{1,\dots, n\}$, $\p_j w$ satisfies \eqref{eq:extend_v} in $B_1\setminus \Gamma_w$. For $j=n+1$,
we extend $w$ to $B_1^-$ by setting $w(x',x_{n+1})=-w(x',-x_{n+1})$. Then $\p_{n+1} w\in C^{0,1/2-\delta}(B_1\setminus \overline{\Omega_w})\cap H^1(B_1\setminus \overline{\Omega_w})$ for any $\delta>0$, as $w=0$ on $\overline{\Omega_w}$. Moreover, it solves the divergence form equation
 \begin{align*}
 \p_i a^{ij}\p_j v &= \p_i F^i \text{ in } B_1\setminus \overline{\Omega_w}, \text{ where } F^i=-(\p_{n+1} a^{ij})\p_j w,\\
 v&=0 \text{ on }\overline{\Omega_w}.
 \end{align*}

Our main results rely on identifying the asymptotic behavior of solutions, c.f. Lemma \ref{lem:lowerbarrier} and Propositions \ref{prop:asympt} and \ref{prop:wasympt}, for which we invoke the $C^{1,\alpha}$ regularity of the regular free boundary. This then implies the optimal local growth behavior (c.f. Corollaries \ref{cor:lowerbound}-\ref{cor:optgrowth}) which in combination with our Carleman estimate from \cite{KRS14} yields the optimal regularity of solutions.\\

This section is organized as follows: We first deduce the behavior of general solutions of the Dirichlet problem in slit domains with $C^{1,\alpha}$ slits in Section \ref{subsubsec:low_up_barrier}. Then in Section \ref{subsubsec:local32} we apply these results to solutions of the thin obstacle problem, which yields the optimal local non-degeneracy and local growth estimates as well as a leading order asymptotic expansion in appropriate cones. In Section \ref{sec:opti} the local results are then improved to global estimates and the optimal regularity result. Finally, in Section \ref{sec:lower} we prove uniform lower bounds for a suitably normalized solution of the variable coefficient thin obstacle problem at regular free boundary points.

\subsection{Asymptotics and growth estimates}
\label{subsubsec:low_up_barrier}

In this subsection we exploit the $C^{1,\alpha}$ regularity of the free boundary, in order to obtain a leading order asymptotic expansion of solutions of elliptic equations with controlled inhomogeneities in $C^{1,\alpha}$ slit domains (Proposition \ref{prop:asympt}). This expansion holds in appropriate cones. Using the asymptotic behavior, we then immediately obtain (local) growth estimates in the form of a Hopf principle (Corollary \ref{cor:lower_bound}) and the optimal upper growth bounds (Corollary \ref{cor:upper_bound}). Our approach relies on applying the boundary Harnack estimate to a given solution and a model function, which we construct in Lemma \ref{lem:lowerbarrier}.\\

In the whole of this section, we work under the following assumptions:

\begin{assum}
\label{assum:set2}
We assume that $\Lambda$ and $\Gamma$ are closed subsets of $\R^n\times\{0\}$ with 
\begin{align*}
\Lambda=\{x_n\leq g(x'')\}, \quad \Gamma=\{x_n=g(x'')\},
\end{align*}
where $g\in C^{1,\alpha}_{loc}(\R^{n-1})$ for some $\alpha\in (0,1]$ and $g(0)=0$. 
\end{assum}
As in the previous section, we use the abbreviation $L$ for the operator $L=\p_i a^{ij} \p_j$, and $L_0$ for $L_0 = a^{ij}\p^2_{ij}$. Again we use the abbreviations from Notation \ref{not:not1} and set $\ell_0 := (2^{4}\sqrt{n})^{-1}$ and $c_{\ast}:=\left\| \nabla a^{ij} \right\|_{L^p(B_1)}$.\\

Under these assumptions we now refine the estimates from the barrier construction from Proposition \ref{prop:barrier} in Section \ref{sec:boundary}, in order to identify the leading order contribution in the solutions to our equations close to free boundary points.

\begin{lem}
\label{lem:lowerbarrier}
Let $\Lambda, \Gamma$  and $g$ satisfy the conditions in Assumption \ref{assum:set2} and let $\ell_0, c_{\ast}$ be as in Notation \ref{not:not1}. There exists a function $h_0:B_1\rightarrow \R$ such that
\begin{itemize}
\item[(i)] $h_0 \geq 0$ on $B_1 \setminus \Lambda$ \mbox{ and } $h_0=0$ on $\Lambda$,
\item[(ii)] $h_0(x)\geq c_n\dist(x,\Gamma)^{1/2}$ in $B_1\cap \{\dist(x,\Lambda)\geq \ell_0\dist(x,\Gamma)\}$,
\item[(iii)] $L h_0=g_1+g_2$, where $g_1$ and $g_2$ satisfy 
\begin{align*}
\|\dist(\cdot,\Gamma)^{\frac{1}{2}}g_1\|_{L^p(B_1)}&\leq C_n c_{\ast},\\
\|\dist(\cdot,\Gamma)^{\frac{3}{2}-\min\{\alpha,1-\frac{n+1}{p}\}}g_2\|_{L^\infty(B_1)}&\leq C_n\left(c_{\ast}+\|\nabla'' g\|_{C^{0,\alpha}(B_1'')}\right).
\end{align*}
\end{itemize}
\end{lem}

\begin{proof}
As the proof follows along the lines of Proposition \ref{prop:barrier}, we only indicate the main differences.\\
Similarly as in our previous construction, we consider
\begin{align*}
v^{0}(x):= w_{1/2}(x_n, x_{n+1}),
\end{align*}
and define $v_k^0(x)$ by composition with the change of coordinates $T_k$ as in Proposition \ref{prop:barrier}.
Using these building blocks, we define $h_0$ in analogy to Proposition \ref{prop:barrier}, i.e. for an appropriate partition of unity, $\eta_k$, we set:
\begin{align*}
h_0(x) := \sum\limits_{k} \eta_k v_k^0(x).
\end{align*}
By construction $h_0$ satisfies (i)-(ii).\\
To show (iii), we first notice that $Lh_0=(\p_i a^{ij})(\p_j h_0)+L_0h_0:=g_1+g_2.$ The bound for $g_1$ follows immediately by noting that $|\nabla h_0(x)|\leq C_n\dist(x,\Gamma)^{-1/2}$.
In order to bound $g_2$, we compute $L_0(h_0)$ similarly as in step 1 of Proposition \ref{prop:barrier}. Combining the fact that for each $k$, $a^{ij}(x_k)\p_{ij}v_k^0(x)=0$ in $\R^{n+1}\setminus\{x_{n+1}=0,(x-x_k)\cdot \nu_k\leq 0\}$, with the properties of the Whitney decomposition leads to
\begin{align*}
\sum_k (a^{ij}(x_k)\p_{ij}v_k^0(x))\eta_k(x) =0 \text{ in } B_1\setminus \Lambda.
\end{align*}
We proceed with the estimate for the error contributions: Let $\nu_k$ and $\nu_\ell$ denote the normals of $\Lambda$ at $x_k \in \Gamma$ and $x_\ell \in \Gamma$. Here $x_k, x_\ell\in \Gamma\cap B_1$ are the projection points of two centers, $\hat x_k$ and $\hat x_\ell$, of neighboring Whitney cubes, $Q_k$ and $Q_\ell$. For these normals we infer $|\nu_k-\nu_\ell|\lesssim \|\nabla'' g\|_{C^{0,\alpha}(B_1'')} r_\ell^{\alpha}$.  Thus, defining $\gamma:=1-\frac{n+1}{p}$ and following the computation in Proposition~\ref{prop:barrier}, we have that for all $x\in Q_\ell$,
\begin{align*}
&\sum_k(a^{ij}(x_k)\p_{ij}\eta_k(x)) v_k^0(x)
\leq C(n)\left(c_{\ast} r_\ell ^{-\frac{3}{2}+\gamma}+ \|\nabla'' g\|_{C^{0,\alpha}(B_1'')} r_\ell^{-\frac{3}{2}+\alpha}\right),\\
&\sum_k a^{ij}(x_k) \p_i \eta_k \p_j v_k^0 
\leq C(n) \left(c_{\ast} r_\ell^{-\frac{3}{2}+ \gamma}+ \|\nabla'' g\|_{C^{0,\alpha}(B_1'')} r_\ell^{-\frac{3}{2}+\alpha}\right),\\
&\sum_k (a^{ij}(x)-a^{ij}(x_k))\p_{ij} (\eta_k v_k^0(x))\leq C(n) c_{\ast} r_\ell ^{-\frac{3}{2}+\gamma}.
\end{align*}
Thus, the bound for $g_2$ follows. 
\end{proof}

We continue by deriving the boundary asymptotic behavior of solutions $u$ to general divergence form equations
 in slit domains with $C^{1,\alpha}$ slit.  The main tool is our boundary Harnack inequality (Lemma~\ref{lem:boundHarn}).\\

\begin{prop}
\label{prop:asympt}
Let $\Lambda, \Gamma$ and $g$ be as in Assumption \ref{assum:set2}.
Suppose that $u\in C(B_1)\cap H^{1}(B_1)$ is a solution of 
\begin{equation}
\label{eq:divergence}
\p_i a^{ij}\p_j u =\p_i F^i +g_1 \text{ in } B_1\setminus \Lambda, \quad u=0\text{ on }\Lambda.
\end{equation}
Further assume that $a^{ij}$, $F^i$, $g_1$ and $u$ satisfy the assumptions from Proposition~\ref{prop:nondeg}.\\
If $c_\ast$, $\bar\delta$ and $\|\nabla'' g\|_{C^{0,\alpha}(B_1'')}$ are sufficiently small depending on $n$, $p$ and $\alpha$, then
there exists a constant $\beta\in (0,\min\{\alpha, 1-\frac{n+1}{p}\}]$, a $C^{0, \beta}(\Gamma \cap B_{1/2}')$ function $b: \Gamma \rightarrow \R$, $b>0$,
such that for all $x_0\in \Gamma\cap  B_{1/4}'$ and $x\in B_{1/4}(x_0)$ 
\begin{equation}
\label{eq:asympt}
\begin{split}
\left|u(x) - b(x_0) w_{1/2}\left(\frac{\nu_{ x_0} \cdot (x- x_0) }{(\nu_{ x_0} \cdot A( x_0) \nu_{ x_0})^{1/2}},\frac{x_{n+1}}{(a^{n+1,n+1}( x_0))^{1/2}}\right)\right|\\
 \leq C \left(c_\ast+\left\| \nabla'' g \right\|_{C^{0,\alpha}(B_{1/2}'')}\right) |x - x_{0}|^{\frac{1}{2}+\beta}. 
\end{split}
\end{equation}
Here $\nu_{x_0}$ is the outer unit normal of $\Lambda$ at $x_0\in \Gamma$, $g$ denotes the parametrization of $\Gamma$ and $A(x):= (a^{ij}(x))$.
\end{prop}

\begin{proof}
We seek to apply the boundary Harnack inequality (Lemma~\ref{lem:boundHarn}) to the pair $u, h_0$ (with $\delta_0=\min\{\alpha,1-\frac{n+1}{p}\}$), where $h_0$ is the function from Lemma~\ref{lem:lowerbarrier}. By Lemma~\ref{lem:lowerbarrier}, $h_0$ satisfies the conditions in Lemma~\ref{lem:boundHarn}. Moreover, we note that by Proposition~\ref{prop:nondeg}
\begin{equation}\label{eq:nondege}
u(x)\geq c_n\dist(x,\Gamma)^{\frac{1}{2}+\frac{\bar\epsilon}{2}} \text{ in } B_{1/2}\cap \{\dist(x,\Lambda)\geq\ell_0\dist(x,\Gamma)\},
\end{equation}
and 
\begin{equation}\label{eq:lower_bound_a}
u(x)\geq -\dist(x,\Gamma)^{\frac{1}{2}+\bar\epsilon} \text{ in } B_{1/2}.
\end{equation}
Here $\bar{\epsilon}= \sigma+\frac{1}{2}-\frac{n+1}{p}$ and $\sigma$ is the constant from Proposition \ref{prop:nondeg}.
Hence, by the boundary Harnack inequality we have that 
\begin{align}
\label{eq:bound_Harnack}
s_0\frac{u(e_n/2)}{h_0(e_n/2)} \leq \frac{u(x)}{h_0(x)}\leq s_0^{-1} \frac{u(e_n/2)}{h_0(e_n/2)},
\end{align} 
for $x\in B_{1/2}\cap \{\dist(x,\Lambda)\geq \ell_0\dist(x,\Gamma)\}$ and for some $s_0=s_0(n,p,\|\nabla'' g\|_{C^{0,\alpha}(B_1'')})$. Moreover, $\frac{u(x)}{h_0(x)}$ is H\"older continuous in $\mathcal{C}_\theta(\Gamma\cap B'_{1/4})$ up to $\Gamma\cap B'_{1/4}$. More precisely, for each $x_0\in \Gamma\cap B'_{1/4}$, the limit of $\frac{u(x)}{h_0(x)}$ as $x\rightarrow x_0$, $x\in x_0+\mathcal{C}_\theta(e_n)$, exists. Denoting it by $b(x_0)$, gives
\begin{align}
\label{eq:quant_est}
\left| \frac{u(x)}{h_0(x)} - b(x_0) \right| \leq C \frac{u(e_n/2)}{h_0(e_n/2)}|x-x_0|^{\tilde{\beta}}, \quad x\in (x_0+\mathcal{C}_\theta(e_n)) \cap B_{1/2},
\end{align} 
where $C=C(n,p)$, for some $\tilde{\beta} \in (0,1)$. By \eqref{eq:nondege}, Lemma~\ref{lem:lowerbarrier} (iii) for $h_0$ and by the fact that $g\in C^{1,\alpha}(B_1'')$, we have $ u(e_n/2)/h_0(e_n/2)\geq c_n>0$.  Combining this with the lower bound in \eqref{eq:bound_Harnack} implies that $b(x_0)>0$.\\
We now extend this to the whole neighborhood $B_{1/4}(x_0)$ instead of only considering the cones $x_0 + C_{\theta}(e_n)$:
Using (\ref{eq:lower_bound_a}) on all scales and arguing as in the proof of Lemma~\ref{lem:boundHarn}, we obtain for  $x\in B_{1/4}(x_0)$
\begin{align*}
u(x) - b(x_0) h_0(x) &\geq  -C|x-x_0|^{\min\{\frac{1}{2}+\bar\epsilon, \frac{1}{2}+\tilde{\beta}\}},\\
u(x)-b(x_0)h_0(x)&\leq  C|x-x_0|^{\min\{\frac{1}{2}+\bar\epsilon, \frac{1}{2}+\tilde{\beta}\}}.
\end{align*}
Consequently, for $\beta=\min\{\tilde{\beta}, \bar\epsilon\}$, 
\begin{align}
\label{eq:dist1}
|u(x)-b(x_0)h_0(x)|\leq C|x-x_0|^{\frac{1}{2}+\beta} \mbox{ for } x\in B_{1/4}(x_0).
\end{align}

Next we show that for each $x_0\in \Gamma\cap B'_{1/4}$, 
\begin{equation}\label{eq:asym_h0}
|h_0(x)-w_{1/2}(T_{x_0}x)|\leq C \left(c_\ast+\left\| \nabla'' g \right\|_{C^{\alpha}(B_{1/2}'')}\right) |x-x_0|^{\frac{1}{2}+\min\{\alpha,1-\frac{n+1}{p}\}}, \quad x\in  B_{1/2},
\end{equation}
where $T_{x_0}$ is defined as the following affine transformation in $\R^{n+1}$: $T_{x_0}(x)=R_{x_0}B_{x_0}^{-1}(x-x_0)$, where $B_{x_0}B_{x_0}^t=A(x_0)$ and $R_{x_0}\in SO(n+1)$ is such that $R_{x_0}B_{x_0}^t \nu_{x_0}=\bar c_{1,x_0} e_n$, $R_{x_0}B_{x_0}^t e_{n+1}=\bar c_{2,x_0} e_{n+1}$. A direct computation leads to 
\begin{align*}
(\bar{c}_{1,x_0})^2 &=  (R_{x_0} B_{x_0}^{t} \nu_{x_0}) \cdot (R_{x_0} B_{x_0}^{t} \nu_{x_0}) = \nu_{x_0} \cdot A(x_0) \nu_{x_0},\\
(\bar{c}_{2,x_0})^2 &= (R_{x_0} B_{x_0}^{t} e_{n+1}) \cdot (R_{x_0} B_{x_0}^{t} e_{n+1}) = a^{n+1, n+1}(x_0).
\end{align*}
In order to show \eqref{eq:asym_h0}, we note that by a similar estimate as \eqref{eq:barrier_est} with $s=0$, we deduce
\begin{align}
\label{eq:varyT}
|w_{1/2}(T_k x)-w_{1/2}(T_{x_0}x)|\leq C \left(c_\ast+\left\| \nabla'' g \right\|_{C^{\alpha}(B_{1/2}'')}\right) |x-x_0|^{\frac{1}{2}+\min\{\alpha,1-\frac{n+1}{p}\}},  
\end{align}
for all $x\in Q\in \mathcal{N}(Q_k)$. Instead of Lemma~\ref{lem:normals}, we use
\begin{align*}
|\nu_k-\nu_{x_0}|\leq C \left\| \nabla'' g \right\|_{C^{0,\alpha}(B_{1/2}'')} |x_k-x_0|^{\alpha},
\end{align*}
in this context (we recall that $\nu_k$ is the normal of $\Gamma$ at $x_k$, where $x_k\in \Gamma$ realizes the distance of the center $\hat x_k$ of $Q_k$ to $\Gamma$). \\
Thus, by using the properties (W1)-(W3) of the Whitney decomposition, we have
\begin{align*}
|h_0(x)-w_{1/2}(T_{x_0}x)|&\leq \sum_k \eta_k |w_{1/2}(T_k x)-w_{1/2}(T_{x_0}x)|\\
&\leq C\left(c_\ast+\left\| \nabla'' g \right\|_{C^{\alpha}(B_{1/2}'')}\right) |x-x_0|^{\frac{1}{2}+\min\{\alpha,1-\frac{n+1}{p}\}}.
\end{align*}
Writing $T_{x_0}$ out explicitly, we therefore obtain
\begin{align*}
&\left| h_0(x) - w_{1/2}\left(\frac{\nu_{ x_0}\cdot (x- x_0) }{(\nu_{ x_0} \cdot A( x_0) \nu_{ x_0})^{1/2}},\frac{x_{n+1}}{(a^{n+1,n+1}( x_0))^{1/2}}\right)\right|\\
 & \leq C\left(c_\ast+\left\| \nabla'' g \right\|_{C^{\alpha}(B_{1/2}'')}\right) |x-x_0|^{\frac{1}{2}+\min\{\alpha,1-\frac{n+1}{p}\}}.
\end{align*}
Using the above asymptotics of $h_0$ in \eqref{eq:dist1}, we arrive at the desired result. 
\end{proof}

The above asymptotic estimate leads to two immediate corollaries: On the one hand, it yields a Hopf principle and on the other hand, if combined with an interior elliptic estimate, it results in an upper growth bound.

\begin{cor}
\label{cor:lower_bound}
Assume that the conditions of Proposition~\ref{prop:asympt} are satisfied and that $0\in \Gamma$. 
Then, if $c_{\ast}, \bar\delta$ and $\|\nabla'' g\|_{C^{0,\alpha}(B_1'')}$ are sufficiently small depending on $n,p$ and $\alpha$, there is a dimensional constant $c_n>0$ such that
\begin{align*}
u(x)\geq c_n\dist(x,\Gamma)^{\frac{1}{2}}
\end{align*}
for $x\in B'_{1/2}\times (-1/2,1/2)\cap\{x| \dist(x,\Lambda)\geq \ell_0 \dist(x,\Gamma)\}$.
\end{cor}

\begin{cor}
\label{cor:upper_bound}
Assume that the conditions of Proposition~\ref{prop:asympt} are satisfied. If $c_\ast$, $\bar\delta$ and $\|\nabla'' g\|_{C^{0,\alpha}(B_1'')}$ are sufficiently small depending on $n,p$ and $\alpha$, then there exists a constant $\bar C=\bar C(n,p)>0$ such that for all $x\in B_{1/2}$
\begin{align*}
|u(x)|\leq \bar C M_0 \dist(x,\Gamma)^{\frac{1}{2}},
\end{align*}
where $M_0:=\sup _{B_1}|u|$. 
\end{cor}

\subsection{Local $3/2$-growth estimate, asymptotics and blow-up uniqueness}
\label{subsubsec:local32}
In this section we transfer the results from the previous section into the setting of the thin obstacle problem. Hence, this yields leading order asymptotics in appropriate cones, a Hopf principle and an optimal local growth estimate. Last but not least, we explain how the asymptotics and the growth estimates can be used to derive the uniqueness of the \emph{$3/2$-homogeneous blow-ups.}\\

We begin by transferring the general results of Section \ref{subsubsec:low_up_barrier} to the setting of solutions of the thin obstacle problem:

\begin{prop}[First order asymptotics]
\label{prop:wasympt}
Let $w$ be a solution of (\ref{eq:varcoef}) satisfying (A0), (A1), (A2), (A3), (A4). 
Assume that for some small positive constants $\epsilon_0$ and $c_\ast$
\begin{itemize}
\item[(i)] $\left\| w -  w_{3/2} \right\|_{C^1(B_{1}^+)} \leq \epsilon_0$,
\item[(ii)]$\|\nabla a^{ij}\|_{L^p(B_1^+)}\leq c_\ast$.
\end{itemize}
Then if $\epsilon_0$ and $c_\ast$ are chosen sufficiently small depending on $n,p$, there exists a function $b_e:\Gamma_w\cap B_{1/2}\rightarrow \R$, $b_e\in C^{0,\alpha}$ for some $\alpha\in (0,1-\frac{n+1}{p}]$, 
(where $\alpha$ can be chosen the same as the H\"older regularity of $\Gamma_w$ in Proposition~\ref{prop:C1a}), 
such that for all $x_0\in \Gamma_w\cap B_{1/2}$, we have the following asymptotic expansion for $\p_e w$, $e\in S^{n-1}\times\{0\}$:\\
\begin{multline}
\label{eq:asymptotics2}
\left| \p_{e} w(x)-b_e(x_0)w_{1/2}\left(\frac{(x- x_0) \cdot\nu_{ x_0}}{(\nu_{ x_0} \cdot A( x_0) \nu_{ x_0})^{1/2}},\frac{x_{n+1}}{(a^{n+1,n+1}( x_0))^{1/2}}\right)\right|\\
\leq C(n,p)\max\{\epsilon_0,c_\ast\}  \dist(x,\Gamma_w)^{\frac{1}{2}+\alpha},\quad x\in B_{1/4}(x_0).
\end{multline} 
Here $\nu_{ x_0}$ is the outer unit normal of $\Lambda_w$ at $ x_0$ and $A(x_0)=(a^{ij}(x_0))$. 
\end{prop}

Before coming to the proof of this result, we make the following observations:

\begin{rmk}\label{rmk:asym}
\begin{itemize}
\item[(i)] Similarly as in \eqref{eq:asymptotics2}, it is possible to obtain an asymptotic expansion for $\p_{n+1}w$: There exists a function $ b_{n+1}\in C^{0,\alpha}(\Gamma_w)$, such that for $ x_0\in \Gamma_w\cap B_{1/2}'$ and $x \in B_{1/4}(x_0)$
\begin{align*}
\left| \p_{n+1} w(x)-b_{n+1}(x_0)\bar w_{1/2}\left(\frac{(x- x_0) \cdot\nu_{ x_0}}{(\nu_{ x_0} \cdot A( x_0) \nu_{ x_0})^{1/2}},\frac{x_{n+1}}{(a^{n+1,n+1}(x_0))^{1/2}}\right)\right|\\
\leq C(n,p)\max\{\epsilon_0,c_\ast\}  \dist(x,\Gamma_w)^{\frac{1}{2}+\alpha},
\end{align*} 
where $\bar w_{1/2}(x)=-\Imm(x_n+ix_{n+1})^{1/2}$.
\item[(ii)] Later, we will see that $b_e$ and $b_{n+1}$ satisfy the following relation (c.f. \eqref{eq:curl} and \eqref{eq:defa} in Corollary~\ref{cor:wasympt}): There exists a $C^{0,\alpha}$ function $a:\Gamma_w \cap B_{1/2}\rightarrow \R$, $a>0$,  such that 
\begin{align*}
b_e(x_0)=\frac{3}{2}\frac{a(x_0)(e\cdot \nu_{x_0})}{(\nu_{x_0}\cdot A(x_0)\nu_{x_0})^{1/2}},\quad
b_{n+1}(x_0)=\frac{3}{2} \frac{a(x_0)}{(a^{n+1,n+1}(x_0))^{1/2}}.
\end{align*}
\end{itemize}
\end{rmk}

\begin{proof}[Proof of Proposition \ref{prop:wasympt}]
The proof of Proposition \ref{prop:wasympt} follows immediately from Proposition~\ref{prop:asympt}. More precisely, we first note that, by Proposition \ref{prop:C1a} if $c_\ast, \epsilon_0$ are sufficiently small depending on $n,p$, then (up to a rotation)
$$\Gamma_w\cap B_{1/2}=\{ x_n=g(x'')\}\cap B_{1/2},$$
where $g\in C^{1,\alpha}$,  $\|g\|_{\dot{C}^{1,\alpha}}\lesssim \max\{c_\ast,\epsilon_0\}$ and $\alpha$ only depends on $n$ and $p$. 
Let $v=\partial_e w$ and recall from the proof of Proposition~\ref{prop:Lip} that it satisfies the following divergence form equation in $B_1\setminus \Lambda_w$
\begin{align*}
\partial_i a^{ij}\partial_j v = \partial_i F^i,\quad F^i=-(\partial_e a^{ij})\partial_j w,
\end{align*}
with $\|\dist(\cdot,\Gamma_w)^{-\frac{1}{2}}\ln(\dist(\cdot,\Gamma_w))^{-2}F^i\|_{L^p(B_1)}\leq C(n,p)c_\ast$ (c.f. Lemma \lemfreebgrowth).
Moreover, $\p_e w$ satisfies the assumptions of Proposition \ref{prop:asympt} if $e$ is in a sufficiently large cone of directions, i.e. $e\in \mathcal{C}'_\eta (e_n)$ with $\eta>0$ sufficiently small. Thus applying Proposition~\ref{prop:asympt} we obtain the desired result.  Note that the exponent $\alpha\in (0,1-\frac{n+1}{p}]$ can be chosen to have the same value as the H\"older regularity exponent of $\Gamma_w$ in Proposition~\ref{prop:C1a}, since both of these come from the boundary Harnack inequality (Lemma~\ref{lem:boundHarn}). The result for general directions $e$ follows from Remark~\ref{rmk:asym} (ii). 
\end{proof}

As a corollary of the asymptotic expansion for $\nabla w$, it is now possible to obtain the asymptotic expansion for $w$ via integration along appropriate contours:

\begin{cor}
\label{cor:wasympt}
Under the assumptions of Proposition~\ref{prop:wasympt},
there exists a Hölder continuous function $a:\Gamma_w\rightarrow \R$, $a>0$, such that for $x_0 \in \Gamma_w \cap B'_{1/2}$ and all $x\in  B_{1/4}(x_0)$ with $x -x_0 \in \spa \{ \nu_{x_0}, e_{n+1}\}$ 
\begin{multline}
\label{eq:asymptotics}
\left|w(x)-a(x_0)w_{3/2}\left(\frac{(x- x_0) \cdot\nu_{ x_0}}{(\nu_{ x_0} \cdot A( x_0) \nu_{ x_0})^{1/2}},\frac{x_{n+1}}{(a^{n+1,n+1}( x_0))^{1/2}}\right)\right| \\
\leq C(n,p)\max\{\epsilon_0,c_\ast\}\dist(x,\Gamma_w)^{\frac{3}{2}+\alpha}.
\end{multline}
Here $\nu_{ x_0}$ is the outer unit normal of $\Lambda_w$ at $ x_0$ and $A(x_0)=(a^{ij}(x_0))$. \\
In particular, $w(x_0+rx)/r^{3/2}$ has a unique blow-up limit as $r\rightarrow 0$.
\end{cor}

\begin{proof} 
The asymptotic expansion of $w$ follows from the asymptotics for $\p_{\nu_{x_0}} w$, $\p_{n+1}w$ and an integration along an appropriate path in the plane spanned by $\nu_{x_0}, e_{n+1}$. Given $x_0\in \Gamma_w$, we write  
$$ x=x_0+t_1\nu_{x_0}+t_2 e_{n+1} \text{ for all }x\in \spa\{\nu_{x_0}, e_{n+1}\}.$$
In particular, $t_1=(x-x_0)\cdot \nu_{x_0}$ and $t_2=(x-x_0)\cdot e_{n+1}=x_{n+1}$. We consider the restriction, $\tilde{w}$, of $w$ onto the two-dimensional plane spanned by $\nu_{x_0}$ and $e_{n+1}$, i.e. 
$$\tilde{w}(t_1,t_2):=w(x_0+t_1\nu_{x_0}+t_2e_{n+1}).$$ 
Then, 
\begin{equation*}
\begin{split}
\p_{t_1}\tilde{w}(t_1,t_2)&=\p_{\nu_{x_0}} w(x_0+t_1\nu_{x_0}+t_2e_{n+1}),\\
\p_{t_2}\tilde{w}(t_1,t_2)&=\p_{e_{n+1}}w(x_0+t_1\nu_{x_0}+t_2e_{n+1}).
\end{split}
\end{equation*}
From the asymptotics derived in Proposition \ref{prop:wasympt}, we obtain 
\begin{equation}
\label{eq:expand1}
\begin{split}
\p_{t_1}\tilde{w}(t_1,t_2)&=b_{\nu_{x_0}}w_{1/2}(c_{1}^{-1}t_1, c_{2}^{-1}t_2)+g_1(t_1,t_2),\\
\p_{t_2}\tilde{w}(t_1,t_2)&=b_{n+1}\bar w_{1/2}(c_{1}^{-1} t_1, c_{2}^{-1}t_2)+g_2(t_1,t_2).
\end{split}
\end{equation}
Here
$$|g_1(t_1,t_2)|,\ |g_2(t_1,t_2)|\leq C(n,p)\max\{\epsilon_0,c_\ast\}(|t_1|^{1/2+\alpha}+|t_2|^{1/2+\alpha})$$
and  $b_{\nu_{x_0}}=b_{\nu_{x_0}}(x_0)$, $b_{n+1}=b_{n+1}(x_0)$. The constants  $c_{1}$, $c_2$ are (for fixed $x_0$) defined as
\begin{align*}
c_{1} = (\nu_{x_0} \cdot A(x_0) \nu_{x_0})^{\frac{1}{2}}, \quad
c_{2} =  (a^{n+1, n+1}(x_0))^{\frac{1}{2}}.
\end{align*}
Since $\frac{3}{2}(w_{1/2},\bar w_{1/2})=\nabla w_{3/2}$,
we can rewrite (\ref{eq:expand1}) as
\begin{equation}
\label{eq:expand2}
\begin{split}
\p_{t_1}\tilde{w}(t_1,t_2)&=\frac{2b_{\nu_{x_0}}c_1}{3}\p_{t_1}w_{3/2}(c_{1}^{-1}t_1, c_{2}^{-1}t_2)+g_1(t_1,t_2),\\
\p_{t_2}\tilde{w}(t_1,t_2)&=\frac{2b_{n+1}c_2}{3}\p_{t_2} w_{3/2}(c_{1}^{-1}t_1, c_{2}^{-1}t_2)+g_2(t_1,t_2).
\end{split}
\end{equation}
Since 
\begin{align*}
\int\limits_{\R^2} \p_{t_1}\tilde{w}\p_{t_2}\phi=\int\limits_{\R^2} \p_{t_2}\tilde{w}\p_{t_1}\phi \text{ for any } \phi\in C^\infty_c(\R^2),
\end{align*}
we have
\begin{align*}
\left(\frac{2b_{\nu_{x_0}}c_1}{3}-\frac{2b_{n+1}c_2}{3}\right)\int\limits_{\R^2} w_{3/2}(c_{1}^{-1}t_1, c_{2}^{-1}t_2) \p^2_{t_1t_2}\phi= \int\limits_{\R^2} \tilde{g}(t_1,t_2) \p^2_{t_1t_2}\phi.
\end{align*}
Here
\begin{align*}
\tilde{g}(t_1,t_2)&=\tilde{g}_1(t_1,t_2)-\tilde{g}_2(t_1,t_2)=-\int_0^{t_1}g_1(s,t_2)ds+\int_0^{t_2}g_2(t_1,s) ds.
\end{align*}
A direct computation then gives
\begin{align*}
|\tilde{g}(t_1,t_2)| &\leq  C(n,p)\max\{\epsilon_0,c_\ast\}(|t_1|^{3/2+\alpha}+|t_1||t_2|^{1/2+\alpha}\\
& \quad +|t_2||t_1|^{1/2+\alpha}+|t_2|^{3/2+\alpha}),
\end{align*}
which is of higher (than $3/2$) order in $(t_1,t_2)$. Thus necessarily we have
\begin{align}\label{eq:curl}
b_{\nu_{x_0}}c_{1}=b_{n+1}c_{2}.
\end{align}
Now let $\ell(s)$ be the path from $(0,0)$ to $(t_1,t_2)$:
$\ell(s)= (t_1 s, t_2 s)$, $s\in [0,1]$, 
and let 
$\widetilde{\ell}(s)= (c_{1}^{-1}t_1 s, c_{2}^{-1}t_2 s)$. Then
\begin{align*}
\left|\frac{d }{ds} \tilde{w}(\ell(s))- \frac{2b_{n+1}c_{2}}{3}\frac{d }{ds} w_{3/2}(\widetilde{\ell}(s)) \right| &\leq C(n,p)\max\{\epsilon_0,c_\ast\}|\tilde{\ell}(s)|^{1/2+\alpha}.
\end{align*}
Integrating from $0$ to $1$ leads to
$$\left| \tilde{w}(t_1,t_2) -\frac{2b_{n+1}c_2}{3}w_{3/2}(c_1^{-1}t_1,c_2^{-1}t_2)\right| \leq C(n,p)\max\{\epsilon_0,c_\ast\}(|t_1|^{3/2+\alpha}+|t_2|^{3/2+\alpha}).$$
Finally defining
\begin{align}\label{eq:defa}
a(x_0):=\frac{2 b_{n+1}(x_0)c_2}{3},
\end{align}
and recalling the explicit expression of $c_{1}$ and $c_{2}$ yields the asymptotics \eqref{eq:asymptotics} for $w$.\\
The fact that $a>0$ follows from Corollary~\ref{cor:lowerbound}.
\end{proof}

From Proposition~\ref{prop:wasympt} and Remark~\ref{rmk:asym} (i), we immediately obtain the following lower and upper bounds on the growth rate of $\p_e w$ away from the regular free boundary:

\begin{cor}[Lower bound]
\label{cor:lowerbound}
Under the assumptions of Proposition~\ref{prop:wasympt} there exist a constant $c_0=c_0(n,p)>0$ and a radius $r_0=r_0(n,p)$ such that
\begin{align*}
\p_ew(x)\geq c_0 \dist(x,\Gamma_w)^{\frac{1}{2}}, \quad x\in \{x\in B_{1/2}^+| \ \dist(x,\Lambda_w)\geq \ell_0 \dist(x,\Gamma_w)\},
\end{align*}
for any $e\in \{x\in S^{n-1} \times \{0\}| \ x_n\geq \frac{1}{2} |x''|\}$.
\end{cor}

\begin{rmk}
A uniform lower bound may not hold for $\p_e w$ in a full neighborhood of the origin. However, after subtracting an ``error term" or under higher regularity assumptions on the metric $a^{ij}$, it is possible to obtain a uniform lower bound in a full neighborhood of the origin (c.f. Section~\ref{sec:lower}).
\end{rmk}

\begin{cor}[Upper bound]
\label{cor:optgrowth}
Under the assumptions of Proposition~\ref{prop:wasympt} there exist a constant $\bar C=\bar C(n,p)$ and a radius $r_0= r_0(n,p)\in (0,1)$ such that for each $x_0\in \Gamma_w\cap B_{1/2}$ 
\begin{equation}
\label{eq:rescale_upper}
\sup_{B_r(x_0)}|\nabla w|\leq \bar C r^{1/2}, \quad 0<r\leq r_0.
\end{equation}
\end{cor}

\begin{rmk}
The above upper bound holds for solutions which are sufficiently close to $w_{3/2}$ (c.f. assumption (i) of Proposition~\ref{prop:wasympt}). In the next section we will remove this assumption and show a uniform upper bound for solutions $w$ with $\|w\|_{L^2(B_1^+)}=1$ (c.f. \eqref{eq:optimain}). 
\end{rmk}

\subsection{Optimal regularity of the solution}
\label{sec:opti}

In this section we exploit the comparison arguments from the previous section in combination with the Carleman estimates from \cite{KRS14}, to obtain the optimal regularity of solutions to the thin obstacle problem with coefficients $a^{ij}\in W^{1,p}$ with $p\in (n+1,\infty]$. In comparison to the results of Proposition \propalmost our main improvement here is a \emph{uniform, optimal} upper growth bound without losses (c.f. (\ref{eq:optimain})). This then allows us to remove the logarithmic losses from the regularity estimates from Proposition \propalmost (which were exemplified in the non-uniform constant $C(\gamma)>0$ and the only $C^{1,1/2-\epsilon}$ regularity of solutions to the variable coefficient problem in Proposition \propalmost) and to obtain the following optimal regularity estimates:

\begin{thm}
\label{thm:optimal_reg}
Let $w$ be a solution of (\ref{eq:varcoef}) in $B_{1}^+$ which satisfies the normalization condition (A0) and let $a^{ij}:B_{1}^+ \rightarrow \R^{(n+1)\times (n+1)}_{sym}$ be a $W^{1,p}$ tensor field with $p\in(n+1,\infty]$ satisfying (A1), (A2), (A3) and (A4). Then the following statements hold:
\begin{itemize}
\item[(i)] There exist a constant $C=C(n,p)$ and a radius $R_0=R_0(n,p, \left\| \nabla a^{ij}\right\|_{L^p(B_1)})$, such that  
\begin{equation}\label{eq:optimain}
\sup _{B_r(x_0)}|w|\leq C r^{3/2}\text{ for any } r\in (0,R_0) \text{ and } x_0\in \Gamma_w\cap B'_{\frac{1}{2}}.
\end{equation}
\item[(ii)] If $p\in(n+1,2(n+1)]$, it holds that $w\in C^{1,\gamma}(B_{1/2}^+)$ with $\gamma= 1-\frac{n+1}{p}$. Moreover, there exists a constant $C=C(n,p, \| a^{ij} \|_{ W^{1,p}(B_1^+)})$ (which remains bounded as $\gamma \nearrow 1/2$) such that
\begin{align*}
|\nabla w(x^1)- \nabla w(x^2)| \leq C|x^1 - x^2|^{\gamma} \mbox{ for all } x^1,x^2 \in B_{1/2}^+ . 
\end{align*}
\item[(iii)] If $p\in [2(n+1),\infty]$, then $w\in C^{1,\frac{1}{2}}(B_{1/2}^+)$. More precisely, there exists a constant $C=C(n,p, \| a^{ij} \|_{ W^{1,p}(B_1^+)})$ such that
\begin{align*}
|\nabla w(x^1)- \nabla w(x^2)| \leq C|x^1 - x^2|^{1/2} \mbox{ for all } x^1,x^2 \in B_{1/2}^+ . 
\end{align*}
\end{itemize}
\end{thm}

\begin{proof}
We will only show the growth estimate from part (i). With \eqref{eq:optimain} at hand, arguing similarly as for Proposition \propalmost, we then obtain the corresponding optimal regularity results (ii) and (iii). \\
By Lemma \lemgrowthimp and by the gap of the vanishing order (i.e. either $\kappa_{x_0} = 3/2$ or $\kappa_{x_0} \geq 2$), it suffices to show \eqref{eq:optimain} for $x_0\in \Gamma_{3/2}(w)\cap B'_{1/2}$. 
For simplicity we assume that $0\in \Gamma_{3/2}(w)$ and show the growth estimate (\ref{eq:optimain}) at $x_0=0$.\\

\emph{Step 1:} Suppose that $w$ is a solution to the thin obstacle problem which satisfies the assumptions (A0), (A1), (A2), (A4) with a metric $a^{ij}\in W^{1,p}(B_1^+)$ for some $p\geq p_0$, $p_0 \in (n+1,\infty)$. Suppose that $0\in \Gamma_{3/2}(w)$. We will show that for any $\delta>0$, there exists a constant $\epsilon=\epsilon(n,p, p_0, \delta)>0$, such that if 
\begin{itemize}
\item[(i)] $\|a^{ij}-\delta^{ij}\|_{L^\infty(B_1^+)}\leq \epsilon$,
\item[(ii)] $\|x\cdot \nabla w - \frac{3}{2} w\|_{L^2(A_{1/2,1}^+)}\leq \epsilon$, 
\end{itemize}
then there exists a nontrivial $3/2$-homogeneous global solution $w_0$ such that $\|w-w_0\|_{C^1(B_{3/4}^+)}\leq \delta$.

Suppose that this were wrong, then there existed a parameter $\delta_0>0$, a sequence $\epsilon_k\rightarrow 0$ and a sequence, $w_k$, of solutions to the thin obstacle problem   
\begin{align*}
\p_i a^{ij}_k \p_j w_k = 0 \text{ in } B_1^+,\\
w_k\geq 0, \ -\partial_{n+1} w_k\geq 0, \ w_k \p_{n+1}w_k=0 \text{ on } B'_1,
\end{align*}
which satisfy the assumptions (A0), (A1), (A2), (A3), (A4), such that 
\begin{itemize}
\item[(i)] $a^{ij}_k \in W^{1,p}(B_1^+)$ for some $p\geq p_0$ with $p_0$ as above, i.e. $p_0\in (n+1,\infty]$ and $p_0$ being independent of $k$,
\item[(ii)] $\|a^{ij}_k-\delta^{ij}\|_{L^\infty(B_1^+)}\leq \epsilon_k$,
\item[(iii)] $\|x\cdot \nabla w_k-\frac{3}{2}w_k\|_{L^2(A_{1/2,1}^+)}\leq \epsilon_k$, 
\end{itemize}
but $\|w_k- w_0\|_{C^1(B_{3/4}^+)}\geq \delta_0$ for any nontrivial $3/2$-homogeneous global solution $w_0$.

By the fundamental theorem of calculus, there exists an absolute constant $C>0$ with
\begin{align*}
\|w_k\|_{L^2(A_{7/8,1}^+)}&\leq C \|w_k\|_{L^2(A_{1/2,7/8}^+)}+ C\| |x|^{3/2}x\cdot \nabla (|x|^{-3/2} w_k)\|_{L^2(A_{1/2,1}^+)}\\
&\leq C\|w_k\|_{L^2(A_{1/2,7/8}^+)}+ C\epsilon_k \|w_k\|_{L^2(B_1^+)}\\
&\leq C\|w_k\|_{L^2(B_{7/8}^{+})}+C\epsilon_k\|w_k\|_{L^2(A_{7/8,1}^+)}+C\epsilon_k\|w_k\|_{L^2(B_{7/8}^+)}.
\end{align*}
Hence, if $\epsilon_k>0$ is sufficiently small, then 
$$\|w_k\|_{L^2(A_{7/8,1}^+)}\leq 2C \|w_k\|_{L^2(B_{7/8}^+)},$$
which implies that
\begin{align}\label{eq:doubling22}
1=\|w_k\|_{L^2(B_1^+)}\leq (2C+1)\|w_k\|_{L^2(B_{7/8}^+)}.
\end{align}
By the interior $H^1$ estimate and by Proposition \propalmost, $w_k$ is uniformly bounded in $H^1(B_{7/8}^+)$ and $C^{1,\gamma_0}(B_{7/8}^+)$ for some $\gamma_0=\gamma_0(n,p_0)>0$. Thus, up to a subsequence
\begin{align*}
& w_k\rightarrow \bar{w} \text{ weakly in }H^1(B_{7/8}^+), \\
& w_k\rightarrow \bar{w} \text{ strongly in }L^2(B_{7/8}^+),\\
& w_k\rightarrow \bar{w} \text{ in } C^{1}_{loc}(B_{7/8}^+),
\end{align*}
where $\bar{w}$ (by the contradiction assumption (i)) solves the constant coefficient thin obstacle problem
\begin{align*}
\Delta \bar{w} =0 \text{ in } B_{7/8}^+,\\
\bar{w} \geq 0, \ -\partial_{n+1} \bar{w} \geq 0, \ \bar{w} \p_{n+1}\bar{w}=0 \text{ on } B'_{7/8}.
\end{align*}
Furthermore, by the contradiction assumption (ii), $\bar{w}$ is homogeneous of degree $3/2$ in $B_{7/8}^+$. By analyticity, $\bar{w}$ is a $3/2$-homogeneous global solution in $\R^{n+1}$. Moreover, by \eqref{eq:doubling22} and by the strong convergence in $L^2(B_{7/8}^+)$,
$$\|\bar{w}\|_{L^2(B_{7/8})}\geq (2C+1)^{-1},$$
thus $\bar{w}$ is nontrivial. Therefore, we have found a nontrivial $3/2$-homogeneous global solution $\bar{w}$ such that (up to a subsequence) $\|w_k-\bar{w}\|_{C^{1}(B_{3/4}^+)}\rightarrow 0$ as $k\rightarrow \infty$. This is a contradiction.\\

\emph{Step 2:} We show that there exists $\epsilon=\epsilon(n,p)$, such that if for some $r\in (0,R_0)$ with $R_0=R_0(n,p,\|\nabla a^{ij}\|_{L^p(B_1^+)})$
\begin{equation}\label{eq:opti2}
\|x\cdot \nabla w-\frac{3}{2}w\|_{L^2(A_{r/2,r}^+)}\leq \epsilon \|w\|_{L^2(B_{r}^+)},
\end{equation} 
then there exist constants $\mu_0=\mu_0(n,p)\in (0,1)$ and 
$C_1=C_1(n,p)>1$ such that 
\begin{align*}
|w(x)|\leq C_1 r^{-\frac{3}{2}} r^{- \frac{n+1}{2}} \left\| w \right\|_{L^2(B_{r}^+)} |x|^{\frac{3}{2}} \mbox{ for } x\in B_{\mu_0 r}^{+}.
\end{align*}

For this we first notice that \eqref{eq:opti2} can be rewritten in terms of $L^2$ normalized rescalings
\begin{equation}
\label{eq:L2blowup}
w_r(x) := \frac{w(rx)}{r^{- \frac{n+1}{2}} \left\| w \right\|_{L^2(B_{r}^+)}},
\end{equation}
to yield
\begin{align}\label{eq:opti3}
\|x\cdot \nabla w_r - \frac{3}{2} w_r\|_{L^2(A_{1/2,1}^+)}\leq \epsilon \text{ for some } r\in (0,R_0).
\end{align}
Here $w_r$ is a solution to the thin obstacle problem with coefficients $a^{ij}_r (x):=a^{ij}(rx)$, which satisfy $\|\nabla a^{ij}_r\|_{L^p(B_1^+)}= r^{1-\frac{n+1}{p}}\|\nabla a^{ij}\|_{L^p(B_r^+)}$. \\
Secondly, by Proposition \ref{prop:C1a} in Section~\ref{sec:boundary}, there exist constants $\epsilon_0=\epsilon_0(n,p)$ and $c_{\ast}=c_{\ast}(n,p)$, which are less than some universal constant, such that if $\|w_r-w_{3/2}\|_{C^1(B_{3/4}^+)}<\epsilon_0$ with $w_{3/2}(x)=c_n\Ree(x_n+ix_{n+1})^{3/2}$ (possibly after a rotation) and $\| \nabla a^{ij} \|_{L^p(B_{3/4}^+)}\leq c_{\ast}$, then the free boundary of $w_r$ is a $C^{1,\alpha}$ graph in $B'_{1/2}$. Moreover, its $\dot{C}^{1, \alpha}$ Hölder norm is uniformly bounded, depending only on $n,p$. The Hölder exponent, $\alpha$, only depends on $n,p$ (c.f. Remark \ref{rmk:alpha}).
Thus, by Corollary~\ref{cor:optgrowth}, there exist constants $\mu_0 $ and $C_1$ depending on $n,p$, such that
$$|w_{r}(x)|\leq C_1|x|^{3/2} \mbox{ for any } x\in B_{\mu_0}^+.$$ 
Scaling back then yields
\begin{align}
\label{eq:rescaled_opt}
| w(x) | \leq C_1 r^{- \frac{3}{2}} r^{- \frac{n+1}{2}}\left\| w \right\|_{L^2(B_{r}^+)}  |x|^{\frac{3}{2}} \text{ for any } x\in B_{\mu_0r}^+.
\end{align}
To complete the proof of step 2, we apply step 1 to $w_r$ with $\delta=\epsilon_0$. This determines the parameter $\epsilon=\epsilon(n,p)$ (if $R_0=R_0(n,p,\|\nabla a^{ij}\|_{L^p(B_1^+)})$ is chosen such that $\|\nabla a^{ij}_r\|_{L^p(B_1^+)}\leq c_\ast(n,p)$ for any $r\in (0,R_0]$). \\

\emph{Step 3:} Fix $\epsilon>0$ as in step 2. Let $r_1\in (0,R_0]$ be the largest radius such that \eqref{eq:opti2} holds (we remark that the existence of an $r_1>0$ such that (\ref{eq:opti2}) is satisfied, follows from the fact that the $L^2$ rescaling has a $3/2$ homogeneous blow-up along a certain subsequence, c.f. Proposition \lemhomo). Then for any $r\in [r_1,R_0]$,
\begin{align}
\label{eq:anti2}
\|x\cdot \nabla w-3/2w\|_{L^2(A_{r/2,r)}}\geq \epsilon \|w\|_{L^2(A_{r/2,r}^+)}.
\end{align}
In this regime we discuss two cases:
\begin{itemize}
\item{Case 1:} $\|w\|_{L^2(B_{ r_1}^+)}< r_1^{\frac{3}{2}+\frac{n+1}{2}+\frac{\epsilon}{2}}$. 
In this case
\begin{align*}
r_1^{- \frac{3}{2}} r_1^{- \frac{n+1}{2}}\left\| w \right\|_{L^2(B_{r_1}^+)} \leq r_1^{\frac{\epsilon}{2}}.
\end{align*}
Thus, by (\ref{eq:rescaled_opt}) in step 2, there exists $C=C(n,p,\|\nabla a^{ij}\|_{L^{p}})$ such that
$$|w(x)|\leq C r^{\frac{\epsilon}{2}}_1 |x|^{3/2} \text{ for any } x\in B_{\mu_0r_1}^+.$$
Recalling the parameter dependence of $\mu_0$, we obtain 
\begin{equation}
\label{eq:est_imp}
|w(x)|\leq C |x|^{\frac{3}{2} + \frac{\epsilon}{2}}\text{ for any } x\in A_{\mu_0 r_1, r_1}^+.
\end{equation}
For any $r \in (r_1,R_0)$, we then use Corollary \corconsequenceCarl (note that this estimate has a logarithmic loss so that it is necessary to use the slightly improved estimate (\ref{eq:est_imp}) in the application of the Carleman inequality; alternatively it would have been possible to use the Carleman estimate from Lemma \lemnewCarl, for instance in the version of Remark \rmkten which does not need the $\epsilon$-improvement). Thus, we obtain 
\begin{align*}
\|w\|_{L^2(A_{r/2,r}^+)}\leq C r^{\frac{3}{2}+\frac{n+1}{2}} \mbox{ with } C=C(n,p,\|\nabla a^{ij}\|_{L^{p}})
\end{align*}
for all $r\in(r_1, R_0)$.
\item{Case 2:} $\|w\|_{L^2(B_{r_1}^+)}\geq r_1^{\frac{3}{2}+\frac{n+1}{2}+\frac{\epsilon}{2}}$.
In this case statement (i) from Lemma \lemalmosthom (with the two radii $r_1$ and $R_0$) would either imply that $r_1$ is bounded from below, i.e. $r_1\geq C$, for some constant $C>0$ which only depends on $n,p, \left\| \nabla a^{ij}\right\|_{L^p}$, or it would imply a contradiction to the maximal choice of $r_1$. 
\end{itemize}
In summary, we either have a uniform lower bound for $r_1$ (depending only on $n,p, \left\| \nabla a^{ij}\right\|_{L^p}$) or we obtain that the growth estimate \eqref{eq:optimain} holds for all $r\in (0,R_0)$. \\

Combining steps 1-3, we have thus obtained a radius $R_0 = R_0(n,p, \left\| \nabla a^{ij} \right\|_{L^p})$ and a constant $C=C(n,p)$ such that $\sup_{B_r}|w|\leq C r^{3/2}$ for any $r\in (0, R_0)$. This, together with Proposition \propalmost, completes the proof of Theorem \ref{thm:optimal_reg}.  
\end{proof}

\subsection{Improved lower bounds}
\label{sec:lower}
In this section we improve our lower bounds from Corollary \ref{cor:lowerbound} by making them uniform: While in general a uniform version of Corollary \ref{cor:lowerbound} does not hold for the full function $w$, after a suitable splitting, it is possible to show that the leading order term satisfies a uniform lower bound. More precisely, let  $a^{ij}:B_1^+ \rightarrow \R^{(n+1)\times (n+1)}_{sym}$ be a uniformly elliptic $W^{1,p}$, $p\in (n+1,\infty]$, metric, let $f\in L^p(B_1^+)$ and let $w$ be a solution of 
\begin{equation}
\label{eq:inhom}
\begin{split}
\p_{i} a^{ij} \p_j w & = f \mbox{ in } B_1^+,\\
\p_{n+1}w \leq 0, \ w \geq 0, \ w\p_{n+1}w&=0 \mbox{ in } B_1',
\end{split}
\end{equation}
satisfying the normalization conditions (A0)-(A4) from Section \ref{sec:prelim} and conditions (i), (ii) from Proposition \ref{prop:wasympt}.\\
As in Remark \ref{rmk:ref_ext}, we extend $w$, the metric $a^{ij}$ as well as $f$ from $B_1^+$ to $B_1 \setminus \Lambda_w$ and further to $\R^{n+1}$. We now split $w$ into two components $$w=u+\tilde{u},$$ where $\tilde{u}$ solves
\begin{align}
\label{eq:split1}
a^{ij}\p_{ij}\tilde{u}-\dist(x,\Gamma_w)^{-2}\tilde{u}=f - (\p_i a^{ij})\p_j w \text{ in } \R^{n+1}\setminus \Lambda_w, \quad \tilde{u}=0\text{ on }\Lambda_w,
\end{align}
and the function $u$ solves 
\begin{align}
\label{eq:split2}
a^{ij}\p_{ij}u=-\dist(x,\Gamma_w)^{-2}\tilde{u}\text{ in } \R^{n+1}\setminus \Lambda_w, \quad u=0\text{ on } \Lambda_w.
\end{align}
As explained in the passage following Remark \ref{rmk:ref_ext} in Section \ref{subsec:Lip} the intuition is that $\tilde{u}$ is a ``controlled error'' and that $u$ captures the essential behavior of $w$. Moreover, as we will see later, $u$ is $C^{2,1-\frac{n+1}{p}}(K)$ regular for any $K\Subset \R^{n+1}\setminus \Gamma_w$.

\begin{lem}[Positivity]
\label{lem:lower1'}
Let $f\in L^p(B_1^+)$, $a^{ij}\in W^{1,p}$ with $p\in (2(n+1),\infty]$, and $u, w$ be as above. Then we have that $u\in C^{2,1-\frac{n+1}{p}}(K)\cap C^{1,\frac{1}{2}}(B_{1/2})$ for any $K\Subset B_1\setminus \Gamma_w$. Moreover, for $e\in \mathcal{C}'_\eta(e_n)$, $\p_eu $ satisfies the lower bound
\begin{align*}
\p_{e}u (x) \geq c\dist(x,\Lambda_w)\dist(x,\Gamma_w)^{-\frac{1}{2}}, \quad c>0.
\end{align*}
\end{lem}

\begin{rmk}
If $f=0$, the statement of Lemma \ref{lem:lower1'} can be improved to hold for metrics $a^{ij}\in W^{1,p}$ with $p\in(n+1,\infty]$.
\end{rmk}

\begin{proof}
Let $\tilde{u}$ be defined as in (\ref{eq:split1}). By the refined estimate from Remark \ref{rmk:impv1}, we infer that
\begin{align}
\label{eq:auxv}
|\tilde{u}(x)|\lesssim c_\ast  \dist(x,\Lambda_w)\dist(x,\Gamma_w)^{1-\frac{n+1}{p}}.
\end{align}
Moreover, combined with the $C^{1,1-\frac{n+1}{p}}$ regularity of $\tilde{u}$ away from $\Gamma_w$, we obtain the up to $\Gamma_w$ regularity of $\tilde{u}$, i.e. $\tilde{u}\in C^{1,1-\frac{n+1}{p}}(B_1)$ and $\|\tilde{u}\|_{C^{1,1-\frac{n+1}{p}}(B_1^+)}\leq C c_\ast$. Thus, applying the classical elliptic estimates to $u$ away from $\Gamma_w$, we deduce that $u\in C^{2,1-\frac{n+1}{p}}(K)$ for any $K\Subset B_1\setminus \Gamma_w$.  Recalling the decay estimate of $w$ (c.f. Theorem \ref{thm:optimal_reg}) and combining it with the estimate \eqref{eq:auxv} for $\tilde{u}$, yields that $u=w-\tilde{u}$ satisfies
\begin{align}\label{eq:upper_u}
|u(x)|\leq C_0 \dist(x,\Gamma_w)^{\frac{3}{2}}.
\end{align}
This together with the $W^{3,p}$ interior estimate gives that $u\in C^{1,\frac{1}{2}}(B_1)$.
Next we show the lower bound of $\p_e u$. This follows from an application of the comparison principle from Proposition \ref{prop:nondeg}.
First we note that at $B_1\cap \{|x_{n+1}|=\ell_0\}$, where $\ell_0=\frac{1}{\sqrt{4n}}>0$, $\p_e u$ is non-degenerate in the sense that 
\begin{align}\label{eq:lower_u}
\p_e u(x)\geq c_n \dist(x,\Gamma_w)^{\frac{1}{2}},\quad e\in \mathcal{C}'_\eta(e_n).
\end{align}
Indeed, for $x\in B_1\cap\{|x_{n+1}|=\ell_0\}$ we have $\dist(x,\Gamma_w)\sim \dist(x,\Lambda_w)$, where by \eqref{eq:auxv} the function $\p_e\tilde{u}$ satisfies
\begin{align}
\label{eq:lower_ua}
|\p_e\tilde{u}(x)|\leq c_\ast \dist(x,\Gamma_w)^{1-\frac{n+1}{p}}.
\end{align}
On the other side, the asymptotics from Proposition~\ref{prop:wasympt} yield that $\p_ew(x)\geq c \dist(x,\Gamma_w)^{1/2}$ for $x\in B_1\cap\{|x_{n+1}|=\ell_0\}$. As $p>2(n+1)$,  $\dist(x,\Gamma_w)^{1/2}$ dominates and thus $\p_e u$ inherits the lower bound  of $\p_e w$ on $x\in B_1\cap\{|x_{n+1}|=\ell_0\}$.\\
Next we consider the equation of $\p_e u$ (in non-divergence form):
\begin{align*}
a^{ij}\p_{ij}\p_eu &= -(\p_ea^{ij})\p_{ij}u-\p_e(\dist(x,\Gamma_w)^{-2}\tilde{u})\\
&=:H+\p_i G^i.
\end{align*}
The functions $H$, $G^i$  satisfy
\begin{align*}
&\dist(x,\Gamma_w)^{\frac{1}{2}}H\in L^p(B_1),\\
&\dist(x,\Gamma_w)^{\frac{n+1}{p}} G^i\in L^{\infty}(B_1).
\end{align*}
Here the first estimate follows from the $W^{1,p}$ regularity of $a^{ij}$ and the pointwise interior estimates for $\p_{ij}u$; the second estimate is a direct consequence of the previously derived bound (\ref{eq:auxv}) for $\tilde{u}$.
Furthermore, in $B_1$ we have $\p_e u=\p_ew-\p_e\tilde{u}\geq -c_\ast$. Thus, the assumptions of Proposition \ref{prop:nondeg} are satisfied (note that Proposition \ref{prop:nondeg} also holds for non-divergence form equation) and we obtain the desired lower bound for $\p_e u$ in the region $\{x\in B_{1}| \ \dist(x,\Gamma_w)\sim \dist(x,\Lambda_w) \}$. \\
In the end we remark that this lower bound holds in the whole ball. Indeed, this follows once more from a comparison argument: Let $h=c_s h_s+h_0$, where $h_s$ is the barrier function from Proposition \ref{prop:barrier} and $h_0$ is the barrier from Lemma \ref{lem:lowerbarrier}. Choosing $s=s(\alpha,n,p)$ appropriately yields that $h$ satisfies the conditions (i) and (ii) of Proposition \ref{prop:nondeg}. Moreover, by using an up to $\Lambda_w$ estimate for $q(x)$ as well as the fact that $h_0$ dominates $h_s$ (for an appropriately chosen constant $c_s$), we obtain
\begin{align}\label{eq:barrier}
h(x)\geq c\dist(x,\Lambda_w)\dist(x,\Gamma_w)^{-\frac{1}{2}}.
\end{align}
Relying on the test function 
\begin{align*}
u(x):=\p_e u(x)+|x'-(x_0)'|^2-2^{-8}h(x)-2nx_{n+1}^2,
\end{align*}
and introducing a similar splitting as in Proposition \ref{prop:nondeg} then yields
\begin{align*}
\p_e u(x)\geq ch(x)-c_\ast \dist(x,\Lambda_w)\dist(x,\Gamma_w)^{-\frac{n+1}{p}}.
\end{align*}
The second term on the RHS originates from the estimate of the auxiliary function $u_1$ (from a similar splitting as in the proof of Proposition \ref{prop:nondeg}), where we have exploited the Lipschitz regularity of $G^i$ away from $\Gamma_w$ (indeed, without this observation, the error estimate would not suffice to absorb the error contribution, c.f. Remark \ref{rmk:errorLip} below). 
Recalling the lower bound for $h$ (c.f. \eqref{eq:barrier}) implies the desired lower bound for $\p_e u$.
\end{proof}

\begin{rmk}
\label{rmk:errorLip} 
We note that without the Lipschitz regularity of $G^i$, we would only have had error estimates of the form
\begin{align*}
\p_eu(x)\geq ch (x)-c_\ast\dist(x,\Lambda_w)^{1-\frac{n+1}{q}}\dist(x,\Gamma_w)^{-\frac{n+1}{p} + \frac{n+1}{q}}
\end{align*}
for an arbitrary value $q\in(n+1,\infty)$. Clearly, this would not have sufficed for our absorption argument. Thus, the key improvement here consists in the observation that $G^i$ is actually a Lipschitz function away from $\Gamma_w$. As the estimates, which yield the improved bounds in Proposition \ref{prop:v1}, are interior estimates around $\Lambda_w$, this regularity away from $\Gamma_w$ suffices for our purposes. 
\end{rmk}

Assuming more regularity on the metric, the lower bound from Lemma \ref{lem:lower1'} can be further improved:

\begin{lem}[Positivity']
\label{lem:lower1}
Let $a^{ij}:B_1^+ \rightarrow \R^{(n+1)\times (n+1)}_{sym}$ be a tensor field that satisfies the normalization conditions (A1)-(A4) from Section \ref{sec:prelim}, condition (ii) from Proposition \ref{prop:wasympt} and in addition is $C^{1,\gamma}$ regular for some $\gamma \in (0,1)$. Let $w:B_1^+ \rightarrow \R$ be a solution of the thin obstacle problem with metric $a^{ij}$ and assume that it satisfies the normalizations (A0) from Section \ref{sec:prelim} and (i) from Proposition \ref{prop:wasympt}. Then there exists $\eta= \eta(n)>0$ such that for $e\in \mathcal{C}'_\eta(e_n)$
\begin{align}
\label{eq:lower1}
\p_ew(x)\geq c\dist(x,\Lambda_w)\dist(x,\Gamma_w)^{-\frac{1}{2}}.
\end{align}
\end{lem}

\begin{proof}
We only sketch the proof here. For each $x_0\in B_{1/2}^+$ we consider
\begin{align*}
u(x):=\p_ew(x)+|x'-(x_0)'|^2-2^{-8}h(x)-2nx_{n+1}^2,
\end{align*}
where $h$ is as in the proof of Lemma \ref{lem:lower1'}.
Next we split $u=u_1+u_2$, where $u_1$ solves
\begin{align*}
&\p_i a^{ij}\p_j u_1-K\dist(x,\Gamma_w)^{-2}u_1=\p_iF^i+g,\\
&\text{ where } F^i=-\p_ea^{ij}\p_jw, \quad g=2(x_j-(x_0)_j)\p_ia^{ij}-4nx_{n+1}\p_ja^{n+1,j},
\end{align*}
with $K=K(n)$ sufficiently large. Hence, $u_2$ solves
\begin{align*}
&\p_ia^{ij}\p_j u_2=\tilde{g}-K\dist(x,\Gamma_w)^{-2}u_1,\\
&\text{ where } \tilde{g}=-2^{-8}\p_ia^{ij}\p_jh +2a^{ii}-4na^{n+1,n+1}.
\end{align*}
We apply Proposition \ref{prop:v1} to $u_1$. Since $a^{ij}\in C^{1,\gamma}$ and $w\in C^{1,1/2}$, we have $\dist(x,\Gamma_w)^{-1/2} F^i\in L^\infty$ and $g\in L^\infty$. By Remark~\ref{rmk:impv1}, for any $s\in (n+1,\infty)$ 
\begin{equation}
\label{eq:aux}
\begin{split}
|u_1(x)|&\leq C_0 \dist(x,\Lambda_w)\dist(x,\Gamma_w)^{-\frac{n+1}{s}+\frac{1}{2}}\cdot\\
&\cdot\left(\|F^i\dist(\cdot,\Gamma_w)^{-\frac{1}{2}}\|_{L^s(B_1^+)}+\|g\dist(\cdot,\Gamma_w)^{1-\frac{n+1}{s}-\frac{1}{2}}\|_{L^{s/2}(B_1^+)}\right).
\end{split}
\end{equation}
For $u_2$ we apply a comparison argument. By the argument from the proof of Proposition \ref{prop:nondeg} we obtain that $u_2\geq 0$. Thus $u=u_1+u_2\geq u_1$. Evaluating at $x_0$ yields that $\p_ew(x_0)\geq 2^{-8}h(x_0)-u_1(x_0)$. Since $x_0$ is arbitrary in $B_{1/2}^+$ and by using the bound from \eqref{eq:aux} for $u_1$, we infer that 
\begin{align}\label{eq:dew}
\p_ew(x)\geq 2^{-8} h (x)-u_1(x)\geq 2^{-8}h(x)-c_\ast \dist(x,\Lambda_w).
\end{align}
Here we have used that $s>2(n+1)$ can be chosen to be sufficiently large and $\|Da^{ij}\|_{L^\infty}\leq c_\ast$. Recalling the estimate for $h$ in \eqref{eq:barrier}, we obtain the desired estimate for $\p_ew$ in \eqref{eq:lower1}.
\end{proof}

\begin{rmk}
We emphasize that without an additional splitting step, we cannot hope for an analogous result for a $W^{1,p}$, $p\in(2(n+1),\infty]$, metric $a^{ij}$. This is due to the fact that by Remark~\ref{rmk:impv1} for $u_1$ we only have  
\begin{align*}
|u_1(x)|&\leq C_0 \dist(x,\Lambda_w)^{1-\frac{n+1}{p}}\dist(x,\Gamma_w)^{\frac{1}{2}}\cdot\\
&\cdot\left(\|F^i\dist(\cdot,\Gamma_w)^{-\frac{1}{2}}\|_{L^p(\R^{n+1})}+\|g\dist(\cdot,\Gamma_w)^{1-\frac{n+1}{p}-\frac{1}{2}}\|_{L^{p/2}(\R^{n+1})}\right)\\
&\lesssim c_\ast \dist(x,\Lambda_w)^{1-\frac{n+1}{p}}\dist(x,\Gamma_w)^{\frac{1}{2}}.
\end{align*}
The bound by $\dist(x,\Lambda_w)^{1-\frac{n+1}{p}}$ is optimal in general, due to the divergence right hand side $\p_iF^i$ with $F^i\in L^p$. Thus, the resulting bound cannot be absorbed into $h^-$ in the estimate \eqref{eq:dew}.
\end{rmk}

\section[Perturbations]{Perturbations}
\label{sec:pert}

In this section we consider variants of the variable coefficient thin obstacle problem. Examples are settings in which the obstacle (c.f. Section \ref{sec:nonfl}) or the underlying manifold (c.f. Section \ref{sec:nfboundaries}) is not flat. Moreover, it is also possible to deal with interior thin obstacles (c.f. Section \ref{sec:int_obst}) and inhomogeneities in the equations (c.f. Section \ref{sec:inhomo}). We show that under suitable conditions on the metric and obstacles, it is possible to recover the regularity results from the flat setting. Instead of repeating all the necessary arguments, we only point out the main difficulties and differences with respect to the flat problem which was discussed in \cite{KRS14} and Sections \ref{sec:boundary}-\ref{sec:optimal} of the present article.\\
While treating the cases of inhomogeneities, non-flat obstacles and boundaries and interior obstacles separately in order to stress the essential aspects of the respective situation, we emphasize that it is possible to combine these results into a setting involving several/all of these features, e.g. a non-flat obstacle and a non-flat hypersurface.

\subsection{Inhomogeneities}
\label{sec:inhomo}

Exploiting the scaling of the Carleman inequality, it is possible to deal with inhomogeneous thin obstacle problems along similar lines as in Sections 2-4 in \cite{KRS14} and Sections \ref{sec:boundary}-\ref{sec:optimal} in this paper. Here the main observation is that both for the Carleman estimate and the comparison arguments in Sections \ref{sec:boundary} and \ref{sec:optimal} the inhomogeneity is a ``lower order'' contribution. 

\begin{prop}[Inhomogeneous thin obstacle problem]
\label{prop:inhomo}
Let \\
$a^{ij}: B_1^+ \rightarrow \R^{(n+1)\times (n+1)}$ be a $W^{1,p}(B_1^+)$, $p\in(n+1,\infty]$, metric satisfying (A1), (A2), (A4). Suppose that $f\in L^q(B_1^+)$ for some $q\in (n+1,\infty]$. Assume that $w\in H^1(B_1^+)$ is an $L^2$ normalized solution to the thin obstacle problem 
\begin{equation}
\label{eq:varcoef_f}
\begin{split}
\p_i a^{ij} \p_j w = f \mbox{ in } B_{1}^+,\\
w \geq 0, \ -a^{n+1,j}\p_j w \geq 0, \ w (a^{n+1, j}\p_j w) = 0 \mbox{ in } B_1'.
\end{split}
\end{equation}

Then, the following statements hold:
\begin{itemize}
\item[(i)] The solution $w$ has the following growth estimate: There exists $C>0$ such that
\begin{align*}
|w(x)|\leq C\left\{
\begin{array}{ll}
\dist(x,\Gamma_w)^{2-\frac{n+1}{q}}(\ln (\dist(x,\Gamma_w)))^2 & \mbox{ if } n+1 < q< 2(n+1), \\
\dist(x, \Gamma_w)^{\frac{3}{2}}(\ln (\dist(x,\Gamma_w)))^2 & \mbox{ if } q \geq 2(n+1).
\end{array}
\right.
\end{align*}
The constant $C>0$ depends on $n,p, q, \|f\|_{L^q}, \|\nabla a^{ij}\|_{L^p}, \left\| w \right\|_{L^2(B_1^+)}$.
\item[(ii)] If additionally $a^{ij}\in W^{1,p}(B_1^+)$ and $f\in L^p(B_1^+)$ with $p\in (2(n+1),\infty]$, then $w$ has the optimal Hölder regularity: 
$$w \in C^{1,1/2}(B_{1/2}^+).$$
\item[(iii)] Under the same assumptions on $a^{ij}$ and $f$ as in (ii) and assuming that $0\in \Gamma_{3/2}(w)$, there exist a radius $\rho=\rho(w,f)>0$, a parameter $\alpha \in (0,1]$ and a $C^{1,\alpha}$ function $g$ such that (potentially after a rotation) 
$$\Gamma_{3/2}(w) \cap B_{\rho}' = \{x = (x'',x_n,0)| \ x_{n}= g(x'') \mbox{ for } x''\in B_{\rho}'' \}.$$
\end{itemize}
\end{prop}

\begin{rmk}
\label{rmk:reg_est}
Since $w$ solves a linear elliptic equation away from the free boundary, we can obtain an interior regularity result for $w$ by  combining the growth estimate (i) with a local gradient estimate. The proof is standard (see e.g. Proposition~\propalmost) and yields
\begin{align*}
&|\nabla w(x) - \nabla w(y) |\\
&\leq  \left\{
\begin{array}{ll}
C(q)|x-y|^{1-\frac{n+1}{q}}\ln^2 (|x-y|) & \mbox{ if } n+1 < q< 2(n+1), \ q \leq p,\\
C(p,q)|x-y|^{1-\frac{n+1}{p}} & \mbox{ if } n+1 < p < q< 2(n+1),\\
C |x-y|^{\frac{1}{2}}\ln^2(|x-y|) & \mbox{ if }\min\{p, q\} \geq 2(n+1),
\end{array}
\right.
\end{align*}
for all $x,y \in B_{1/2}^+$.
We have 
\begin{align*}
C(q) \rightarrow + \infty \mbox{ as } q\searrow 0 \mbox{ and } C(p,q) \rightarrow + \infty \mbox{ as } q\searrow p.
\end{align*}
\end{rmk}

\begin{proof}
We only point out the main differences with respect to the previously presented arguments for $f=0$.  \\

\emph{Step 1: Modifications in the Carleman estimate and in the proof of Proposition \lemhomo}.
As $f\in L^q$ for some $q\in (n+1,\infty]$, Uraltseva's $C^{1,\alpha}$ regularity result remains valid \cite{U87}.
Furthermore, we remark that it is still possible to carry out Uraltseva's change of coordinates (c.f. Proposition \propchange and \cite{U85}), which transforms the variable coefficient problem into a new variable coefficient thin obstacle problem with an off-diagonal property as in (A1) on $B_1'$. Here, if $f$ is the original inhomogeneity, the inhomogeneity of the transformed equation is given by $\tilde{f}(y) = f(x)|\det(DT(x))|^{-1}\left. \right|_{x=T^{-1}(y)}$. Consequenty, $\tilde{f}$ inherits the $L^q$ (and $W^{1,p}$) regularity from $f$. For convenience of notation, in the sequel we suppress the tildas.\\
The preceding discussion now allows us to prove the Carleman estimate along the same lines as in Section \secCarl. The only necessary modification is to incorporate a further right hand side contribution:
\begin{equation}
\label{eq:vCarl_f}
\begin{split}
&\tau^{\frac{3}{2}} \left\|  e^{\tau \phi}|x|^{-1} (1+\ln(|x|)^2)^{-\frac{1}{2}}  w \right\|_{L^2(A_{\rho, r}^+)} +  \tau^{\frac{1}{2}}\left\| e^{\tau \phi} (1+\ln(|x|)^2)^{-\frac{1}{2}}  \nabla w \right\|_{L^2(A_{\rho,r}^+)}\\
&\leq C(n)c_0^{-1} \left(  
\tau^2 C(a^{ij})\left\| e^{\tau\phi}|x|^{\gamma-1} w \right\|_{L^2(A_{\rho,r}^+)}  + \left\|e^{\tau \phi} |x| f\right\|_{L^2(A_{\rho,r}^+)} \right),
\end{split}
\end{equation}
where $\gamma := 1-\frac{n+1}{p}$ and
\begin{align*}
C(a^{ij}) = \sup\limits_{A_{\rho,r}^+} \left| \frac{a^{ij}(x)-\delta^{ij}}{|x|^{\gamma}} \right| + \sup\limits_{\rho\leq \tilde{r}\leq r/2} \left\| |x|^{-\gamma} \nabla a^{ij} \right\|_{L^{n+1}(A_{\tilde{r},2\tilde{r}}^+)}.
\end{align*}
We estimate the additional term coming from the inhomogeneity for parameters $\tau \in [1, \frac{ \tau_0}{1+ c_0\pi/2}]$, where $\tau_0= \kappa_0 + \frac{n-1}{2}$ with $\kappa_0\geq 0 $:
\begin{equation}
\label{eq:bound_f}
\begin{split}
\left\| e^{\tau \phi}|x|f \right\|_{L^2(A_{\rho,r}^+)} & \lesssim \left\| f \right\|_{L^q(B_1^+)} \left\| e^{\tau \phi} |x| \right\|_{L^{\frac{2q}{q-2}}(A_{\rho,r}^+)} \\
& \lesssim \left\| f \right\|_{L^q(B_1^+)}\left\| |x|^{-\tau_0 +1} \right\|_{L^{\frac{2q}{q-2}}(A_{\rho,r}^+)}\\
&  \lesssim \left\| f \right\|_{L^q(B_1^+)}(r^{\delta}-\rho^\delta),\quad \delta= 2-\frac{n+1}{q}-\kappa_0.
\end{split}
\end{equation}
Hence, the contribution originating from the inhomogeneity $f$ in the Carleman inequality can be regarded as a ``subcritical'' contribution (in the sense that $\delta\geq0$) for the above range of $\tau$ if $\kappa_0\leq 2-\frac{n+1}{q}$. 
Analogous to the Carleman estimate, Corollary \corconsequenceCarl remains valid, if the contribution originating from $f$ is included on the right hand side:
\begin{equation}
\label{eq:consequence_Carl_f}
\begin{split}
&\tau^{\frac{3}{2}}(1+|\ln(r_2)|)^{-1}  e^{\tau \tilde{\phi}(\ln(r_2))} r_2^{-1}\left\| w \right\|_{L^2(A^+_{r_2,2r_2}(x_0))}\\
&+ \tau \max\{\ln(r_2/r_1)^{-1},\ln(r_3/r_2)^{-1}\} e^{\tau \tilde{\phi}(\ln(r_2))}r_2^{-1}  \left\| w \right\|_{L^2(A^+_{r_2,2r_2}(x_0))} \\
& \leq C( e^{\tau \tilde{\phi}(\ln(r_1))} r_1^{-1} \left\| w \right\|_{L^2(A^+_{r_1,2r_1}(x_0))} \\
& \quad + e^{\tau \tilde{\phi}(\ln(r_3))}r_3^{-1}\left\| w \right\|_{L^2(A^+_{r_3,2 r_{3}}(x_0))} + \left\| e^{\tau \tilde{\phi}(\ln(|x|))} |x| f \right\|_{L^2(A^+_{r_1,2 r_{3}}(x_0))}).
\end{split}
\end{equation}

\emph{Step 2: Modifications in Section \secreg.}
Estimate (\ref{eq:consequence_Carl_f}) combined with the boundedness of (\ref{eq:bound_f}), immediately allows us to infer the analogue of Proposition \propindep if $\kappa_x < 2-\frac{n+1}{q}$. Moreover, arguing by contradiction, we also obtain that if $\kappa_x\geq 2-\frac{n+1}{q}$, then not only
\begin{align*}
\limsup\limits_{r \rightarrow 0}\frac{\ln\left( \fint\limits_{A_{r,2r}^+} w^2 \right)^{1/2}}{\ln(r)} \geq 2 - \frac{n+1}{q},
\end{align*}
but also 
\begin{align*}
\liminf\limits_{r \rightarrow 0} \frac{\ln\left( \fint\limits_{A_{r,2r}^+} w^2 \right)^{1/2}}{\ln(r)} \geq 2- \frac{n+1}{q}.
\end{align*}
Similarly, the doubling property (c.f. Proposition \propdoubling), the blow-up result (c.f. Proposition \propblowup) and the lower semi-continuity statement (c.f. Proposition \propsemicont) remain valid at points of vanishing order less than $2-\frac{n+1}{q}$. Also the growth Lemma \lemgrowthimp remains valid for all $x_0 \in \Gamma_w$ with $\bar{\kappa}=2-\frac{n+1}{q}$:
\begin{align}
\label{eq:growth_f}
\sup\limits_{B_{r}^+(x_0)}|w| \leq C r^{\min\{\kappa_{x_0},2-\frac{n+1}{q}\}}|\ln(r)|^2,
\end{align}
where $C=C(n,p,\left\| \nabla a^{ij}\right\|_{L^p(B_1^+)}, \left\| f \right\|_{L^q(B_1^+)})$.\\
Apart from these results, we also obtain the existence of homogeneous blow-up limits at free boundary points with vanishing order less than $2-\frac{n+1}{q}$. In fact, this is again a consequence of the boundedness of the contributions originating from $f$ in the Carleman estimate. \\
Therefore, by Proposition \propblowup and Proposition \lemhomo, at each $x\in \Gamma_w$ with $\kappa_x<2-\frac{n+1}{q}$, there exists an $L^2$ normalized blow-up sequence centered at $x$, whose limit is a nontrivial homogeneous global solution (to the constant coefficient thin obstacle problem) with homogeneity equal to $\kappa_x$. Then by the classification of the homogeneous global solutions (Proposition \prophomotwo) on the one hand, we obtain that there is no free boundary point with $\kappa_x<2-\frac{n+1}{q}$ if $q\in (n+1,2(n+1))$. Hence, (\ref{eq:growth_f}) turns into:
\begin{align}
\label{eq:growth_f}
\sup\limits_{B_{r}^+(x_0)}|w| \leq C r^{2-\frac{n+1}{q}}|\ln(r)|^2,
\end{align}
for all $x_0\in \Gamma_w$.
On the other hand, $\kappa_x \geq\frac{3}{2}$ if $q\in [2(n+1),\infty]$. Combining this information with the uniform upper growth estimate (Lemma \lemgrowthimp) results in  
\begin{align}
\label{eq:modifL43}
\sup\limits_{B_r^+(x)}|w| \leq \begin{cases}
Cr^{2-\frac{n+1}{q}} (\ln r)^2 &\text{ if } q\in (n+1,2(n+1))\\
Cr^{\frac{3}{2}} (\ln r)^2 &\text{ if } q \in [2(n+1),\infty]
\end{cases},\quad r\in (0,1/2),
\end{align}
where the constant $C>0$ depends on $\left\| \nabla a^{ij} \right\|_{L^p},n,p,\left\| f\right\|_{L^{q}}$.\\

\emph{Step 3: Modifications in Sections \ref{sec:boundary} and \ref{sec:optimal}.} In order to deduce the regularity of the free boundary and to obtain the optimal regularity of $w$, we argue along the lines of Sections \ref{sec:boundary} and \ref{sec:optimal}. Here we interpret the inhomogeneity in the equation for the derivative as a divergence right hand side:
\begin{align*}
\p_i a^{ij} \p_i \p_e w & = \p_i F^i \mbox{ in } B_{1}\setminus \Lambda,\\
\p_e w & = 0 \mbox{ in } \Lambda,
\end{align*}
where for all $i\in\{1,...,n+1\}$
\begin{align}
\label{eq:inhom1}
F^i := - (\p_e a^{ij})\p_j w + e^i f \in L^p(B_1).
\end{align}
Thus, (potentially) after a rescaling which only depends on $\left\| \nabla a^{ij}\right\|_{L^p(B_1)}$ and $\left\| f \right\|_{L^p(B_1)}$, we may assume that the inhomogeneity is small. Therefore, all the results in Sections \ref{sec:boundary} and \ref{sec:optimal} are valid.
\end{proof}

\subsection{Non-flat obstacles}
\label{sec:nonfl}

In this section we present the main ideas of dealing with non-flat obstacles. Due to our almost scaling critical lower bound in Proposition \ref{prop:nondeg}, we are able to deal with the non-flat obstacle problem involving metrics $a^{ij}\in W^{1,p}(B_1)$ and obstacles $\varphi\in W^{2,p}(B_1)$ with $p\in (2(n+1),\infty]$. We however stress that the analogues of \cite{KRS14} are valid under the even weaker integrability assumption $p> n+1$.

\begin{prop}[Non-flat obstacles]
\label{prop:inhomo}
Let $a^{ij}: B_1^+ \rightarrow \R^{(n+1)\times (n+1)}_{sym}$ be a $W^{1,p}$ metric with $p\in (n+1,\infty]$ satisfying (A1), (A2), (A4). Suppose that $\varphi \in W^{2,p}(B_{1}')$. Let $w:B_1^+ \rightarrow \R$ be a solution of the thin obstacle problem
\begin{equation}
\label{eq:varcoef_f}
\begin{split}
\p_i a^{ij} \p_j w & = 0 \mbox{ in } B_{1}^+,\\
w  \geq \varphi, \ - a^{n+1,j} \p_j w \geq 0, \ (w-\varphi) ( a^{n+1,j} \p_j w) & = 0 \mbox{ on } B_{1}'.
\end{split}
\end{equation}
Then, the following statements hold:
\begin{itemize}
\item[(i)] The solution $w$ has the following growth estimate: 
There exists $C>0$ such that
\begin{align*}
|w(x)|\leq C\left\{
\begin{array}{ll}
\dist(x,\Gamma_w)^{2-\frac{n+1}{p}}(\ln (\dist(x,\Gamma_w)))^2 & \mbox{ if } n+1 < p< 2(n+1), \\
\dist(x, \Gamma_w)^{\frac{3}{2}}(\ln (\dist(x,\Gamma_w)))^2 & \mbox{ if } p \geq 2(n+1).
\end{array}
\right.
\end{align*}
The constant $C>0$ depends on $n,p, \left\| \varphi \right\|_{W^{2,p}(B_1^+)}, a^{ij}, \left\| w \right\|_{L^2(B_1^+)}$.
\item[(ii)] If additionally $a^{ij}\in W^{1,p}(B_1)$ and $\varphi\in W^{2, p}(B_1')$ with $p\in (2(n+1),\infty]$, then the solution $w$ has the optimal Hölder regularity: 
$$ w \in C^{1,1/2}(B_{1/2}^+).$$
\item[(iii)] Under the conditions in (ii) and assuming that $0\in \Gamma_{3/2}(w)$,  there exist a radius $\rho>0$, a parameter $\alpha \in (0,1]$ and a $C^{1,\alpha}$ function $g$ such that (potentially after a rotation) 
$$\Gamma_{3/2}(w) \cap B_{\rho}' = \{x = (x'',x_n,0)| \ x_{n}= g(x'') \mbox{ for } x''\in B_{\rho}'' \}.$$
\end{itemize}
\end{prop}

\begin{rmk}
We remark that as in Proposition \ref{prop:inhomo} the estimate in (i) immediately entails a corresponding regularity result (c.f. Remark \ref{rmk:reg_est}).
\end{rmk}

We prove this result by reducing to the setting of flat obstacles.

\begin{proof}
\emph{Step 1: Recovery of flat boundary conditions.}
We first carry out Uraltseva's change of coordinates \cite{U85}.
Then in order to recover flat boundary conditions, we introduce a function $v:B_{1}^+ \rightarrow \R$ with $v=w - \varphi$. This then leads to an inhomogeneous thin obstacle problem with a flat obstacle:
\begin{equation}
\label{eq:nfl}
\begin{split}
\p_i a^{ij} \p_j v & = f \mbox{ in } B_{1}^+,\\
 v \geq 0, \ -\p_{n+1} v \geq 0 , \ v \p_{n+1} v & = 0 \mbox{ on } B_{1}',
\end{split}
\end{equation}
with $f= - \p_i a^{ij} \p_j \varphi = - (\p_i a^{ij}) \p_j \varphi - a^{ij}\p_{ij}\varphi$.
In recovering the results from the flat obstacle setting, the main difficulties arise from the error contributions which result from the inhomogeneity. In order to derive (i), we hence argue that the assumptions of Section \ref{sec:inhomo} are satisfied. In order to obtain (ii) and (iii) we interpret the error as a divergence form right hand side and argue as in Sections \ref{sec:boundary}-\ref{sec:optimal}.\\

\emph{Step 2: Bounding the inhomogeneity.}
Due to our assumptions on $f$ and $a^{ij}$ and by Sobolev/ Morrey embedding, the metric and the inhomogeneity have the right integrability. Indeed, in the setting of (i) we have
\begin{itemize}
\item $\p_i a^{ij} \in L^p(B_1^+)$, $a^{ij}\in L^{\infty}(B_1^+)$,
\item $\p_{ij} \varphi \in L^p(B_1')$, $\p_j \varphi \in L^{\infty}(B_1')$.
\end{itemize}
Hence, $f\in L^{p}(B_1^+)$ which allows us to invoke the results from Section \ref{sec:inhomo}. This then yields the growth estimate stated in (i).\\

\emph{Step 3: Argument for (ii) and (iii).} We argue as in step 3 in Section \ref{sec:inhomo}. In order to obtain the results (ii) and (iii), we consider the equation for tangential derivatives of $v$ (after carrying out an odd reflection as described in (\ref{eq:extend_reflect})):
\begin{equation}
\label{eq:deriv_v}
\begin{split}
\p_i a^{ij} \p_{j e} v & = -\p_i ((\p_e a^{ij}) \p_j v) + \p_e f \mbox{ in } B_1 \setminus \Lambda, \\
\p_e v & = 0 \mbox{ in } \Lambda.
\end{split}
\end{equation}
We interpret the right hand side of (\ref{eq:deriv_v}) as a divergence form contribution with
\begin{equation}
\label{eq:div_form}
F^{i}:= -(\p_e a^{ij}) \p_j v +e^{i} f, \ i\in\{1,...,n+1\}. 
\end{equation}
The regularity of the metric and obstacle, $a^{ij} \in W^{1,p}(B_1)$, $\varphi  \in W^{2,p}(B'_1)$ with $p\in (2(n+1),\infty]$, entail that
\begin{align*}
F^i\in L^p(B_1).
\end{align*}
Moreover, by an appropriate scaling argument (which only depends on $\left\| \nabla a^{ij} \right\|_{L^p(B_1)}$ and $\left\| \varphi \right\|_{W^{2,p}(B_1')})$, we may assume that $\left\| F^i \right\|_{L^p(B_1)}$ is small. Hence, all the results in Sections \ref{sec:boundary}-\ref{sec:optimal} remain valid. This then yields the desired results of (ii) and (iii). 
\end{proof}

\subsection{Non-flat boundaries}
\label{sec:nfboundaries}
In this section we very briefly comment on the situation with non-flat boundaries. 
In this context we have the following result:

\begin{prop}[Non-flat boundaries]
Let $\Omega \subset \R^{n+1}$ be a bounded open subset, whose boundary contains the $W^{2,p}$ hypersurface $\mathcal{M}$.
Let $a^{ij}: \Omega \rightarrow \R^{(n+1)\times (n+1)}_{sym}$ be a uniformly elliptic, symmetric tensor field of class $W^{1,p}(\Omega)$, $p\in( n+1,\infty]$. Assume that $w: \Omega \cup \mathcal{M} \rightarrow \R $ is a solution of the thin obstacle problem with zero obstacle on the hypersurface $\mathcal{M}$:
\begin{align*}
\p_i a^{ij} \p_j w &= 0 \mbox{ in } \Omega,\\
w \geq 0, \ \nu_i a^{ij} \p_j w \geq 0, \ w(\nu_i a^{ij} \p_j w)&=0 \mbox{ on } \mathcal{M},
\end{align*}
where $\nu: \mathcal{M} \rightarrow \R^{n+1}$ denotes the outer unit normal field on $ \mathcal{M}$. Then, the following statements hold: 
\begin{itemize}
\item[(i)] (Optimal regularity) If $p\in (n+1,2(n+1)]$, then $w\in C^{1,1-\frac{n+1}{p}}_{loc}(\Omega\cup \mathcal{M})$; if $p\in (2(n+1),\infty]$, then $w \in C^{1,1/2}_{loc}(\Omega \cup \mathcal{M})$.
\item[(ii)] Assuming that $0\in \Gamma_{3/2}(w)$, there exist a radius $\rho>0$ and a parameter $\alpha \in (0,1-\frac{n+1}{p}]$ such that $\Gamma_{3/2}(w)\cap B_\rho$ is an $(n-1)$-dimensional $C^{1,\alpha}$ submanifold.
\end{itemize}
\end{prop}

\begin{proof}
The statement can be immediately reduced to the setting of the usual thin obstacle problem by carrying out a change of coordinates (c.f. Proposition \propchange and \cite{U85}) that flattens the free boundary. As the free boundary is $W^{2,p}$ regular, this change of coordinates followed by an application of Uraltseva's change of coordinates from Proposition \propchange, then implies the integrability and differentiability properties (A1), (A2), (A3), (A4). This allows to carry out the analysis from \cite{KRS14} and Sections \ref{sec:boundary}-\ref{sec:optimal} of the present article.
\end{proof}

\subsection{Interior thin obstacles}
\label{sec:int_obst}

Finally, we comment on the necessary modification steps in obtaining analogous results for interior thin obstacles. In this direction we have:

\begin{prop}[Interior thin obstacles]
\label{prop:intobst}
Let $a^{ij}: B_1 \rightarrow \R^{(n+1)\times (n+1)}_{sym}$ be a uniformly elliptic, symmetric tensor field of class $W^{1,p}$ with $p\in (n+1,\infty]$, satisfying (A1), (A2), (A3) and (A4). Assume that $w: B_1 \rightarrow \R $ is a solution of the interior thin obstacle problem with zero obstacle on $B'_1$:
\begin{align*}
\p_i a^{ij} \p_j w &= 0 \mbox{ in } \inte{B_1^+}\cup \inte{B_1^-},\\
 w \geq 0, \ [\p_{n+1} w] \geq 0, \ w[ \p_{n+1} w] & = 0 \mbox{ on } B'_1. 
\end{align*}
Here $[ \p_{n+1} w]:=( \p_{n+1} w)^+  -(\p_{n+1} w)^-$ and $(\p_{n+1} w)^{\pm}$ denotes the one-sided trace of the fluxes on $B'_1$. More precisely, for $H^n$ a.e. $x\in B'_1$ we have
\begin{align*}
( \p_{n+1} w)^{\pm}(x):= \lim\limits_ { y \rightarrow x, \pm y_{n+1}>0}  \p_{n+1} w (y).
\end{align*}
Then, the following statements hold: 
\begin{itemize}
\item[(i)] (Optimal regularity) If $p\in (n+1,2(n+1)]$, then $w\in C^{1,1-\frac{n+1}{p}}_{loc}(B_1^+\cup B_1^-)$; if $p\in (2(n+1),\infty]$, then $w \in C^{1,1/2}_{loc}(B_1^+\cup B_1^- )$.
\item[(ii)] Assuming that $0\in \Gamma_{3/2}(w)$, there exist a radius $\rho>0$, a parameter $\alpha \in (0,1-\frac{n+1}{p}]$ and a $C^{1,\alpha}$ function $g$ such that (potentially after a rotation) 
\begin{align*}
\Gamma_{3/2}(w) \cap B_{\rho}' = \{x = (x'',x_n,0)| \ x_{n}= g(x'') \mbox{ for } x''\in B_{\rho}'' \}.
\end{align*}
\item[(iii)] There exist functions $a, \tilde{b}:\Gamma_{3/2}(w)\cap B_{\rho/2}\rightarrow \R$ with $a \in C^{0,\alpha}(\Gamma_{3/2}(w)\cap B_{\rho/2})$ and $\tilde{b}\in C^{0,\frac{1}{2}+\alpha}(\Gamma_{3/2}(w)\cap B_{\rho/2})$ for some $\alpha>0$ (where $\alpha$ can be chosen as in (ii)), such that: for each $x_0 \in \Gamma_{3/2}(w)\cap B_{\rho/4}$, 
\begin{itemize}
\item For each $e\in S^{n-1}\times\{0\}$ and $x\in B_{\rho/4}(x_0)$,
\begin{align*}
&\left| \p_{e} w(x)-\frac{3}{2}\frac{a(x_0)(e\cdot \nu_{x_0})}{(\nu_{x_0}\cdot A(x_0)\nu_{x_0})^{1/2}}w_{1/2}\left(\frac{(x- x_0) \cdot\nu_{ x_0}}{(\nu_{ x_0} \cdot A( x_0) \nu_{ x_0})^{1/2}},\frac{x_{n+1}}{(a^{n+1,n+1}( x_0))^{1/2}}\right)\right|\notag\\
&\quad \leq C(n,p)\max\{\epsilon_0,c_\ast\}  \dist(x,\Gamma_w)^{\frac{1}{2}+\widetilde{\gamma}}.
\end{align*}
\item For $x\in B_{\rho/4}(x_0)$,
\begin{align*}
&\left| \p_{n+1} w(x)- \tilde{b}(x_0)\right.\\
&\quad \left. -\frac{3}{2} \frac{a(x_0)}{(a^{n+1,n+1}(x_0))^{1/2}}\bar w_{1/2}\left(\frac{(x- x_0) \cdot\nu_{ x_0}}{(\nu_{ x_0} \cdot A( x_0) \nu_{ x_0})^{1/2}},\frac{x_{n+1}}{(a^{n+1,n+1}(x_0))^{1/2}}\right)\right|\\
&\quad \leq C(n,p)\max\{\epsilon_0,c_\ast\}  \dist(x,\Gamma_w)^{\frac{1}{2}+\widetilde{\gamma}}.
\end{align*}
\item For $x\in B_{\rho/4}(x_0)$ and $x-x_0\in \spa\{\nu_{x_0},e_{n+1}\}$,
\begin{align*}
&\quad \left|w(x)-\tilde{b}(x_0)x_{n+1}-a(x_0)w_{3/2}\left(\frac{(x- x_0) \cdot\nu_{ x_0}}{(\nu_{ x_0} \cdot A( x_0) \nu_{ x_0})^{1/2}},\frac{x_{n+1}}{(a^{n+1,n+1}( x_0))^{1/2}}\right)\right|\notag\\
&\quad \leq C(n,p)\max\{\epsilon_0,c_\ast\} \dist(x,\Gamma_w)^{\frac{3}{2}+\widetilde{\gamma}}.
\end{align*}
\end{itemize}
Here $w_{1/2}(x)= \Ree(x_n+ix_{n+1})^{1/2}$, $\bar{w}_{1/2}(x)=- \Imm(x_n + i x_{n+1})^{1/2}$, $w_{3/2}(x)=\Ree(x_n+ix_{n+1})^{3/2}$.
Moreover, $\tilde{\gamma}=\min\{\alpha, \frac{1}{2}-\frac{n+1}{p}\}$ and $A(x_0)=(a^{ij}(x_0))$.
\end{itemize}
\end{prop}

\begin{rmk}
As the proof of Proposition \ref{prop:intobst} shows, the function $\tilde{b}$ can be identified explicitly as $\tilde{b}(x_0)= (\p_{n+1}w)^+(x_0)$, where $x_0\in \Gamma_{3/2}(w)\cap B_{\rho/2}$.
\end{rmk}

\begin{rmk}
Similarly as in Sections~\ref{sec:inhomo}-\ref{sec:nfboundaries}, our results generalize to non-flat $W^{2,p}$ obstacles $\varphi$, non-flat $W^{2,p}$ manifolds $\mathcal{M}$ and $L^p$ inhomoneneities $f$, where $p>2(n+1)$. 
\end{rmk}

\begin{proof}
As in the previous sections, we only comment on the main changes. 
In the first step we derive an analogue of the Carleman inequality \eqvCarl. Since an essential ingredient in the proof of the Carleman estimate was the vanishing of the boundary contributions due to the complementary boundary conditions, slight changes are necessary at this point. As $w$ satisfies an elliptic equation in both the upper and the lower half planes, we carry out the Carleman conjugation procedure in both half-balls separately. On each of these we obtain boundary contributions which do not necessarily vanish. However, adding the contributions originating from the lower and upper half-balls allows us to exploit the complementary boundary conditions for the interior thin obstacle problem. Hence, we obtain the vanishing of all involved boundary contributions. The Carleman estimate then reads:
\begin{equation}
\label{eq:consequence_Carl_i}
\begin{split}
&\tau^{\frac{3}{2}}(1+|\ln(r_2)|)^{-1}  e^{\tau \tilde{\phi}(\ln(r_2))} r_2^{-1}\left\| w \right\|_{L^2(A_{r_2,2r_2}(x_0))}\\
&+ \tau \max\{\ln(r_2/r_1)^{-1},\ln(r_3/r_2)^{-1}\} e^{\tau \tilde{\phi}(\ln(r_2))}r_2^{-1}  \left\| w \right\|_{L^2(A_{r_2,2r_2}(x_0))} \\
& \leq C( e^{\tau \tilde{\phi}(\ln(r_1))} r_1^{-1} \left\| w \right\|_{L^2(A_{r_1,2r_1}(x_0))} \\
& \quad + e^{\tau \tilde{\phi}(\ln(r_3))}r_3^{-1}\left\| w \right\|_{L^2(A_{r_3,2 r_{3}}(x_0))} + \left\| e^{\tau \tilde{\phi}(\ln(|x|))} |x| f \right\|_{L^2(A_{r_1,2 r_{3}}(x_0))}),
\end{split}
\end{equation}
where $A_{r_1,r_3}(x_0)$ denotes the full annulus around the point $x_0 \in B_{1}'$. \\
In deriving the analogues of the statements of Section \secreg, we modify the definition of the blow-up sequences slightly: Indeed, we note that if $w(x)$ is a solution of the interior thin obstacle problem, then $w\in C^{1,\alpha}(B^+_1\cup B^-_1)$ for some $\alpha>0$ (c.f. page 207 in \cite{U87}). Then for 
$$\tilde{b}(x_0):=\lim\limits_{y \rightarrow x_0, y_{n+1}> 0}\p_{n+1}w(y),$$
the function
\begin{align}
\label{eq:normal1}
v_{x_0}(x):= w(x) - \tilde{b}(x_0) x_{n+1},
\end{align}
solves
\begin{align*}
\p_i a^{ij} \p_j v_{x_0} &= f \mbox{ in } B_{1/2},\\
v_{x_0 } & \geq 0, \ [\p_{n+1}v_{x_0}] \geq 0, \ v_{x_0}[\p_{n+1} v_{x_0}]=0 \mbox{ on } B_{1/2}'.
\end{align*}
Here $f(x):= -(\p_i a^{ij})(x) \tilde{b}(x_0) \in L^p(B_{1})$ by the regularity and integrability assumptions on $a^{ij}$ and the $C^{1,\alpha}(B_{1/2}^+)$ regularity of $w$ (c.f. \cite{U87}).
Therefore, for any free boundary point $x_0 \in \Gamma_w \cap B_{1/2}$ the modified functions $v_{x_0}(x)$ now satisfy
\begin{align}
\label{eq:normalize}
v_{x_0}(x_0) = 0 = |\nabla v_{x_0}(x_0)|.
\end{align}
Here, by approaching $x=x_0$ from the interior of $\Omega_w$, we observed that $[\p_{n+1}w](x_0)=0$, which yields the second equality in (\ref{eq:normalize}).\\
Analogously, we consider the following blow-ups around the point $x_0$, which are based on $v_{x_0}$ instead of on $w$:
\begin{align*}
v_{x_0,r}(x):= \frac{v_{x_0}(x_0 + r x)}{ r^{- \frac{n+1}{2}} \| v\|_{L^2(B_r(x_0))}}.
\end{align*}
In the case of the interior thin obstacle problem this replaces the blow-up functions defined in Section \ref{sec:not}. Arguing on the level of $v_{x_0}$ and $v_{x_0,r}$ and using the observations from Section \ref{sec:inhomo}, all the results from Section \secreg follow. In particular (\ref{eq:normalize}) ensures that the discussion of the vanishing order in Proposition \prophomotwo remains valid.\\
Further relying on the functions $v_{x_0}$, we derive the analogues of the results from Sections \ref{sec:boundary}-\ref{sec:optimal} for $v_{x_0}$ by arguing along the same lines as for the boundary thin obstacle problem (in this context, we remark that since solutions are, by definition of the interior obstacle problem, already defined in the whole ball, it is not necessary to carry out reflection arguments). Hence, all the properties from these sections are valid on the level of $v_{x_0}$. These can then be directly transferred to the original function $w$. Due to the presence of the normalizing factor $\tilde{b}(x_0) x_{n+1}$ this changes the asymptotic expansion of $\p_{n+1}w$ with respect to the one derived in Remark \ref{rmk:asym} by the constant term $\tilde{b}(x_0)$. 
Using the $C^{0,\alpha}$ regularity of $a$ and the triangle inequality, we infer from the asymptotics of $\p_{n+1}w$ that $\tilde{b}\in C^{0,\frac{1}{2}+\alpha}(\Gamma_{3/2}(w)\cap B_{\rho/2})$. Indeed, for $x_0,x_1\in \Gamma_{3/2}(w)\cap B_{\rho/2}$, we fix a point 
$x\in B_\rho$ with $\dist(x,\Gamma_w)\sim |x_0-x_1|$. By the triangle inequality we have
\begin{align*}
|\tilde{b}(x_0)-\tilde{b}(x_1)|\leq |G_{x_0}(x)-G_{x_1}(x)|+C(n,p)\max\{\epsilon_0,c_\ast\}\dist(x,\Gamma_w)^{\frac{1}{2}+\widetilde{\gamma}},
\end{align*}
where for $\xi\in \Gamma_{3/2}(w)\cap B_{\rho/2}$,
$$G_\xi(x):=\frac{3}{2} \frac{a(\xi)}{(a^{n+1,n+1}(\xi))^{1/2}}\bar w_{1/2}\left(\frac{(x- \xi) \cdot\nu_{ \xi}}{(\nu_{ \xi} \cdot A( \xi) \nu_{ \xi})^{1/2}},\frac{x_{n+1}}{(a^{n+1,n+1}(\xi))^{1/2}}\right).$$
By using the $C^{0,\alpha}$ regularity of $a$, the ellipticity and regularity of $A=(a^{ij})$, the $C^{1,\alpha}$ regularity of $\Gamma_{3/2}(w)$, as well as the relation $\dist(x,\Gamma_w)\sim |x_0-x_1|$, we have
\begin{align*}
|G_{x_0}(x)-G_{x_1}(x)|\lesssim |x_0-x_1|^{\frac{1}{2}+\alpha}.
\end{align*}
This proves the $C^{0,\frac{1}{2}+\alpha}$ regularity of $\tilde{b}$. The remaining observations follow along similar lines as in Sections \ref{sec:boundary}-\ref{sec:optimal}.
\end{proof}

\bibliography{citations}

\begin{thebibliography}{GSVG14}

\bibitem[AC06]{AC06}
Ioannis Athanasopoulos and Luis~A. Caffarelli.
\newblock Optimal regularity of lower-dimensional obstacle problems.
\newblock {\em Journal of Mathematical Sciences}, 132(3):274--284, 2006.

\bibitem[ACS08]{ACS08}
Ioannis Athanasopoulos, Luis~A. Caffarelli, and Sandro Salsa.
\newblock The structure of the free boundary for lower dimensional obstacle
  problems.
\newblock {\em American Journal of Mathematics}, 130(2):485--498, 2008.

\bibitem[And13]{An}
John Andersson.
\newblock Optimal regularity and free boundary regularity for the {S}ignorini
  problem.
\newblock {\em St. Petersburg Mathematical Journal}, 24(3):371--386, 2013.

\bibitem[Ben80]{Bene80}
Michael Benedicks.
\newblock Positive harmonic functions vanishing on the boundary of certain
  domains in {$R^n$}.
\newblock {\em Arkiv f{\"o}r Matematik}, 18(1):53--72, 1980.

\bibitem[CKL05]{CKL05}
Luca Capogna, Carlos~E. Kenig, and Loredana Lanzani.
\newblock {\em Harmonic measure: geometric and analytic points of view},
  volume~35.
\newblock American Mathematical Soc., 2005.

\bibitem[CSS08]{CSS}
Luis~A. Caffarelli, Sandro Salsa, and Luis Silvestre.
\newblock Regularity estimates for the solution and the free boundary of the
  obstacle problem for the fractional {L}aplacian.
\newblock {\em Inventiones mathematicae}, 171(2):425--461, 2008.

\bibitem[DSS14]{DSS14}
Daniela De~Silva and Ovidiu Savin.
\newblock Boundary {H}arnack estimates in slit domains and applications to thin
  free boundary problems.
\newblock {\em arXiv preprint arXiv:1406.6039}, 2014.

\bibitem[EG15]{E15}
Lawrence~Craig Evans and Ronald~F Gariepy.
\newblock {\em Measure theory and fine properties of functions}.
\newblock CRC press, 2015.

\bibitem[FH76]{FH76}
Shmuel Friedland and Walter~K. Hayman.
\newblock Eigenvalue inequalities for the {D}irichlet problem on spheres and
  the growth of subharmonic functions.
\newblock {\em Commentarii Mathematici Helvetici}, 51(1):133--161, 1976.

\bibitem[Fre77]{Fr77}
Jens Frehse.
\newblock On {S}ignorini's problem and variational problems with thin
  obstacles.
\newblock {\em Annali della Scuola Normale Superiore di Pisa-Classe di
  Scienze}, 4(2):343--362, 1977.

\bibitem[GP09]{GP09}
Nicola Garofalo and Arshak Petrosyan.
\newblock Some new monotonicity formulas and the singular set in the lower
  dimensional obstacle problem.
\newblock {\em Inventiones mathematicae}, 177(2):415--461, 2009.

\bibitem[GPG15]{GPSVG15}
Nicola Garofalo, Arshak Petrosyan, and Mariana Smit~Vega Garcia.
\newblock An epiperimetric inequality approach to the regularity of the free
  boundary in the {S}ignorini problem with variable coefficients.
\newblock {\em arXiv preprint arXiv:1501.06498}, 2015.

\bibitem[GSVG14]{GSVG14}
Nicola Garofalo and Mariana Smit Vega~Garcia.
\newblock New monotonicity formulas and the optimal regularity in the
  {S}ignorini problem with variable coefficients.
\newblock {\em Advances in Mathematics}, 262:682--750, 2014.

\bibitem[GT01]{GT}
David Gilbarg and Neil~S. Trudinger.
\newblock {\em Elliptic partial differential equations of second order}, volume
  224.
\newblock Springer, 2001.

\bibitem[Gui09]{Gu}
Nestor Guillen.
\newblock Optimal regularity for the {S}ignorini problem.
\newblock {\em Calculus of Variations and Partial Differential Equations},
  36(4):533--546, 2009.

\bibitem[Kin81]{Ki81}
David Kinderlehrer.
\newblock The smoothness of the solution of the boundary obstacle problem.
\newblock {\em J. Math. Pures Appl.}, 60(2):193--212, 1981.

\bibitem[KPS14]{KPS}
Herbert Koch, Arshak Petrosyan, and Wenhui Shi.
\newblock Higher regularity of the free boundary in the elliptic {S}ignorini
  problem.
\newblock {\em Arxiv preprint arXiv:1406.5011}, 2014.

\bibitem[KRS15]{KRS14}
Herbert Koch, Angkana R{\"u}land, and Wenhui Shi.
\newblock The {V}ariable {C}oefficient {T}hin {O}bstacle {P}roblem: {C}arleman
  {I}nequalities.
\newblock {\em arXiv preprint arXiv:1501.04496}, 2015.

\bibitem[Lew72]{Le72}
Hans Lewy.
\newblock On the coincidence set in variational inequalities.
\newblock {\em Journal of Differential Geometry}, 6(4):497--501, 1972.

\bibitem[LMS12]{LMS12}
Antoine Lemenant, Emmanouil Milakis, and Laura~V. Spinolo.
\newblock On the extension property of {R}eifenberg-flat domains.
\newblock {\em arXiv preprint arXiv:1209.3602}, 2012.

\bibitem[PSU12]{PSU}
Arshak Petrosyan, Henrik Shahgholian, and Nina~N. Uraltseva.
\newblock {\em Regularity of {F}ree {B}oundaries in {O}bstacle-{T}ype
  {P}roblems}.
\newblock AMS, 2012.

\bibitem[Ric78]{Ri78}
David Richardson.
\newblock {\em Variational problems with thin obstacles}.
\newblock PhD thesis, University of British Columbia, 1978.

\bibitem[Sil07]{Si}
Luis Silvestre.
\newblock Regularity of the obstacle problem for a fractional power of the
  {L}aplace operator.
\newblock {\em Communications on Pure and Applied Mathematics}, 60(1):67--112,
  2007.

\bibitem[Sim96]{Si96}
Leon Simon.
\newblock {\em Theorems on regularity and singularity of energy minimizing
  maps}.
\newblock Springer Science \& Business Media, 1996.

\bibitem[Ste70]{St}
Elias~M. Stein.
\newblock {\em Singular integrals and differentiability properties of
  functions}, volume~30.
\newblock Princeton university press, Princeton, 1970.

\bibitem[SW]{SW10}
Leon Simon and Neshan Wickramasekera.
\newblock A frequency function and singular set bounds for branched minimal
  immersions.
\newblock {\em ArXiv preprint, 1012.5029v1}.

\bibitem[Ura85]{U85}
Nina~Nikolaevna Ural'tseva.
\newblock H{\"o}lder continuity of gradients of solutions of parabolic
  equations with boundary conditions of {S}ignorini type.
\newblock In {\em Dokl. Akad. Nauk SSSR}, volume 280, pages 563--565, 1985.

\bibitem[Ura87]{U87}
Nina~Nikolaevna Ural'tseva.
\newblock Regularity of solutions of variational inequalities.
\newblock {\em Russian Mathematical Surveys}, 42(6):191--219, 1987.

\end{thebibliography}
\bibliographystyle{alpha}
\end{document}